\documentclass[10pt, reqno]{amsart}

\usepackage[latin1]{inputenc}
\usepackage{amsmath,amsthm,amssymb,color,mathtools,esint,comment}
\definecolor{darkgreen}{rgb}{0.0, 0.6, 0.13}
\usepackage{graphicx}
\usepackage{caption}
\usepackage[shortlabels]{enumitem}

\usepackage{hyperref}

\usepackage{float}

\usepackage{mathrsfs}

\usepackage{import}  

\usepackage{geometry}
\usepackage{my_tikz}

\usepackage{bbm}

\usepackage{longtable}
\usepackage{autobreak}

\usepackage{bm}

\usepackage[normalem]{ulem}

\newcommand{\newterm}[1]{{\textbf{#1}}}

\newcommand{\pb}{{\boldsymbol{p}}}

\newcommand{\Dirac}{\boldsymbol{\Delta}}
\newcommand{\dirac}{\boldsymbol{\delta}}

\newcommand{\Eb}{\mathbb E}

\newcommand{\Mb}{\mathbb M}
\newcommand{\Nb}{\mathbb N}

\newcommand{\Pb}{\mathbb P}

\newcommand{\Rb}{\mathbb R}
\newcommand{\Sb}{\mathbb S}
\newcommand{\Tb}{\mathbb T}

\newcommand{\Zb}{\mathbb Z}
\newcommand{\Ac}{\mathcal A}
\newcommand{\Bc}{\mathcal B}
\newcommand{\Cc}{\mathcal C}
\newcommand{\Dc}{\mathcal D}
\newcommand{\Ec}{\mathcal E}
\newcommand{\Fc}{\mathcal F}

\newcommand{\Hc}{\mathcal H}
\newcommand{\Ic}{\mathcal I}
\newcommand {\Jc}{\mathcal J}

\newcommand{\Lc}{\mathcal L}
\renewcommand{\Mc}{\mathcal M}
\newcommand{\Nc}{\mathcal N}

\newcommand{\Pc}{\mathcal P}
\newcommand{\Qc}{\mathcal Q}
\newcommand{\Rc}{\mathcal R}
\newcommand{\Sc}{\mathcal S}

\newcommand{\Xc}{\mathcal X}
\newcommand{\Yc}{\mathcal Y}
\newcommand{\Zc}{\mathcal Z}

\newcommand{\mf}{\mathfrak m}
\newcommand{\pf}{\mathfrak p}

\newcommand{\rf}{\mathfrak r}

\newcommand{\qf}{\mathfrak q}

\newcommand{\Lf}{\mathfrak L}

\newcommand{\Mf}{\mathfrak M}

\newcommand{\nf}{\mathfrak n}

\newcommand{\vz}{\boldsymbol{z}}

\newcommand{\vt}{\boldsymbol{t}}

\theoremstyle{definition}
\newtheorem{mainthm}{Theorem}
\newtheorem{theorem}{Theorem}[section]

\newtheorem{lemma}[theorem]{Lemma}
\newtheorem{proposition}[theorem]{Proposition}

\newtheorem{definition}[theorem]{Definition}
\theoremstyle{remark}
\newtheorem{remark}[theorem]{Remark}

\numberwithin{equation}{section}

\begin{document}

\title{Hilbert's Sixth Problem: derivation of fluid equations via Boltzmann's kinetic theory}

\author{Yu Deng}
\address{\textsc{Department of Mathematics, University of Chicago, Chicago, IL, USA}}
\email{\texttt{yudeng@uchicago.edu}}
\author{Zaher Hani}
\address{\textsc{Department of Mathematics, University of Michigan, Ann Arbor, MI, USA}}
\email{\texttt{zhani@umich.edu}}
\author{Xiao Ma}
\address{\textsc{Department of Mathematics, University of Michigan, Ann Arbor, MI, USA}}
\email{\texttt{mxiao@umich.edu}}
\date{}

\begin{abstract} In this paper, we rigorously derive the fundamental PDEs of fluid mechanics, such as the compressible Euler and incompressible Navier-Stokes-Fourier equations, starting from the hard sphere particle systems undergoing elastic collisions. This resolves Hilbert's sixth problem, as it pertains to the program of deriving the fluid equations from Newton's laws by way of Boltzmann's kinetic theory. The proof relies on the derivation of Boltzmann's equation on 2D and 3D tori, which is an extension of our previous work \cite{DHM24}. 
\end{abstract}

\maketitle
\tableofcontents

\section{Introduction}\label{sec.intro}

\subsection{Hilbert's sixth problem}
In his address to the International Congress of Mathematics in 1900, David Hilbert proposed a list of problems as challenges for the mathematics of the new century. Of those problems, the sixth problem asked for an \emph{axiomatic derivation of the laws of physics}. In his description of the problem, Hilbert says: 
\begin{center}
\tt{"The investigations on the foundations of geometry suggest the problem: To treat in the same manner, by means of axioms, those physical sciences in which already today mathematics plays an important part; in the first rank are the theory of probabilities and mechanics.}"
\end{center}
Broadly interpreted, this problem can encompass all of modern mathematical physics. However, in his followup article \cite{H01}, which describes and elaborates on the list of problems, Hilbert specified two concrete questions under this sixth problem. The first concerns the axiomatic foundation of probability, which was settled in the first half of the twentieth century. The second question was described as follows:
\begin{center}
\tt{"...Boltzmann's work on the principles of mechanics suggests the problem of developing mathematically the limiting processes, there merely indicated, which lead from the atomistic view to the laws of motion of continua."}
\end{center}

In this question, Hilbert suggests a program, referred to as \emph{Hilbert's program} below, which aims at giving a rigorous derivation of the laws of fluid motion, starting from Newton's laws on the atomistic level, using Boltzmann's kinetic theory as an intermediate step. More precisely, this refers to giving a rigorous justification of the following diagram:
\begin{figure}[h!]
    \centering
    \includegraphics[width=14cm]{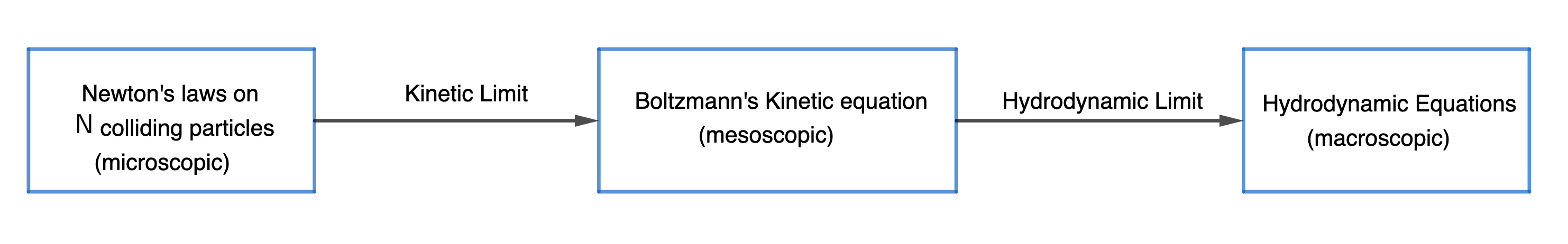} 
    \caption{From Newton to Boltzmann to fluid equations}
    \label{fig:microtomacro}  \end{figure}

In this paper we will complete Hilbert's program and justify the limiting process in Figure \ref{fig:microtomacro}. We start by describing the scope of this program and its two steps, as well as the historical accounts.

 In the first step of this program, one gives a rigorous derivation of Boltzmann's kinetic theory starting from Newton's laws, taken as axioms, on a \emph{microscopic} system formed of $N$ particles of diameter $\varepsilon$ undergoing elastic collisions. This is done by taking the \emph{kinetic limit} in which $N\to \infty, \,\varepsilon \to 0$, and we aim to show that the one-particle density of the particle system is well-approximated by the solution $n(t, x,v)$ to the Boltzmann equation:
\begin{equation}\label{eq.Boltzintro}
 (\partial_t +v\cdot \nabla_x) n(t, x, v) =\alpha \mathcal Q(n, n).
 \end{equation}
 Here $\Qc(n, n)$ is the hard-sphere collision kernel defined below in \eqref{eq.boltzmann}, and $\alpha:=N \varepsilon^{d-1}$ stands for the collision rate of the particle system, which is kept constant in this kinetic limit. The necessity of this scaling relation between $N$ and $\varepsilon$, which corresponds to the setting of dilute gases, was discovered by Grad \cite{Grad58}, and is referred to as the \emph{Boltzmann-Grad limit}. In the second step, called the \emph{hydrodynamic limit}, one derives the equations of fluid mechanics (like compressible Euler, incompressible Euler, incompressible Navier-Stokes etc.) as appropriate limits of Boltzmann's kinetic equation when the collision rate $\alpha$ is taken to infinity (i.e. the Knudson number, or the length of mean free path, is taken to zero).

To complete Hilbert's original program, one needs to take a proper combination of these two limits, to pass from the atomistic view of matter to the laws of motion of continua, as Hilbert put it in his own words. Crucially, here we note that establishing the link between these two limits requires deriving the Boltzmann equation as stated above on time intervals of length $O(1)$, which corresponds (by rescaling $t$ and $x$) to solutions of the $\alpha=1$ Boltzmann equation that exist on time intervals of length $\gtrsim \alpha$. Since $\alpha\to\infty$ in the hydrodynamic limit, it is thus necessary to obtain \emph{long time derivation} of the Boltzmann equation.
 
 \medskip
 
In the two limits, the first (kinetic) limit turned out to be more challenging. In fact, it even took some time to properly clarify it as a concrete mathematical question. This was eventually done by Grad \cite{Grad58} who specified the precise Boltzmann-Grad scaling law $N\varepsilon^{d-1}=O(1)$ mentioned above, which is assumed between the number of particles $N$ and their diameter $\varepsilon$, in order for Boltzmann's kinetic theory to be valid in the limit $N\to \infty,\,\varepsilon\to 0$. Following several pioneering works of Grad \cite{Grad58} and Cercignani \cite{Cer72}, the first major landmark in the derivation of Boltzmann's equation happened in 1975, when Lanford \cite{Lan75} first completed this derivation for sufficiently small times. After Lanford's work, the problem of deriving Boltzmann's equation has attracted a tremendous amount of research interest from various angles and perspectives \cite{King75, Kac56, MS11, MT12, MM13, IP89, OVY93, SpohnBook, UT06, R03}. We single out particularly the progress in the past fifteen years which reinvigorated the effort of fully executing Hilbert's program, thanks to the deep works of various subsets of the following authors: Bodineau, Gallagher, Pulvirenti, Saint Raymond, Simonella, Texier \cite{GST14, PS17, PS21, BGS16, BGS17, BGS18, BGSS18, BGSS22_2, BGSS20, BGSS20_2, BGSS22, BGSS23}. We note that all those results are restricted to either  short times, or small (near-vaccuum) solutions, or various linearized settings. Such restrictions represented, for the longest time, the major obstruction to fully executing Hilbert's program, which requires accessing long-time solutions of the Boltzmann equation, as described in the above paragraph.  

This barrier was finally broken in the authors' recent work \cite{DHM24} which gave the rigorous derivation of Boltzmann's equation on $\Rb^d (d\geq 2)$  for arbitrarily long times, namely as long as the strong solution to Boltzmann's equation exists. 

Compared to the first kinetic limit, the second hydrodynamic limit was better understood earlier on. The rough idea here, which was explained by Hilbert himself in \cite{H12}, is that the local Maxwellians 
\begin{equation*}
\Mf=\Mf(t,x,v)=\frac{\rho(t,x)}{(2\pi T(t,x))^{d/2}}e^{-\frac{|v-u(t,x)|^2}{2T(t,x)}},
\end{equation*}
are approximate solutions to the Boltzmann equation, in the limit $\alpha\to\infty$, provided that the macroscopic quantities $(\rho, u, T)$ solve a corresponding hydrodynamic equation. The results in this direction can be split according to whether they address strong or weak solutions. In the case of strong solutions, we mention the works of Nishida \cite{Nishida78} and Caflisch \cite{Caf80} for the incompressible and compressible Euler limit and the works of DeMasi-Esposito-Liebowitz \cite{DEL89}, Bardos-Ukai \cite{BU91}, Guo \cite{Guo06} for the incompressible Euler and Navier-Stokes limits. Below we also rely on the more recent works of Gallagher and Tristani \cite{GT20} for the incompressible Navier-Stokes limit, and of Guo-Jang-Jiang \cite{GJJ10} and Jiang-Luo-Tang \cite{JLT21} for the compressible Euler limit. In the case of weak solutions, we mention the works of Bardos-Golse-Levermore \cite{BGL91, BGL93}, Lions-Masmoudi \cite{LM01, LM01-2}, Golse-Saint Raymond \cite{GSR04}, and refer the reader to the textbook treatment in \cite{Sai09}. 

\medskip

The purpose of this work is twofold. First we extend the derivation of Boltzmann's equation in \cite{DHM24} to the periodic setting $\Tb^d \,(d=2,3)$. Second, we connect this kinetic limit to the hydrodynamic limit in the above-cited works, to obtain a full derivation of the fluid equations starting from Newton's laws on the particle system, thereby completing Hilbert's original program. We summarize the main theorems as follows:
\begin{itemize}
\item {\bf Theorem \ref{th.main}: Derivation of the Boltzmann equation on $\Tb^d$ $(d=2,3)$}. Starting from a Newtonian hard-sphere particle system on the torus $\Tb^d$ $(d=2,3)$ formed of $N$ particles of diameter $\varepsilon$ undergoing elastic collisions, and in the Boltzmann-Grad limit $N\varepsilon^{d-1}=\alpha$, we derive the Boltzmann equation \eqref{eq.Boltzintro} as the effective equation for the one-particle density function of the particle system.
\item {\bf Theorem \ref{th.main2}: Derivation of the incompressible Navier-Stokes-Fourier system from Newton's laws.} Starting from the same Newtonian hard-sphere particle system on the torus $\Tb^d$ $(d=2,3)$ close to global equilibrium, and in an iterated limit where first $N\to\infty,\,\varepsilon\to 0$ with $\alpha=N\varepsilon^{d-1}$ fixed and then $\alpha\to \infty$ separately (there are also other variants, see Theorem \ref{th.main2}), we derive the incompressible Navier-Stokes-Fourier system as the effective equation for the macroscopic density and velocity of the particle system.
\item {\bf Theorem \ref{th.main3}: Derivation of the compressible Euler equation from Newton's laws.} Starting from the same Newtonian hard-sphere particle system on the torus $\Tb^d$ $(d=2,3)$, and in an iterated limit where first $N\to\infty,\,\varepsilon\to 0$ with $\alpha=N\varepsilon^{d-1}$ fixed and then $\alpha\to \infty$ separately (there are also other variants, see Theorem \ref{th.main3}), we derive the compressible Euler equation as the effective equation for the macroscopic density, velocity, and temperature of the particle system.

\end{itemize}

A fundamental intriguing question is the justification of the passage from the time-reversible microscopic Newton's theory to the time-irreversible mesoscopic Boltzmann theory. It is well-known (see \cite{Uka74} for smooth data and \cite{Guo10} for $L^\infty$ data, also \cite{DL89} for large-amplitude renormalized solutions) that solutions to Boltzmann's equation are global-in-time near any Maxwellian, and the celebrated Boltzmann H-theorem indicates the increase of physical entropy and time irreversibility in the Boltzmann theory. Since Theorem \ref{th.main} covers the full lifespan of the Boltzmann solution, it could be viewed as a justification of the emergence of the time irrevsesible Boltzmann theory from the time reversible Newton's theory near Maxwellian.

We shall give the precise statement of Theorems \ref{th.main}--\ref{th.main3} in the next sections, together with the proofs of Theorems \ref{th.main2} and \ref{th.main3} which will follow directly by combining Theorem \ref{th.main} with the existing results on the hydrodynamic limit. The proof of Theorem \ref{th.main} follows the same lines as the Euclidean setting in \cite{DHM24}, except for one particular part of the proof which requires substantial modifications. For this, a different algorithm is introduced to deal with a particular set of collision histories that can arise in the periodic setting. The key difference, which necessitates both new algorithms and new integral estimates in the latter setting, is the absence of upper bound on the number of collisions that can happen among a fixed number of particles (in contrast with the Euclidean case where the particles eventually move away from each other). 

\subsection{From Newton to Boltzmann}\label{sec.setup} We start by describing the microscopic hard-sphere particle dynamics, and the Boltzmann equation which describes the limits of the $s$-particle density functions determined by such dynamics.

\subsubsection{The hard sphere dynamics}\label{sec.hard_sphere}
Throughout this paper, we fix the periodic torus $\Tb^d=\Rb^d/\Zb^d$ \textbf{in dimension $d\in\{2,3\}$}. For $x\in\Tb^d$, define $|x|_\Tb$ to be the distance between $x$ and the origin $0$; if we view $x$ as a vector in $\Rb^d$ (understood modulo $\Zb^d$), then
\begin{equation}\label{eq.Tdnorm}|x|_\Tb=\min\big\{|x-y|:y\in\Zb^d\big\}.
\end{equation} If $x\in\Tb^d$ satisfies $|x|_\Tb<1/2$, then $x$ is represented by a unique vector $\widetilde{x}\in\Rb^d$ satisfying $|\widetilde{x}|=|x|_\Tb<1/2$ in the Euclidean norm of $\Rb^d$. Below we will always identify this $\widetilde{x}$ with $x$ (and do not distinguish it with $x$ in writing) without further explanations. 
\begin{definition}[The hard-sphere dynamics \cite{Alexander}]\label{def.hard_sphere} We define the hard sphere system of $N$ particles with diameter $\varepsilon<1/2$, in dimension $d\in\{2,3\}$.
\begin{enumerate}
\item\label {it.O1}\emph{State vectors $z_j$ and $\vz_N$}. We define $z_j=(x_j,v_j)\in \Tb^d\times\Rb^d$ and $\vz_N=(z_j)_{1\leq j\leq N}$, where $x_j$ and $v_j$ are the center of mass and velocity of particle $j$; this $z_j$ and $\vz_N$ are called the \textbf{state vector} of the $j$-th particle and the collection of all $N$ particles, respectively.
\item\label{it.O2}\emph{The domain $\Dc_N$}. We define the \textbf{non-overlapping domain} $\Dc_N$ as
\begin{equation}\label{eq.setDN}\Dc_N:=\big\{\vz_N=(z_1,\cdots,z_N)\in \Tb^{dN}\times\Rb^{dN}:|x_i-x_j|_\Tb\geq\varepsilon\,\,(\forall i\neq j)\big\}.
\end{equation}
\item\label{it.O3} \emph{The hard-sphere dynamics $\vz_N(t)$}. Given initial configuration $\vz_N^0=(z_1^0,\cdots,z_N^0)\in\Dc_N$, we define the \textbf{hard sphere dynamics} $\vz_N(t)=(z_j(t))_{1\leq j\leq N}$ as the time evolution of the following dynamical system:
\begin{enumerate}
\item\label{it.O31} We have $\vz_N(0)=\vz_N^0$.
\item\label{it.O32} Suppose $\vz_N(t')$ is known for $t'\in[0,t]$. If $|x_i(t)-x_j(t)|_\Tb=\varepsilon$ for some $(i,j)$, then we have
\begin{equation}\label{eq.hardsphere}
\left\{
\begin{aligned}x_i(t^+)&=x_i(t),\quad x_j(t^+)=x_j(t),\\
v_i(t^+)&=v_i(t)-\big((v_i(t)-v_j(t))\cdot\omega\big)\omega,\quad \omega:=(x_i(t)-x_j(t))/\varepsilon\in \Sb^{d-1};\\
v_j(t^+)&=v_j(t)+\big((v_i(t)-v_j(t))\cdot\omega\big)\omega.
\end{aligned}
\right.
\end{equation} Here $t^+$ indicates right limit at time $t$, see Figure \ref{fig.1}. Note that $x_j(t)$ are always continuous in $t$; in this definition we also always require $v_j(t)$ to be left continuous in $t$, so $v_j(t)=v_j(t^-)$.
\begin{figure}[h!]
\includegraphics[scale=0.2]{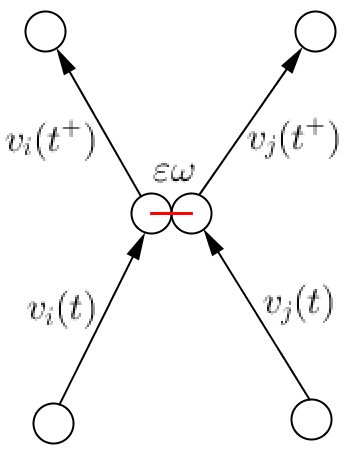}
\caption{An illustration of the hard sphere dynamics (Definition \ref{def.hard_sphere} (\ref{it.O32})): $(v_{i}(t),v_j(t))=(v_{i}(t^-),v_j(t^-))$ and $(v_{i}(t^+),v_j(t^+))$ are incoming and outgoing velocities, and $\varepsilon\omega$ is the vector connecting the centers of the two colliding particles, which has length $\varepsilon$.}
\label{fig.1}
\end{figure}
\item \label{it.O33} If for certain $i$ there is no $j$ such that the scenario in (\ref{it.O32}) happens, then we have
\begin{equation}\label{eq.hardsphere2}\frac{\mathrm{d}}{\mathrm{d}t}(x_i,v_i)=(v_i,0).
\end{equation} Here (and same below) the $\mathrm{d}x_i/\mathrm{d}t$ in (\ref{eq.hardsphere2}) is understood as the time evolution of a suitable $\Rb^d$ representative of $x_i\in\Tb^d$.
\end{enumerate}
\item \label{it.O4}\emph{The flow map $\Hc_N(t)$ and flow operator $\Sc_N(t)$}. We define the \textbf{flow map} $\Hc_N(t)$ by
\begin{equation}\label{eq.defH^O}\Hc_N(t):\Dc_N\to\Dc_N,\quad\Hc_N(t)(\vz_N^0)=\vz_N(t)\end{equation} where $\vz_N(t)$ is defined by (\ref{it.O3}). Define also the \textbf{flow operator} $\Sc_N(t)$, for functions $f=f(\vz_N)$, by
\begin{equation}\label{eq.defS^O}
(\Sc_N(t) f)(\vz_N)=f\big((\Hc_N(t))^{-1}\vz_N\big).
\end{equation}
\item \label{it.O5} \emph{Collision}. We define a \textbf{collision} to be an instance that (\ref{eq.hardsphere}) in (\ref{it.O32}) occurs. Note that these collisions exactly correspond to the discontinuities of $v_j(t)$.
\end{enumerate}
\end{definition}
\begin{proposition}\label{prop.hardsphere} Up to some closed Lebesgue zero subset of $\Dc_N$, the time evolution defined by the hard sphere system (Definition \ref{def.hard_sphere} (\ref{it.O3})) exists and is unique, and satisfies that
\begin{enumerate}
\item No two collisions happen at the same time $t$;
\item The flow maps $\Hc_N(t)$ defined in Definition \ref{def.hard_sphere} (\ref{it.O4}) are measure preserving diffeomorphisms from $\Dc_N$ to itself, and satisfy the flow or semi-group property $\Hc_N(t+s)=\Hc_N(t)\Hc_N(s)$ for $t,s\geq 0$.
\end{enumerate}
\end{proposition}
\begin{proof} This is the same as Proposition 1.2 in \cite{DHM24}. \textbf{Note that unlike the Euclidean case \cite{DHM24}, there is no upper bound on the number of collisions.}
\end{proof}
\subsubsection{The grand canonical ensemble}\label{sec.grand_canon} Define the \textbf{grand canonical domain} $\Dc:=\bigcup_{N}\Dc_N$, so $\vz\in\Dc$ always represents $\vz=\vz_N\in\Dc_N$ for some $N$. We can define the hard sphere dynamics $\vz(t)$ and the flow map $\Hc(t)$ on $\Dc$, by specifying them to be the $\vz_N(t)$ and $\Hc_N(t)$ in Definition \ref{def.hard_sphere} (\ref{it.O3})--(\ref{it.O4}), for initial configuration $\vz^0=\vz_N^0\in\Dc_N$.

The grand canonical ensemble is defined by a density function $W_0:\Dc\to\Rb_{\geq 0}$ (or equivalently a sequence of density functions $W_{0,N}:\Dc_N\to\Rb_{\geq 0}$), which determines the law of the random initial configuration for the hard-sphere system, as follows.
\begin{definition}[The grand canonical ensemble] \label{def.grand_canon} Fix $0<\varepsilon\ll 1$ and $\alpha>0$, and a nonnegative function $n_0=n_0(z)$ with $\int n_0(z)\,\mathrm{d}z=1$. We define the \textbf{grand canonical ensemble} as follows.
\begin{enumerate}
\item \label{it.Can1} \emph{Random configuration and initial density}. Assume $\vz^0\in \Dc$ is a random variable, whose law is given by the \textbf{initial density function} $(W_{0,N})$, in the sense that \begin{equation}\label{eq.inti_density}
\Pb\big(\vz^0=\vz_N^0\in A\subseteq\Dc_N\big)=\frac{1}{N!}\int_{A} W_{0,N}(\vz_N)\,\mathrm{d}\vz_N
\end{equation}for any $N$ and any $A\subseteq\Dc_N$, where $(W_{0,N})$ is given by
\begin{equation}\label{eq.N_par_ensemble} \frac{1}{N!}W_{0,N}(\vz_N):=\frac{1}{\Zc}\frac{(\alpha\cdot\varepsilon^{-(d-1)})^N}{N!}\prod_{j=1}^N n_0(z_j)\cdot \mathbbm{1}_{\Dc_N}(\vz_N).
\end{equation}
Here, $\mathbbm{1}_{\Dc_N}$ is the indicator function of the set $\Dc_N$, and the partition function $\Zc$ is defined to be
\begin{equation}\label{eq.partition_func}\Zc:=1+\sum_{N=1}^\infty\frac{(\alpha\cdot\varepsilon^{-(d-1)})^N}{N!}\int_{\Tb^{dN}\times\Rb^{dN}}\prod_{j=1}^N n_0(z_j)\cdot \mathbbm{1}_{\Dc_N}(\vz_N)\,\mathrm{d}\vz_N.\end{equation}
\item \label{it.Can2}\emph{Evolution of random configuration}. Let $\vz^0\in \Dc$ be the random variable defined in (\ref{it.Can1}) above, and let $\vz(t)=\Hc(t)\vz^0$ be the evolution of initial configuration $\vz^0$ by hard sphere dynamics, defined in Definition \ref{def.hard_sphere} (\ref{it.O3})--(\ref{it.O4}), then $\vz(t)$ is also a $\Dc$-valued random variable. We define the \textbf{density functions} $W_N(t,\vz_N)$ for the law of the random variable $\vz(t)$ by the relation
\begin{equation}\label{eq.inti_density_t}
\Pb\big(\vz(t)=\vz_N(t)\in A\subseteq\Dc_N\big)=\frac{1}{N!}\int_{A} W_{N}(t,\vz_N)\,\mathrm{d}\vz_N.
\end{equation} Then it is easy to see (using the volume preserving property in Proposition \ref{prop.hardsphere}) that
\begin{equation}\label{eq.time_t_ensemble}W_N(t,\vz_N):=(\Sc_N(t)W_{0,N})(\vz_N).
\end{equation}
\item\label{it.Can3} \emph{The correlation function $f_s(t)$}. Given $s\in\Nb$, we define the $s$-particle (rescaled) \textbf{correlation function} $f_s(t,\vz_s)$ by the following formula
\begin{equation}\label{eq.s_par_cor}
f_s(t,\vz_s):=(\alpha^{-1}\varepsilon^{d-1})^s\sum_{n=0}^\infty\frac{1}{n!}\int_{\Tb^{dn}\times\Rb^{dn}}W_{s+n}(t,\vz_{s+n})\,\mathrm{d}z_{s+1}\cdots \mathrm{d}z_{s+n},
\end{equation} where we abbreviate $\vz_s=(z_1,\cdots,z_s)$ and $\vz_{s+n}=(z_1,\cdots,z_{s+n})$ similar to $\vz_N$.
\item\label{it.Can4} \emph{The Boltzmann-Grad scaling law}. The definition (\ref{eq.N_par_ensemble}), and particularly the choice of the factor $(\alpha\cdot\varepsilon^{-(d-1)})^N$, implies that
\begin{equation}\label{eq.Bol-Grad}\Eb(N)\cdot \varepsilon^{d-1}\approx \alpha
\end{equation} up to error $O(\varepsilon)$, where $\vz^0\in\Dc$ is the random initial data defined in (\ref{it.Can1}) and $N$ is the random variable (i.e. number of particles) determined by $\vz^0=\vz_N^0\in\Dc_N$. If $\alpha$ is a constant, then we are considering the hard sphere dynamics with $\approx \alpha\cdot\varepsilon^{-(d-1)}$ particles (on average). This is referred to as the \textbf{Boltzmann-Grad scaling law}.
\end{enumerate}
\end{definition}
\subsubsection{The Boltzmann equation} The kinetic theory predicts that the one-particle correlation function $f_1(t,z)=f_1(t,x,v)$ should solve the Boltzmann equation in the kinetic limit $\varepsilon\to 0$, which is defined as follows.
\begin{definition}[The hard-sphere Boltzmann equation]\label{def.boltzmann} We define the Cauchy problem for the \textbf{Boltzmann equation} for hard sphere collisions, with initial data $n(0,z)=n_0(z)$ where $z=(x,v)\in\Tb^d\times\Rb^{d}$, as follows:
\begin{equation}\label{eq.boltzmann}
\left\{
\begin{aligned}
(\partial_t+v\cdot\nabla_x)n&=\alpha\cdot\int_{\Rb^d}\int_{\Sb^{d-1}}\big((v-v_1)\cdot\omega\big)_+\cdot(n'n_1'-nn_1)\,\mathrm{d}\omega\mathrm{d}v_1;\\
n(0,x,v)&=n_0(x,v).
\end{aligned}
\right.
\end{equation}
The right hand side of (\ref{eq.boltzmann}) is referred to as the \textbf{collision operator}, where $\big((v-v_1)\cdot\omega\big)_+$ is the positive part of $(v-v_1)\cdot\omega$, and we denote
\[
\begin{aligned}
n&=n(x,v),\quad n_1=n(x,v_1),&n'&=n(x,v'),\quad n_1'=n(x,v_1'),\\
 v'&=v-\big((v-v_1)\cdot\omega\big)\omega,&v_1'&=v_1+\big((v-v_1)\cdot\omega\big)\omega.
 \end{aligned}
 \]
\end{definition}
\subsubsection{Main result 1} We now state the first main result of this paper.
\begin{mainthm}\label{th.main} Fix $d\in\{2,3\}$ and $\beta>0$. Let $(\alpha,A,t_{\mathrm{fin}})$ be three parameters that satisfy
\begin{equation}\label{eq.loglog}
\max(1,\alpha)\cdot\max(1,A)\cdot\max(1,t_{\mathrm{fin}})\ll (\log|\log\varepsilon|)^{1/2},
\end{equation} where the implicit constant in (\ref{eq.loglog}) depends only on $(\beta,d)$. Let $n_0=n_0(z)$ be a nonnegative function with $\int n_0(z)\,\mathrm{d}z=1$, and suppose the solution $n(t,z)$ to the Boltzmann equation (\ref{eq.boltzmann}) exists on the time interval $[0,t_{\mathrm{fin}}]$, such that
\begin{equation}\label{eq.boltzmann_decay_1}
\big\|e^{2\beta  |v|^2}n(t)\big\|_{L_{x,v}^\infty}\leq A,\,\,\forall t\in[0,t_{\mathrm{fin}}];\qquad \big\|e^{2\beta  |v|^2}\nabla_xn_0\big\|_{L_{x,v}^\infty}\leq A.
\end{equation}

Consider the $d$ dimensional hard sphere system of diameter $\varepsilon$ particles (Definition \ref{def.hard_sphere}), with random initial configuration $\vz^0$ given by the grand canonical ensemble (Definition \ref{def.grand_canon}), under the Boltzmann-Grad scaling law \eqref{eq.Bol-Grad}. Let $\varepsilon$ be small enough depending on $(\beta,d)$ and the implicit constant in (\ref{eq.loglog}). Then, uniformly in $t\in[0,t_{\mathrm{fin}}]$ and in $s\leq |\log\varepsilon|$, the $s$-particle correlation functions $f_s(t)$ defined in \eqref{eq.N_par_ensemble}--\eqref{eq.s_par_cor}, satisfy that
\begin{equation}\label{eq.maineqn}
\bigg\|f_s(t,\vz_s)-\prod_{j=1}^s n(t,z_j)\cdot \mathbbm{1}_{\Dc_s}(\vz_s)\bigg\|_{L^1(\Tb^{ds}\times\Rb^{ds})}\leq \varepsilon^\theta,
\end{equation} where $\theta>0$ is an absolute constant depending only on the dimension $d$. 
\end{mainthm}
\begin{remark} We make a few remarks concerning Theorem \ref{th.main}.

(1) For simplicity of the proof, we have restricted to dimension $d\in\{2,3\}$ in Theorems \ref{th.main}--\ref{th.main3}. For $d\geq 4$ we believe the result should still be true, but both the algorithm and the integral parts of the proof will get considerably more sophisticated for larger $d$, see Remark \ref{rem.4d}.

(2) In (\ref{eq.loglog}) we have assumed $t_{\mathrm{fin}}\ll(\log|\log\varepsilon|)^{1/2}$. Improving this bound amounts to allowing for a wider range of $\delta$ in the limit in Theorems \ref{th.main2}--\ref{th.main3}. However, such an improvement seems to be hard with our methods, see \textbf{Part 7} of Section \ref{sec.proof_main}.

(3) Note that (\ref{eq.boltzmann_decay_1}) requires one more derivative on $n_0$ than the $L^1$ bound we can prove in (\ref{eq.maineqn}). This loss of derivatives can easily be reduced to $|\nabla|^a$ for any small constant $a$ or $\log\langle \nabla\rangle$, by exactly the same proof; it can also be eliminated (i.e. $a=0$), but one will need to assume the $(\alpha,A,t_{\mathrm{fin}})$ to be independent of $\varepsilon$ instead of allowing the growth in (\ref{eq.loglog}).
\end{remark}
\subsection{From Newton to Euler and Navier-Stokes} It is well known \cite{Caf80,DEL89,GT20,GSR04} that in suitable limits where $\alpha\to\infty$, solutions to the Boltzmann equation (\ref{eq.boltzmann}) that have local Maxwellian form (i.e. being Gaussian in $v$) can be described by solutions to the corresponding fluid equations, which is referred to as the \emph{hydrodynamic limit} in the literature.

Now, by combining Theorem \ref{th.main} with such results, we can obtain these fluid limits \emph{from the colliding particle systems in Definitions \ref{def.hard_sphere} and \ref{def.grand_canon}, via the Boltzmann equation (\ref{eq.boltzmann})}. We consider only two examples in this subsection, which correspond to the \emph{incompressible Navier-Stokes-Fourier} and the \emph{compressible Euler} equations.
\subsubsection{Main result 2: the incompressible Navier-Stokes-Fourier limit} We first recall the following hydrodynamic limit result in Gallagher-Tristani \cite{GT20}; see also Bardos-Ukai \cite{BU91}.
\begin{proposition}[\cite{GT20}]\label{prop.fluid_2} Let $d\in\{2,3\}$, consider the coupled incompressible Navier-Stokes-Fourier equations on $\Tb^d$:
\begin{equation}\label{eq.nsb}
\left\{
\begin{aligned}
\partial_tu+u\cdot\nabla u-\mu_1\Delta u &=-\nabla p,\\
\partial_t\rho+u\cdot\nabla\rho-\mu_2\Delta\rho&=0,\\
\mathrm{div}(u) &=0,
\end{aligned}
\right.
\end{equation} where $\mu_1$ and $\mu_2$ are two positive absolute constants depending on $d$; for exact expressions see \cite{GT20}. Here the $(u,\rho)$ in (\ref{eq.nsb}) is valued in $\Rb^d\times\Rb$, and have zero mean in $x$ (which is preserved by (\ref{eq.nsb})). We fix a smooth solution $(u,\rho)$ to (\ref{eq.nsb}) on an arbitrary time interval $[0,T_{\mathrm{fin}}]$ with initial data $(u_0,\rho_0)$, and also fix an initial perturbation $g_R=g_R(x,v)$ such that it has zero mean in $x$ for each $v$, and 
\begin{equation}\label{eq.error_nsb}\sup_{|\mu|,|\nu|\leq 2d}\big\|\langle v\rangle^{2d}\partial_x^\mu\partial_v^\nu g_R\big\|_{L^\infty}\leq 1.
\end{equation}

Now, for sufficiently small $\delta>0$, consider the Boltzmann equation (\ref{eq.boltzmann}) with $\alpha=\delta^{-1}$, and (well prepared) initial data given by
\begin{equation}\label{eq.boltzmann_nsb_1}
n_0(x,v)=\frac{1}{(2\pi)^{d/2}}e^{-\frac{|v|^2}{2}}\bigg[1+\delta\cdot\bigg(\frac{2+d-|v|^2}{2}\cdot\rho_0(x)+v\cdot u_0(x)\bigg)+\delta^4\cdot e^{\frac{|v|^2}{4}}g_R(x,v)\bigg]\geq 0.
\end{equation} Note that we may choose $g_R$ to make $n_0(x,v)$ nonnegative, and it also satisfies $\int n_0(z)\,\mathrm{d}z=1$ by (\ref{eq.boltzmann_nsb_1}). Then, for $t\in[0,{\delta^{-1}T_{\mathrm{fin}}}]$, we have
\begin{equation}\label{eq.boltzmann_nsb_2}
n(t,x,v)=\frac{1}{(2\pi)^{d/2}}e^{-\frac{|v|^2}{2}}\bigg[1+\delta\cdot\bigg(\frac{2+d-|v|^2}{2}\cdot\rho({\delta t},x)+v\cdot u({\delta t},x)\bigg)\bigg]+h({\delta t},x,v),
\end{equation} where the remainder $h=h({\tau},x,v)$ satisfies the following estimate uniformly in ${\tau}\in[0,T_{\mathrm{fin}}]$:
\begin{equation}\label{eq.boltzmann_nsb_3}
\big\|e^{|v|^2/4}\cdot h({\tau},x,v)\big\|_{L_{x,v}^\infty}+\big\|e^{|v|^2/4}\cdot\nabla_x h({\tau},x,v)\big\|_{L_{x,v}^\infty}\lesssim\delta^{3/2}.
\end{equation}
\end{proposition}
\begin{proof} This is the same as the proof of Theorem 1.1 in \cite{GT20}. We only make a few remarks here:

(1) In \cite{GT20} the Navier-Stokes-Fourier system (\ref{eq.nsb}) is stated in terms of three unknown functions $(u,\rho,\theta)$, but also with the restriction $\rho+\theta\equiv 0$, which makes it equivalent to the current formulation (\ref{eq.nsb}).

(2) The initial data (\ref{eq.boltzmann_nsb_1}) corresponds to the \emph{well-prepared} case in the notion of \cite{GT20}, plus a small remainder term $\delta^4 e^{-|v|^2/4}g_R$. It is clear, by examining the proof in \cite{GT20}, that this perturbation does not affect any of the linear and bilinear estimates there, thus the same proof in \cite{GT20} carries out and leads to the same result (\ref{eq.boltzmann_nsb_2}).

(3) Finally, the explicit convergence rate (\ref{eq.boltzmann_nsb_3}) also follows from the same proof in \cite{GT20} (at least when the solution $(u,\rho)$ is smooth), see Remark 1.4 in \cite{GT20}.
\end{proof}
Now we can state our second main result, concerning the passage from colliding particle systems to the incompressible Navier-Stokes-Fourier equation:
\begin{mainthm}
\label{th.main2}
Let $d\in\{2,3\}$, consider two small parameters $\varepsilon,\delta>0$. Also fix a smooth solution $(u,\rho)$ to the Navier-Stokes-Fourier equation (\ref{eq.nsb}) on an arbitrary time interval $[0,T_{\mathrm{fin}}]$. Here we understand that, if the solution $(u,\rho)$ is global in time, then $T_{\mathrm{fin}}$ is allowed to grow to infinity as $\varepsilon,\delta\to 0$; otherwise it is independent of $(\varepsilon,\delta)$, and in any case we assume the following relation between the parameters
\begin{equation}\label{eq.loglog_nsb}
\max(1,\delta^{-1})\cdot\max(1,{\delta^{-1}T_{\mathrm{fin}}})\ll (\log|\log\varepsilon|)^{1/2}
\end{equation} with the same implicit constant as in (\ref{eq.loglog}). We also fix $g_R$ as in (\ref{eq.error_nsb}).

Now, consider the hard sphere system of diameter $\varepsilon$ particles (Definition \ref{def.hard_sphere}), with random initial configuration $\vz^0$ given by the grand canonical ensemble (Definition \ref{def.grand_canon}), where $\alpha=\delta^{-1}$ and $n_0=n_0(x,v)$ defined as in (\ref{eq.boltzmann_nsb_1}). Note that this corresponds to the (expected) number of particles
\begin{equation}\label{eq.particle_nsb_1}\Eb(N)\approx \varepsilon^{-(d-1)}\cdot\delta^{-1}
\end{equation} which is the Boltzmann-Grad scaling \eqref{eq.Bol-Grad} when $\delta$ is constant, or barely more dense than that when {$\delta$ is a negative power of $\log|\log\varepsilon|$} (see (\ref{eq.loglog_nsb})). Let $f_1=f_1(t,x,v)$ be the $1$-particle correlation function defined in \eqref{eq.N_par_ensemble}--\eqref{eq.s_par_cor} with $s=1$. Then we have the followings:
\begin{enumerate}
\item Uniformly in $t\in[0,{\delta^{-1}T_{\mathrm{fin}}}]$, we have
\begin{equation}\label{eq.particle_nsb_2}
\bigg\|f_1(t,x,v)-\frac{1}{(2\pi)^{d/2}}e^{-\frac{|v|^2}{2}}\bigg[1+\delta\cdot\bigg(\frac{2+d-|v|^2}{2}\cdot\rho({\delta t},x)+v\cdot u({\delta t},x)\bigg)\bigg]\bigg\|_{L_{x,v}^1}\lesssim\delta^{3/2};
\end{equation} as a consequence, we also have the following limit results (with $\kappa=\theta/4$ and $\theta$ as in Theorem \ref{th.main})
\begin{equation}\label{eq.particle_nsb_3}
\begin{aligned}\lim_{\varepsilon,\delta\to 0}\frac{1}{\delta}\int_{|v|\leq \varepsilon^{-\kappa}}f_1(t,s,v)\cdot v\,\mathrm{d}v&=u({\delta t},x),\\
\lim_{\varepsilon,\delta\to 0}\frac{1}{\delta}\bigg(\int_{|v|\leq \varepsilon^{-\kappa}}f_1(t,s,v)\,\mathrm{d}v-1\bigg)&=\rho({\delta t},x).
\end{aligned}
\end{equation}
Here in (\ref{eq.particle_nsb_3}), the limit $\varepsilon,\delta\to 0$ is taken arbitrarily under the assumption (\ref{eq.loglog_nsb}) (or is taken as the iterated limit with $\varepsilon\to 0$ first followed by the $\delta\to 0$ limit), and is measured in the $L_t^\infty L_x^1$ space.
\item Let $\psi(x)$ be any test function in $x$. For each fixed $t\in[0,{\delta^{-1}T_{\mathrm{fin}}}]$, consider the random variables
\begin{equation}\label{eq.empirical_1}\begin{aligned}
u_{\mathrm{em}}[\psi]&:=\frac{1}{\delta}\cdot\frac{1}{N}\sum_{j=1}^N \mathbbm{1}_{|v_j(t)|\leq \varepsilon^{-\kappa}}\cdot\psi(x_j(t))\cdot v_j(t),\\
\rho_{\mathrm{em}}[\psi]&:=\frac{1}{\delta}\cdot\bigg(\frac{1}{N}\sum_{j=1}^N \mathbbm{1}_{|v_j(t)|\leq \varepsilon^{-\kappa}}\cdot\psi(x_j(t))-\int_{\Tb^d}\psi(x)\,\mathrm{d}x\bigg),
\end{aligned}
\end{equation} associated with the hard sphere system under the random initial configuration assumption. Then, in the limit $\varepsilon,\delta\to 0$ as described in (1) above, we have the convergence
\begin{equation}\label{eq.empirical_2}
\begin{aligned}(u_{\mathrm{em}}[\psi],\rho_{\mathrm{em}}[\psi])&\xrightarrow[\varepsilon,\delta\to 0]{\mathrm{prob.}}\int_{\Tb^d}\psi(x)\cdot (u({\delta t},x),\rho({\delta t},x))\,\mathrm{d}x
\end{aligned}
\end{equation} in probability.
\end{enumerate}
\end{mainthm}
\begin{remark}\label{rem.th2} The equalities (\ref{eq.particle_nsb_3}) and (\ref{eq.empirical_2}) show that the fluid parameters $u$ and $\rho$ can be obtained by directly taking limits of the associated statistical quantities coming from the hard sphere dynamics. The velocity truncation $|v|\leq \varepsilon^{-\kappa}$ in (\ref{eq.particle_nsb_3}) (and similarly the $|v_j(t)|\leq \varepsilon^{-\kappa}$ in (\ref{eq.empirical_1})) is merely for technical reasons, and is also natural from the physical point of view, as the probability of a particle having high velocity ($>\varepsilon^{-\kappa}$) is negligible in the kinetic limit, which easily follows from our proof.
\end{remark}
\begin{proof} (1) First, by Proposition \ref{prop.fluid_2} we get the estimates (\ref{eq.boltzmann_nsb_2})--(\ref{eq.boltzmann_nsb_3}) for the solution $n=n(t,x,v)$ to the Boltzmann equation. Using this information and the assumption (\ref{eq.loglog_nsb}), we see that the assumption (\ref{eq.loglog}) in Theorem \ref{th.main} is satisfied with (say) $\beta=1/10$. Now for fixed $t\in[0,{\delta^{-1}T_{\mathrm{fin}}}]$, by applying Theorem \ref{th.main} we get
\begin{equation}\label{eq.fluid_proof_1}\|f_1(t,x,v)-n(t,x,v)\|_{L_{x,v}^1}\leq \varepsilon^{\theta},
\end{equation}which also implies that
\begin{equation}\label{eq.fluid_proof_2}\bigg\|\int_{|v|\leq\varepsilon^{-\kappa}}|f_1(t,x,v)\cdot v-n(t,x,v)\cdot v|\,\mathrm{d}v\bigg\|_{L_x^1}\leq \varepsilon^{\kappa},
\end{equation} where $\kappa=\theta/4$. The desired bound (\ref{eq.particle_nsb_2}) then follows from (\ref{eq.fluid_proof_1}) and (\ref{eq.boltzmann_nsb_2})--(\ref{eq.boltzmann_nsb_3}), where we note that the $L_{x,v}^1$ norm is trivially controlled by the norm in (\ref{eq.boltzmann_nsb_3}) in which $h$ is bounded. As for (\ref{eq.particle_nsb_3}), we just use (\ref{eq.fluid_proof_2}), the estimate
\begin{equation}\label{eq.fluid_proof_3}\bigg\|\int_{\Rb^d}|h\cdot v|\,\mathrm{d}v\bigg\|_{L_x^1}\leq \delta^{3/2}
\end{equation} which trivially follows from (\ref{eq.boltzmann_nsb_3}), and the fact that
\begin{equation}\label{eq.fluid_proof_4}
\int_{\Rb^d}\frac{1}{(2\pi)^{d/2}}e^{-\frac{|v|^2}{2}}\bigg[1+\delta\cdot\bigg(\frac{2+d-|v|^2}{2}\cdot\rho({\tau},x)+v\cdot u({\tau},x)\bigg)\bigg]\cdot v\,\mathrm{d}v=\delta\cdot u({\tau},x)
\end{equation} (with the corresponding integral in $|v|>\varepsilon^{-\kappa}$ being trivially bounded by $\varepsilon^{10}$ in $L_x^1$) to get
\begin{equation}\label{eq.fluid_proof_5}
\bigg|\frac{1}{\delta}\int_{|v|\leq \varepsilon^{-\kappa}}f_1(t,s,v)\cdot v\,\mathrm{d}v-u({\delta t},x)\bigg|\lesssim \varepsilon^{\kappa}\delta^{-1}+\delta^{1/2},
\end{equation} which implies the first limit in (\ref{eq.particle_nsb_3}). The second limit follows in the same way.

(2) We only consider $u_{\mathrm{em}}[\psi]$, as the proof for $\rho_{\mathrm{em}}[\psi]$ is the same. Define
\begin{equation}\label{eq.fluid_proof_6}
\lambda:=\int_{|v|\leq \varepsilon^{-\kappa}}n(t,x,v)\cdot\psi(x)v\,\mathrm{d}x\mathrm{d}v=\delta\cdot \int_{\Tb^d}\psi(x)u({\delta t},x)\,\mathrm{d}x+O(\delta^{3/2}),
\end{equation} then we calculate
\begin{equation}\label{eq.fluid_proof_7}
\begin{aligned}
\Eb\big|u_{\mathrm{em}}[\psi]-\delta^{-1}\lambda\big|^2=\delta^{-2}\bigg(&\lambda^2-2\lambda\int_{|v|\leq \varepsilon^{-\kappa}}\psi(x)v\cdot f_1^*(t)\,\mathrm{d}x\mathrm{d}v+\frac{1}{N}\int_{|v|\leq \varepsilon^{-\kappa}}\psi(x)^2\cdot f_1^*\,\mathrm{d}x\mathrm{d}v\\
&+\frac{N-1}{N}\int_{|v_1|,|v_2|\leq\varepsilon^{-\kappa}}\psi(x_1)\psi(x_2)\cdot f_2^*(t)\,\mathrm{d}x_1\mathrm{d}v_1\mathrm{d}x_2\mathrm{d}v_2\bigg),
\end{aligned}
\end{equation} where $f_s^*(t,\vz_s)$ is the \emph{actual} $s$-particle density function associated with the random hard sphere dynamics, defined similar to (\ref{eq.s_par_cor}) but with different coefficients:
\begin{equation}\label{eq.s_par_cor_1}f_s^*(t,\vz_s)=(\alpha^{-1}\varepsilon^{d-1})^s\sum_{n=0}^{\infty}\frac{1}{n!}\bigg(\prod_{j=1}^s\frac{\alpha\varepsilon^{-(d-1)}}{n+j}\bigg)\cdot \int W_{s+n}(t,\vz_{s+n})\,\mathrm{d}z_{s+1}\cdots \mathrm{d}z_{s+n}.
\end{equation} Now $s\in\{1,2\}$. If we sum over $|n-\alpha\varepsilon^{-(d-1)}|\geq \alpha\varepsilon^{-(d-1)+\theta}$ in (\ref{eq.s_par_cor_1}), then we can simply control the $L_{\vz_s}^1$ norm of the resulting expression, using the $L^1$ conservation property of $\Sc_N(t)$ and the initial configuration (\ref{eq.N_par_ensemble})--(\ref{eq.partition_func}). On the other hand, if we sum over $|n-\alpha\varepsilon^{-(d-1)}|\leq \alpha\varepsilon^{-(d-1)+\theta}$ in (\ref{eq.s_par_cor_1}), then the difference between the coefficients in (\ref{eq.s_par_cor_1}) and (\ref{eq.s_par_cor}) gains an extra power $\varepsilon^{\theta}$. This then leads to 
\begin{equation}\label{eq.fluid_proof_8}\|f_s^*(t)-f_s(t)\|_{L_{x,v}^1}\leq \varepsilon^{\theta};
\end{equation} now by combining (\ref{eq.fluid_proof_8}) with (\ref{eq.maineqn}) in Theorem \ref{th.main} and plugging into (\ref{eq.fluid_proof_7}), we obtain that \begin{equation}\label{eq.fluid_proof_9}\Eb\big|u_{\mathrm{em}}[\psi]-\delta^{-1}\lambda\big|^2\leq \delta^{-2}\varepsilon^{\theta}\to 0,\end{equation} which implies the desired convergence of $u_{\mathrm{em}}[\psi]$. The proof for $\rho_{\mathrm{em}}[\psi]$ is the same.
\end{proof}
\subsubsection{Main result 3: the compresible Euler limit} We apply the compressible Euler limit result of Caflisch \cite{Caf80}, later extended by Guo-Jang-Jiang \cite{GJJ10} and Jiang-Luo-Tang \cite{JLT21}, stated as follows:
\begin{proposition}[\cite{Caf80,GJJ10,JLT21}]\label{prop.fluid_1} Let $d\in\{2,3\}$, consider the compressible Euler equations in $\Tb^d$:
\begin{equation}\label{eq.euler}
\left\{
\begin{aligned}
&\partial_t\rho+\nabla\cdot(\rho u)=0,\\
&\partial_t(\rho u)+\nabla\cdot(\rho u\otimes u)+\nabla p=0,\\
&\partial_t\bigg[\rho\bigg(\frac{dT+|u|^2}{2}\bigg)\bigg]+\nabla\cdot\bigg[\rho u\bigg(\frac{dT+|u|^2}{2}\bigg)\bigg]+\nabla\cdot(pu)=0,\\
&p=\rho\cdot T.
\end{aligned}
\right.
\end{equation} Here the $(\rho,u,T)$ in (\ref{eq.euler}) is valued in $\Rb^+\times\Rb^d\times\Rb^+$. We fix a smooth solution $(\rho,u,T)$ to (\ref{eq.euler}) on an arbitrary time interval $[0,T_{\mathrm{fin}}]$ with initial data $(\rho_0,u_0,T_0)$, and also fix an initial perturbation $F_R=F_R(x,v)$. 
Define the Maxwellian
\begin{equation}\label{eq.maxwell}\Mf=\Mf(t,x,v)=\frac{\rho(t,x)}{(2\pi T(t,x))^{d/2}}e^{-\frac{|v-u(t,x)|^2}{2T(t,x)}},
\end{equation} and subsequently define the Hilbert expansion terms $(F_0,\cdots,F_6)$ as in \cite{Caf80}:
\begin{equation}\label{eq.hilbert_exp}
\begin{aligned}
F_0&=\Mf,\\
F_{n}&=\Lc^{-1}\bigg((\partial_t+v\cdot\nabla_x)F_{n-1}-\sum_{i=1}^{n-1}\Qc(F_i,F_{n-i})\bigg)+\Mf\cdot\bigg(\rho_n\cdot\frac{1}{\rho}+u_n\cdot\frac{v-u}{\rho T}+T_n\cdot\frac{|v-u|^2-dT}{2d\rho T^2}\bigg).
\end{aligned}
\end{equation} Here $(\rho,u,T)$ etc. are functions of $(t,x)$, and $(\rho_n,u_n,T_n)$ are solutions to certain explicit linearized compressible Euler systems, see \cite{Caf80} for the exact expressions. Moreover $\Qc$ is the bilinear collision operator on the right hand side of (\ref{eq.boltzmann}) (without the pre-factor $\alpha=\delta^{-1}$), and the linear operator $\Lc$ is defined by $\Lc F=\Qc(\Mf,F)+\Qc(F,\Mf)$. Define the initial data $F_j^0(x,v)=F_j(0,x,v)$.

Now, for sufficiently small $\delta>0$, consider the Boltzmann equation (\ref{eq.boltzmann}) with $\alpha=\delta^{-1}$, and initial data given by
\begin{equation}\label{eq.boltzmann_euler_1}
n_0(x,v)=\sum_{j=0}^6\delta^j\cdot F_j^0(x,v)+\delta^3\cdot F_R(x,v)\geq 0,
\end{equation} where $F_R(x,v)$ satisfies the bound that
\begin{equation}\label{eq.boltzmann_euler_2}
\sup_{|\mu|,|\nu|\leq 2d}\big\|\Mf^{-1/2}\langle v\rangle^{2d}\partial_x^\mu\partial_v^\nu F_R\big\|_{L^\infty}\leq 1.
\end{equation} Note that we may choose the initial data of $(\rho_n,u_n,T_n)$ and $F_R$ suitably, to make $n_0(x,v)$ nonnegative and satisfy $\int n_0(z)\,\mathrm{d}z=1$. Then, for $t\in[0,T_{\mathrm{fin}}]$, we have
\begin{equation}\label{eq.boltzmann_euler_3}
n(t,x,v)=\sum_{j=0}^6\delta^j\cdot F_j(t,x,v)+h(t,x,v),
\end{equation} where the remainder $h=h(t,x,v)$ satisfies the following estimate uniformly in $t\in[0,T_{\mathrm{fin}}]$:
\begin{equation}\label{eq.boltzmann_euler_4}
\big\|\Mf^{-1/2}\cdot h(t,x,v)\big\|_{L_{x,v}^\infty}+\big\|\Mf^{-1/2}\cdot \nabla_x h(t,x,v)\big\|_{L_{x,v}^\infty}\lesssim\delta^{3/2}.
\end{equation}
\end{proposition}
\begin{proof} This is contained in \cite{GJJ10}. Note that \cite{GJJ10} make an extra technical assumption 
\begin{equation*}
\sup_{(t,x)}T(t,x)<2\cdot\inf_{(t,x)}T(t,x),
\end{equation*} but this can be removed by more refined analysis, for example by adapting the methods in \cite{JLT21} in the boundary layer case.
\end{proof}
Now we can state our third main result, concerning the passage from colliding particle systems to the compressible Euler equation:
\begin{mainthm}
\label{th.main3} Let $d\in\{2,3\}$, consider two small parameters $\varepsilon,\delta>0$. Also fix a smooth solution $(\rho,u,T)$ to the Euler equation (\ref{eq.nsb}) on an arbitrary time interval $[0,T_{\mathrm{fin}}]$. Assume $T_{\mathrm{fin}}$ is independent of $(\varepsilon,\delta)$, and otherwise make the same assumptions for these parameters as in Theorem \ref{th.main2}, including (\ref{eq.loglog_nsb}). We also fix $F_R$ as in (\ref{eq.boltzmann_euler_2}).

Now, consider the hard sphere system of diameter $\varepsilon$ particles with random initial configuration, in the same way as in Theorem \ref{th.main2} (in particular we also have the scaling law (\ref{eq.particle_nsb_1})), with $\alpha=\delta^{-1}$ and $n_0(x,v)$ defined as in (\ref{eq.boltzmann_euler_1}) which has integral $1$. Let $f_1=f_1(t,x,v)$ be the $1$-particle correlation function. Then we have the followings:
\begin{enumerate}
\item Uniformly in $t\in[0,T_{\mathrm{kin}}]$, we have
\begin{equation}\label{eq.boltzmann_euler_5}\bigg\|f_1(t,x,v)-\frac{\rho(t,x)}{(2\pi T(t,x))^{d/2}}e^{-\frac{|v-u(t,x)|^2}{2T(t,x)}}\bigg\|_{L_{x,v}^1}\lesssim\delta;
\end{equation}
as a consequence, we also have the following limit results (with $\kappa=\theta/4$ as in Theorem \ref{th.main2}): 
\begin{equation}\label{eq.particle_euler_6}
\begin{aligned}\lim_{\varepsilon,\delta\to 0}\int_{|v|\leq \varepsilon^{-\kappa}}f_1(t,s,v)\,\mathrm{d}v&=\rho(t,x),\\
\lim_{\varepsilon,\delta\to 0}\int_{|v|\leq \varepsilon^{-\kappa}}f_1(t,s,v)\cdot v\,\mathrm{d}v&=u(t,x),\\
\lim_{\varepsilon,\delta\to 0}\int_{|v|\leq \varepsilon^{-\kappa}}f_1(t,s,v)\cdot\frac{|v|^2-d}{d}\,\mathrm{d}v&=T(t,x).
\end{aligned}
\end{equation} Here in (\ref{eq.particle_euler_6}), the limit is taken in the same way as in Theorem \ref{th.main2}.
\item Let $\psi(x)$ be any test function in $x$. For each fixed $t\in[0,T_{\mathrm{fin}}]$, consider the random variables
\begin{equation}\label{eq.empirical_3}
\begin{aligned}
\rho_{\mathrm{em}}[\psi]&:=\frac{1}{N}\sum_{j=1}^N \mathbbm{1}_{|v_j(t)|\leq \varepsilon^{-\kappa}}\cdot\psi(x_j(t)),\\
u_{\mathrm{em}}[\psi]&:=\frac{1}{N}\sum_{j=1}^N \mathbbm{1}_{|v_j(t)|\leq \varepsilon^{-\kappa}}\cdot\psi(x_j(t))\cdot v_j(t),\\
T_{\mathrm{em}}[\psi]&:=\frac{1}{N}\sum_{j=1}^N \mathbbm{1}_{|v_j(t)|\leq \varepsilon^{-\kappa}}\cdot\psi(x_j(t))\cdot\frac{|v_j(t)|^2-d}{d},
\end{aligned}
\end{equation} associated with the hard sphere system under the random initial configuration assumption. Then, in the limit $\varepsilon,\delta\to 0$ as described in (1) above, we have the convergence
\begin{equation}\label{eq.empirical_4}(\rho_{\mathrm{em}}[\psi],u_{\mathrm{em}}[\psi],T_{\mathrm{em}}[\psi])\xrightarrow[\varepsilon,\delta\to 0]{\mathrm{prob.}}\int_{\Tb^d}\psi(x)\cdot(\rho(t,x),u(t,x),T(t,x))\,\mathrm{d}x
\end{equation}
in probability.
\end{enumerate}
\end{mainthm}
\begin{proof} The proof is the same as Theorem \ref{th.main2} so we will not repeat here. The only thing to note is that due to the presence of the local Maxwellian $\Mf$ instead of the global Maxwellian $e^{-|v|^2/2}$, in applying Theorem \ref{th.main}, the value of $\beta>0$ should depend on the value $\inf_{(t,x,v)}T$. However, by our assumption, this value will not depend on $(\varepsilon,\delta)$ in the limit $\varepsilon,\delta\to 0$, so the same proof in Theorem \ref{th.main2} still carries out here without any change.
\end{proof}
\subsection{Ideas of the proof} Most parts of the proof of Theorem \ref{th.main} is identical to the proof in \cite{DHM24} for the $\Rb^d$ case; in particular, we will rely on the same layered cluster forest and layered interaction diagram expansions (Sections 3--6 in \cite{DHM24}) and the same molecule analysis (Section 7--11 in \cite{DHM24}). See Section \ref{sec.proof_main} for a sketch of the parts of the proof that are identical to \cite{DHM24}. In this paper we will focus on the new ingredients needed to address the torus case. These are due to the new possible sets of collisions that are allowed in the periodic setting but not in the $\Rb^d$ setting. More precisely, we have that:
\begin{enumerate}
\item It is possible for two particles to collide twice (or more times) in a row, which is impossible in $\Rb^d$. This corresponds to the case of \emph{double bonds} (or more genrally, two ov-segments intersecting at two atoms) in terms of molecules, and this absence of double bonds in $\Rb^d$ is needed a few times in the proof in \cite{DHM24}. These parts (and only these parts) need to be modified in the torus case, but the modification is simple, by noticing that
\begin{itemize}
\item If the two particles with state vectors $(x,v)$ and $(x',v')$ ovelap/collide twice on their trajectories, then $v-v'$ must be almost parallel to some nonzero integer vector of length essentially $O(1)$.
\end{itemize}
This is because the two overlaps must happen in two different fundamental domains of $\Tb^d=\Rb^d/\Zb^d$ which are translations of each other by an integer vector $m\neq 0$, so the relative velocity must be almost parallel to $m$, and also $|m|$ is essentially $O(1)$ as it can be controlled by the sizes of $(v,v')$ and time length $t_{\mathrm{fin}}$.

It is then clear, see Proposition \ref{prop.intmini} (\ref{it.intmini_4_extra}), that the above condition restricts $v-v'$ to a set of measure $O(\varepsilon^{d-1})$ which leads to the gain in volume calculations that is enough to close the estimates. For more details, see Section \ref{sec.proof_main}, \textbf{Part 4}.
\item It is possible for a fixed number of partiles to collide arbitrarily many times, compared to the $\Rb^d$ case where the number of collisions is bounded by an absolute constant depending only on the number of particles (see Burago-Ferleger-Kononenko \cite{BFK98}). This is a more substantial difficulty, but fortunately it only affects Section 13 (i.e. the estimate for $f_s^{\mathrm{err}}$) in \cite{DHM24}. In fact, the only place where the result of \cite{BFK98} is used in \cite{DHM24}, is in Proposition 13.2 in Section 13.3 in \cite{DHM24}, which is used to control the number of collisions in some set that involves a bounded number of particles.

In order to compensate for the absence of the Burago-Ferleger-Kononenko (BFK) upper bound, we will need a new method to control the probability of these pathological collision sets. This include new sets of elementary molecules and new integral and volume estimates, as well as new components of the algorithm. We will focus on these modifications in the rest of this subsection.
\end{enumerate}
\subsubsection{Long bonds and new elementary molecules}\label{sec.LB} For the rest of this subsection, we will assume that the upper bound in \cite{BFK98} is violated, i.e. there exist $q$ particles that produce at least $G(q)$ collisions, where $G(q)$ is sufficiently large depending on $q$ (see (\ref{defG})). These collisions form a molecule $\Mb$; in the same way as Sections 13.1--13.2 in \cite{DHM24}, we only need to prove that the integral $\Jc(\Mb)$ associate with $\Mb$ (see (\ref{eq.associated_int})), which essentially represents the normalized probability that all the collisions in $\Mb$ occur for some random configuration, is bounded by $|\Jc(\Mb)|\lesssim \varepsilon^{\theta}$ for some $\theta>0$.

We shall prove the desired estimate of $\Jc(\Mb)$ by performing the cutting operations (Definition \ref{def.cutting}) to cut $\Mb$ into elementary molecules $\Mb_j$ (Definition \ref{def.elementary}), such that $\Jc(\Mb)\lesssim\prod_j\Jc(\Mb_j)$. It is easy to see that $\Jc(\Mb_j)\lesssim 1$ for \emph{all but one} $\Mb_j$, while for the exceptional molecule (i.e. the \{4\}-molecule) we have $\Jc(\Mb_j)\lesssim\varepsilon^{-(d-1)}$. Moreover, some of these estimates can be improved, which is the key to bounding $\Jc(\Mb)$. It is convenient to define the notion of \textbf{excess} for each $\Mb_j$, which is essentially the best power of $\varepsilon$ by which the above trivial estimate can be improved, see Definition \ref{def.excess}.

As such, the goal becomes to find a suitable cutting sequence that cuts $\Mb$ into elementary molecules $\Mb_j$, such that the total excess of these molecules is bounded by $\varepsilon^{d-1+\theta}$ for some $\theta>0$. See Proposition \ref{prop.comb_est_extra_2} for the precise statement. Note that the actual proposition used in the proof of Theorem \ref{th.main} is Proposition \ref{prop.comb_est_extra}, but this prpoosition follows from Proposition \ref{prop.comb_est_extra_2} and the same proof as in Section 13.3 in \cite{DHM24}, which we recall in Section \ref{sec.cutting_final} for the reader's convenience.

\smallskip
The key to the proof of Proposition \ref{prop.comb_est_extra_2} is the following: it is pointed out in \cite{GST14} and \cite{DHM24} that the worst case scenario is the \emph{short-time collisions}, when a number of particles stay close to each other at distance $O(\varepsilon)$, and thus collides many times within a time period $O(\varepsilon)$. However, even though the total number of collisions can be unbounded on torus, the number of collisions \emph{happening in any short time period} (and thus in the short-time collision scenario) does have an absolute upper bound. This is because within a short time period, the particles will stay in some fundamental domain of $\Tb^d$, thus making the short-time dynamics identical with the dynamics on $\Rb^d$, for which the BFK theorem is applicable. In other words, if the number of collisions is too large compared to the number $q$ of particles, then there must exist some adjacent collisions which are \emph{$O(1)$-separated in time}. This $O(1)$ separation is the basic property that substitutes the BFK upper bound, and allows one to avoid the short-time collisions scenario.

Motivated by the above, we define a \emph{long bond} in a molecule $\Mb$ to be a bond $e$ between two atoms $\nf_1$ and $\nf_2$, such that we make the restriction $|t_{\nf_1}-t_{\nf_2}|\geq O(1)$ for the times $t_{\nf_j}$ of the collision $\nf_j$ (for the precise definition see Section \ref{sec.newcase}). We then have the following result, which plays a fundamental role in our proof:
\begin{itemize}
\item[($\clubsuit$)] Assume $d\in\{2,3\}$. Consider any \{33A\}-molecule $\Mb$ \emph{that contains a long bond}. Then this molecule has excess $\leq \varepsilon^{d-1-\theta}$ for any $\theta>0$.
\end{itemize} For the proof of ($\clubsuit$), see Proposition \ref{prop.intmini_2} (3). 

It is known that in the worst case scenario, i.e. short-time collisions, a \{33A\}-molecule generally have excess $\varepsilon^{1-}$, see for example Proposition \ref{prop.intmini_2} (2) with $\lambda\sim 1$. Now in the (harder) case of dimension $d=3$, the long-bond condition allows us to improve this $\varepsilon^{1-}$ to $\varepsilon^{2-}$, which is almost enough for our goal.
\begin{remark}
\label{rem.4d} There are basically two reasons why the proof of Theorem \ref{th.main} in dimensions 4 and higher would require more effort. The first reason is because ($\clubsuit$) is no longer true in higher dimensions, and one would need more non-degeneracy conditions to guarantee the (optimal) $\varepsilon^{d-1-}$ excess; see the proof of Proposition \ref{prop.intmini_2} (3) below. The second reason is due to our choice of the new algorithm, which is explained in Section \ref{sec.new_alg} below.
\end{remark}
Now, suppose we can find one \{33A\}-molecule with a long bond (this requires the new algorithm, see Section \ref{sec.new_alg}). Then $(\clubsuit)$ ensures that we are already arbitrarily close to the goal of total excess $\varepsilon^{d-1+}$, and only need to fill a sub-polynomial gap. As such, we only need to find \emph{one more} good component (i.e. component with nontrivial excess) among $\Mb_j$, in order to finish the proof of Theorem \ref{th.main}.

In fact, the algorithm tells us that, either we obtain \emph{one more \{33A\}-component} other than the one with long bond (which is obviously sufficient), or we get a three-atom molecule that extends our long-bond \{33A\}-molecule, which has the form of either a \emph{\{333A\}-molecule} or a \emph{\{334T\}-molecule}, see Definition \ref{def.elementary} and Figure \ref{fig.14}. These molecules are extensions of the long-bond \{33A\}-molecule, in the sense that they contain one more collision which makes them \emph{more restrictive} than the \{33A\}-molecule alone, leading to a better excess of (say) $\varepsilon^{d-1+(1/4)}$, which is sufficient for Theorem \ref{th.main}.

The proof for the \{333A\}- and \{334T\}-molecules are contained in Proposition \ref{prop.intmini_3}. They are a bit involved but are still maganeable without computer assisted calculations.
\subsubsection{The new algorithm}\label{sec.new_alg} With the discussions in Section \ref{sec.LB}, we now describe our new algorithm that leads to the desired structure of elementary molecules. As discussed above, we would like to find a \{33A\}-molecule with a long bond; however this is not always possible, and there are cases where we need to find \emph{two separated \{33A\}-molecules} (each with excess $\varepsilon^{1-}$) which also provides the same (or better) excess in dimensions $d\in\{2,3\}$. Note that this again does not match the goal in dimensions $d\geq 4$, suggesting that a refinement of the algorithm is needed in these cases.

\smallskip
First, we shall assume our molecule $\Mb$ has a two-layer structure $\Mb=\Mb_U\cup\Mb_D$, see Proposition \ref{prop.new2layer} for the precise description. This two-layer structure has significant similarity and difference with the UD-molecules which is fundamental in the molecular analysis in \cite{DHM24} (for example the $\Mb_U$ and $\Mb_D$ are not required to be forests here), but the key feature here is that \emph{any bond connecting an atom in $\Mb_U$ and an atom in $\Mb_D$ is a long bond}. The reduction of Proposition \ref{prop.comb_est_extra_2} to Proposition \ref{prop.new2layer} is an easy process, which involves the BFK theorem for short time (in order to obtain the long bonds), an induction argument on the number of particles $q$ to guarantee the connectedness of $\Mb_U$ and $\Mb_D$, and a special argument involving \{334T\}-molecules (Lemma \ref{lem.newlem}) that treats the case of a triangle with long bond. This proof is presented in Section \ref{sec.newcase}.

\smallskip
Now, suppose $\Mb=\Mb_U\cup\Mb_D$ as above. We start by choosing a highest atom $\mf_0$ in $\Mb_D$. Note that $\mf_0$ has two parent nodes $\mf_j^+\,(j=1,2)$ in $\Mb_U$, and $\mf_0$ becomes deg 2 if both of them are cut; this can be guaranteed by the assumption that each particle line intersects both $\Mb_U$ and $\Mb_D$, see Proposition \ref{prop.new2layer}. Suppose we choose the cutting sequence such that both $\mf_j^+$ are cut before $\mf_0$, then we would get a long bond \{33A\}-molecule if the $\mf_j^+$ that is cut \emph{after} the other has deg 3 when it is cut. To achieve this, a natural idea is to cut $\mf_1^+$ (or $\mf_2^+$) first, and proceed along a path connecting $\mf_1^+$ and $\mf_2^+$ in $\Mb_U$  until reaching $\mf_2^+$. A natural candicate of this path is the following: note that $\mf_0$ belongs to two particle lines, say $\pb_1$ and $\pb_2$, where $\pb_j$ contains $\mf_j^+$. Then we  consider the \emph{first collision} (which we call $\nf_0$) in $\Mb_U$ that connects the two particles $\pb_1$ and $\pb_2$ (i.e. bring them to the same cluster). In this case, let $P_j$ be the particles (equivalently particle lines) that are connected to $\pb_j$ before $\nf_0$, and let $A_j$ be the set of collisions among particles in $P_j$ before $\nf_0$, see Proposition \ref{prop.2cluster} for details. The natural path going from $\mf_2^+$ to $\mf_1^+$ is given by concatenating the paths from $\mf_2^+$ to $\nf_0$ in $A_2$, and the path from $\nf_0$ to $\mf_1^+$ in $A_1$.

We then have the following dichotomy, as shown in Proposition \ref{prop.2cluster}: either
\begin{enumerate}[{(i)}]
\item one of the sets $A_j$ (say $j=1$) does not contain a recollision within it, or
\item each set $A_j$ contains a recollision within it.
\end{enumerate}
In fact, we will specify this recollision (i.e. cycle) to have a specific form called a \emph{canonical cycle}, see Proposition \ref{prop.newcase_1}.

\smallskip
In case (i), we can apply the plan stated above, where we start by cutting $\mf_2^+$ and follow the path from $\mf_2^+$ to $\nf_0$ in $A_2$ (cutting all atoms in $A_2$ in this process), and then follow the path from $\nf_0$ to $\mf_1^+$ in $A_1$. The fact that $A_1$ does not contain a canonical cycle then implies that $\mf_1^+$ has deg 3 when it is cut, which leads to a \{33A\}-component with a long bond. This already guarentees $\varepsilon^{d-1-\theta}$ excess by $(\clubsuit)$; in order to improve this, we notice that $\Mb$ contains many more recollisions other than the one provided by the \{33A\}-component (say $\Xc$), and a simple argument suffices to show that either there exists another \{33A\}-component other than $\Xc$, or $\Xc$ is contained in some \{333A\}- or \{334T\}-component. In each case the desired excess estimate follows from Propositions \ref{prop.intmini_2} and \ref{prop.intmini_3}. For the detailed proof, see Proposition \ref{prop.newcase_1}.

\smallskip
In case (ii), if we apply the same argument in case (i), then the atom $\mf_1^+$ would also have deg 2 when it is cut; this means that a \{33A\}-component occurs \emph{earlier} in the cutting process, which may not contain a long bond. However, note that each $A_j$ contains a canonical cycle, which in particular implies that each $A_j$ contains a \{33A\}-component in a suitable cutting sequence. Since it is known that each \{33A\}-component (without long bond) provides excess $\varepsilon^{1-}$, and that we only need to obtain excess $\varepsilon^{2+}$ in dimension $d=3$ (the $d=2$ case is much easier), we see that these two \{33A\}-components are already almost sufficient.

It remains to improve the excess a little bit upon $\varepsilon^{2-}$. This is similar to case (i) but slightly more complicated; the idea is to locate a remaining atom that belongs to a particle line \emph{other than the ones in $P_1$ and $P_2$}. This atom is then cut as deg 3, and some subsequent atom will be cut as deg 2 (due to the many recollisions), leading to a third \{33A\}-component which provides the needed gain. This is contained in Proposition \ref{prop.newcase_2}. Finally, if there is no such atom belonging to any ``new" particle line, then either the canonical cycle is not a triangle (in which case an extra gain is possible, see the proof of Proposition \ref{prop.newcase_2}), or the molecule has an essentially explicit form which contains \emph{exactly $3+3=6$ particle lines}. In this last scenario, it is not hard to develop a specific cutting sequence (using the fact that we have exactly 6 particle lines) that guarantees \emph{at least three} \{33A\}-molecules, which is again sufficient. This then finishes the proof of the excess estimate, which completes Theorem \ref{th.main}.
\subsection{Plan for the rest of this paper} In Section \ref{sec.summary0} we summarize the relevant concepts and notations from \cite{DHM24}, especially those related to molecules and cuttings, that are needed in this paper. In Section \ref{sec.integral} we state and prove the integral and volume estimates needed in the torus case. These are similar to those in Section 9.1 in \cite{DHM24} but involve more complicated calculations. We also state Proposition \ref{prop.comb_est_extra}, which is the main estimate proved in this paper, and the new ingredient that will come into the proof of Theorem \ref{th.main}. In Section \ref{sec.error} we prove Proposition \ref{prop.comb_est_extra} using the integral estimates in Section \ref{sec.integral} and the full algorithm, which consists of the same arguments in Section 13.3 in \cite{DHM24} (see Section \ref{sec.cutting_final}) and a new algorithm for the torus case (see Section \ref{sec.newcase}). Finally, in Section \ref{sec.proof_main} we prove Theorem \ref{th.main}. The proof mostly follows the same lines in \cite{DHM24} which is also sketched here; the new ingredients are the modifications in the algorithm due to double bonds (see \textbf{Part 4}), and the estimate for the $f_s^{\mathrm{err}}$ term in \cite{DHM24} which relies on Proposition \ref{prop.comb_est_extra}, see \textbf{Part 6}.

\subsection*{Acknowledgements} The authors are partly supported by a Simons Collaboration Grant on Wave Turbulence. The first author is supported in part by NSF grant DMS-2246908 and a Sloan Fellowship. The second author is supported in part by NSF grants DMS-1936640 and DMS-2350242. The authors would like to thank Isabelle Gallagher, Yan Guo, and Lingbing He for their helpful suggestions.
\section{Summary of concepts from \cite{DHM24}}\label{sec.summary0}
\subsection{Molecules and associated notions}\label{sec.recall_mol} We recall the definition of molecules and some associated notions in \cite{DHM24}. Since the main proofs in the current paper (Sections \ref{sec.integral}--\ref{sec.error}) only requires molecules formed solely by C-atoms, we will only review the relevant definitions involving C-atoms here. For the full definitions involving both C- and O-atoms and empty ends etc., see Definitions 7.1--7.3 in \cite{DHM24}.
\begin{definition}[Molecules]
\label{def.molecule}
We define a \newterm{(layered) molecule} to be a structure $\Mb=(\Mc,\Ec,\Pc)$, see Figure \ref{fig.6}, which satisfies the following requirements.
\begin{enumerate}
    \item \textit{Atoms, edges, bonds and ends.} Assume $\Mc$ is a set of nodes referred to as \newterm{atoms}, and $\Ec$ is a set of \newterm{edges} consisting of \newterm{bonds} and \newterm{ends} to be specified below.
    \begin{enumerate}
        \item Each atom $\nf\in \Mc$ is marked as a \newterm{C-atom}, and is assigned with a unique \newterm{layer} $\ell[\nf]$ which is a positive integer.
        \item\label{it.molecule_edge} Each edge in $\Ec$ is either a directed \newterm{bond} $(\nf\to \nf')$, which connects two distinct atoms $\nf,\nf'\in\Mc$, or a directed \newterm{end}, which is connected to one atom $\nf\in\Mc$. Each end is either \newterm{incoming} $(\to \nf)$ or \newterm{outgoing} $(\nf\to)$, and is also marked as either \newterm{free} or \newterm{fixed}.
    \end{enumerate}
        \item\label{it.particle_line_same} \textit{Serial edges.}  The set $\Pc$ specifies the particle line structure of $\Mb$ (cf. Definition \ref{def.molecule_more} \eqref{it.molecule_more_5}). More precisely, $\Pc$ is a set of maps such that, for each atom $\nf$, $\Pc$ contains a bijective map from the set of 2 incoming edges at $\nf$ to the set of 2 outgoing edges at $\nf$. Given an incoming edge $e_+$ and its image $e_-$ under this map, we say $e_+$ and $e_-$ are \newterm{serial} or \newterm{on the same particle line} at $\nf$.
    \item $(\Mc, \Ec)$ satisfies the following requirements.
    \begin{enumerate}
        \item Each atom has exactly 2 incoming and 2 outgoing edges. 
        \item There does not exist any closed directed path (which starts and ends at the same atom). \textbf{Note that unlike the Euclidean case \cite{DHM24}, double bonds are allowed in $\Mb$.}
        \item For any bond $\nf\to\nf'$, we must have $\ell[\nf]\geq\ell[\nf']$.
    \end{enumerate}
\end{enumerate}
\end{definition}
\begin{figure}[H]
    \includegraphics[scale=0.17]{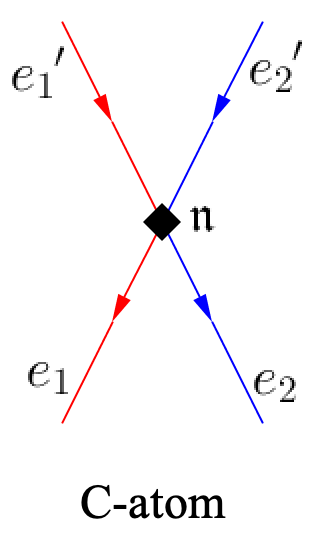}
    \caption{Illustration of a C-atom in the molecule $\Mb$ in Definition \ref{def.molecule}. Here the C-atom is represented by a black diamond shape, the arrow indicates the direction of an edge, and we always require that this direction goes from top to bottom (cf. Definition \ref{def.molecule_more}). Serial edges are indicated by the same color, and each edge can be either a free or fixed end, or a bond connecting to another C-atom (not drawn in the picture).}
    \label{fig.6}
\end{figure}
\begin{definition}[More on molecules]
\label{def.molecule_more}
We define some related concepts for molecules.
\begin{enumerate}
    \item\label{it.full_molecule} \emph{Full molecules.} If a molecule $\Mb$ has no fixed end, we say $\Mb$ is a \newterm{full molecule}.
    \item \emph{Parents, children and descendants.} Let $\Mb$ be a molecule. If two atoms $\nf$ and $\nf'$ are connected by a bond $\nf\to\nf'$, then we say $\nf$ is a \newterm{parent} of $\nf'$, and $\nf'$ is a \newterm{child} of $\nf$. If $\nf'$ is obtained from $\nf$ by iteratively taking parents, then we say $\nf'$ is an \newterm{ancestor} of $\nf$, and $\nf$ is a \newterm{descendant} of $\nf'$ (this includes when $\nf'=\nf$). 
    
    \item \emph{The partial ordering of nodes.} We also define a \newterm{partial ordering} between all atoms, by defining $\nf'\prec\nf$ if and only if $\nf'\neq\nf$ and is a descendant of $\nf$. In particular $\nf'\prec\nf$ if $\nf'$ is a child of $\nf$. Using this partial ordering, we can define the \newterm{lowest} and \newterm{highest} atoms in a set $A$ of atoms, to be the minimal and maximal elements of $A$ under the partial ordering (such element need not be unique).
    
    \item \emph{Top and bottom edges.} We also refer to the incoming and outgoing edges as \newterm{top} and \newterm{bottom} edges.
    
    \item\label{it.molecule_more_5} \emph{Particle lines.} For any $\Mb$, we define a \newterm{particle line} of $\Mb$ to be a sequence of distinct edges $(e_1,\cdots,e_q)$ (together with all atoms that are endpoints of these edges), such that $e_1$ is a top (free or fixed) end, $e_q$ is a bottom (free or fixed) end, and $(e_j,e_{j+1})$ are serial (Definition \ref{def.molecule} \eqref{it.particle_line_same}) at some atom for each $1\leq j<q$. 
    \item\label{it.subsets_Ec} \emph{Sets $\Ec_*$ and $\Ec_{\mathrm{end}}^-$.} Define $\Ec_*$ to be the set of all bonds and free ends. Define also the sets $\Ec_{\mathrm{end}}$ (resp. $\Ec_{\mathrm{end}}^-$) to be the set of all ends (resp. all bottom ends).
    \item\label{it.deg}\emph{Degree and other notions.} Define the \textbf{degree} (abbreviated deg) of an atom $\nf$ to be the number of bonds and free ends at $\nf$ (i.e. not counting fixed ends). We also define $S_\nf$ to be the set of descendants of $\nf$. Finally, for a subset $A\subset\Mb$, define $\rho(A):=|\Bc(A)|-|A|+|\Fc(A)|$, where where $\Bc(A)$ is the set of bonds between atoms in $A$, and $\Fc(A)$ is the set of components of $A$ (where $A$ is viewed as a subgraph of $\Mb$).
\end{enumerate}

It is clear, with the definition of particle lines in (\ref{it.molecule_more_5}), that each edge belongs to a unique particle line, and each particle line contains a unique bottom end and a unique top end, so the set of particle lines is in one-to-one correspondence with the set $\Ec_{\mathrm{end}}^-$ of bottom ends and also with the set of top ends.
\end{definition}
In the next definition, we define the integral $\Jc(\Mb)$ associated with a molecule $\Mb$. Note that for the purpose of this paper, the more relevant notion here is the $\Jc(\Mb)$ in Definition \ref{def.associated_int} (\ref{it.associated_int}), which is the same as the one occurring in Section 9.1 in \cite{DHM24} (see also the $\widetilde{\Ic}(\Mb')$ in Proposition 8.10 in \cite{DHM24}), instead of the $\Ic\Nc(\Mb,H,H')$ in Definition 7.3 in \cite{DHM24}.
\begin{definition}[The integral $\Jc(\Mb)$ for molecules $\Mb$]
\label{def.associated_int}
Consider the molecule $\Mb=(\Mc,\Ec,\Pc)$, we introduce the following definitions.
\begin{enumerate}
    \item\label{it.associated_vars} \emph{Associated variables $\vz_\Ec$ and $\vt_\Mc$.} We associate each node $\nf\in \Mc$ with a time variable $t_\nf$ and each edge $e\in \Ec$ with a position-velocity vector $z_e=(x_e,v_e)$. We refer to them as \newterm{associated variables} and denote the collections of these variables by $\vz_\Ec=(z_e:e\in\Ec)$ and $\vt_\Mc=(t_\nf:\nf\in\Mc)$.
    \item\label{it.associated_dist} \emph{Associated distributions $\boldsymbol{\Delta}_\nf$.} Given a C-atom $\nf\in\Mc$, let $(e_1,e_2)$ and $(e_1',e_2')$ be are two bottom edges and top edges at $\nf$ respectively, such that $e_1$ and $e_1'$ (resp. $e_2$ and $e_2'$) are serial. We define the \newterm{associated distribution} $\Dirac_\nf=\boldsymbol{\Delta}=\boldsymbol{\Delta}(z_{e_1},z_{e_2},z_{e_1'},z_{e_2'},t_\nf)$ by
    \begin{equation}\label{eq.associated_dist_C}
        \begin{aligned}
            \boldsymbol{\Delta}(z_{e_1},z_{e_2},z_{e_1'},z_{e_2'},t_\nf)&:=\boldsymbol{\delta}\big(x_{e_1'}-x_{e_1}+t_\nf(v_{e_1'}-v_{e_1})\big)\cdot \boldsymbol{\delta}\big(x_{e_2'}-x_{e_2}+t_\nf(v_{e_2'}-v_{e_2})\big)\\&\times \boldsymbol{\delta}\big(|x_{e_1}-x_{e_2}+t_\nf(v_{e_1}-v_{e_2})|_\Tb-\varepsilon\big)\cdot \big[(v_{e_1}-v_{e_2})\cdot \omega\big]_-\\&\times\boldsymbol{\delta}\big(v_{e_1'}-v_{e_1}+[(v_{e_1}-v_{e_2})\cdot\omega]\omega\big)\cdot\boldsymbol{\delta}\big(v_{e_2'}-v_{e_2}-[(v_{e_1}-v_{e_2})\cdot\omega]\omega\big),
        \end{aligned}
        \end{equation} 
        where $\omega:=\varepsilon^{-1}(x_{e_1}-x_{e_2}+t_\nf(v_{e_1}-v_{e_2}))$ is a unit vector, and $z_-=-\min(z,0)$.
    \item\label{it.associated_domain} \emph{The associated domain $\mathcal{D}$.} The \textbf{associated domain} $\mathcal{D}$ is defined by
    \begin{equation}\label{eq.associated_domain}
        \Dc:=\big\{\vt_\Mc=(t_\nf)\in (\mathbb{R}^+)^{|\Mc|}: (\ell'-1)\tau<t_\nf<\ell'\tau\mathrm{\ if\ }\ell[\nf]=\ell';\ t_{\nf}<t_{\nf^+}\mathrm{\ if\ }\nf^+\mathrm{\ is\ parent\ of\ }\nf \big\}.
    \end{equation} Here $1\gg\tau>(\log|\log\varepsilon|)^{-1/2}$ is a small parameter (for the exact choice see Section \ref{sec.proof_main}, \textbf{Part 7}).  
    \item\label{it.associated_int} \emph{The integral $\Jc(\Mb)$.} Let $\Mb=(\Mc,\Ec,\Pc)$ be a molecule (which may or may not be full), with relevant notions as in Definitions \ref{def.molecule}--\ref{def.molecule_more}, and $Q$ be a given nonnegative function. We define the integral $\Jc(\Mb)$ by
    \begin{equation}\label{eq.associated_int}
        \Jc(\Mb):=\varepsilon^{-(d-1)(|\Ec_*|-2|\Mc|)}\int_{\Tb^{d|\Ec_*|}\times\mathbb{R}^{d|\Ec_*|}\times \Rb^{|\Mc|}}\prod_{\nf\in\Mc}\boldsymbol{\Delta}(z_{e_1},z_{e_2},z_{e_1'},z_{e_2'},t_\nf)\cdot Q(\vz_\Ec,\vt_\Mc)\mathrm{d}\vz_{\Ec_*}\mathrm{d}\vt_\Mc.
    \end{equation}
    Here in \eqref{eq.associated_int}, $|\Ec_*|$ and $|\Mc|$ are the cardinalities of $\Ec_*$ and $\Mc$ with $\Ec_*$ as in Definition \ref{def.molecule} (note $\Ec_*=\Ec$ if $\Mb$ is a full molecule), and $\vz_{\Ec_*}$ is defined similar to $\vz_\Ec$ in (\ref{it.associated_vars}). We will assume that $Q$ is supported in the set $t_\Mc\in\Dc$, with possibly other restrictions on its support to be specified below.
    \end{enumerate}
\end{definition}
\subsection{Cutting operations}\label{sec.recall_cut} We next recall the notion of cutting in \cite{DHM24}. Like in Section \ref{sec.recall_mol}, we will only consider the case of cuttings involving C-atoms. For the full definition involving cutting both C- and O-atoms, see Definition 8.6 (and related properties in Section 8.2) in \cite{DHM24}.
\begin{definition}\label{def.cutting} Let $\Mb$ be a molecule formed solely by C-atoms and $A\subseteq\Mb$ be a set of atoms. Define the operation of \textbf{cutting $A$ as free}, where for each bond between $\nf\in A$ and $\nf'\in \Mb\backslash A$ is turned into a free end at $\nf$ and a fixed end at $\nf'$. Note that each free end $e_{\mathrm{free}}$ formed after this cutting (at atoms in $A$) is uniquely paired with a fixed end $e_{\mathrm{fix}}$ formed after this cutting (an atoms in $\Mb\backslash A$), and this pair is associated with a unique bond $e$ before the cutting, see Figure \ref{fig.11} for an illustration. We also define \textbf{cutting $A$ as fixed} to be the same as cutting $\Mb\backslash A$ as free. 

Given a molecule $\Mb$, we define a \textbf{cutting sequence} to be a sequence of cutting operations, each applied to the result of the previous cutting in the sequence, starting with $\Mb$. In considering cutting sequences below, we will adopt the following convention. Note that after cutting any $A$ as free or fixed from $\Mb$, the atom set of $\Mb$ remains unchanged, but $\Mb$ is divided into two disjoint molecules, $\Mb^1$ with atom set $A$ and $\Mb^2$ with atom set $\Mb\backslash A$. Now depending on the context, we may tag $\Mb^1$ as \textbf{protected}, in which case we will not touch anything in $\Mb^1$ and will only work in $\Mb^2$ in subsequent cuttings. Throughout this process we will abuse notations and replace $\Mb$ by $\Mb^2$.
\begin{figure}[h!]
\includegraphics[scale=0.18]{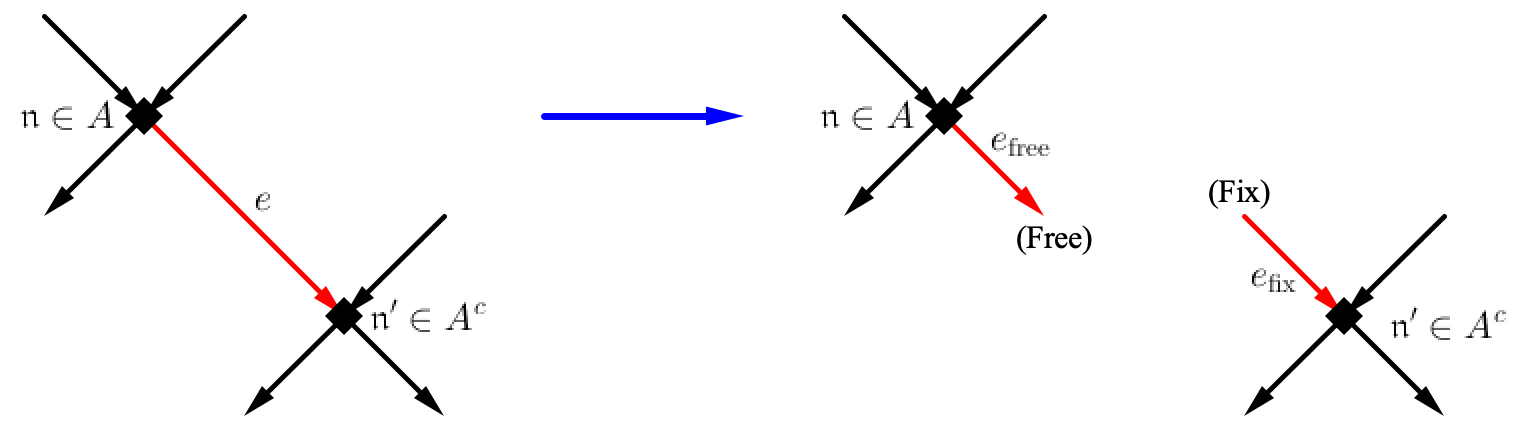}
\caption{An illustration of cutting operation in Definition \ref{def.cutting}, when $\nf\in A$ and $\nf'\in A^c$ are C-atoms. The relevant notations also correspond to those in Definition \ref{def.cutting}.}
\label{fig.11}
\end{figure}
\end{definition}
Next we define a natural linear ordering $\prec_{\mathrm{cut}}$ among all connected components of the resulting molecule after any cutting sequence, see Proposition 8.8 in \cite{DHM24}.
\begin{definition}\label{def.match_ends} Let $\Mb$ be a \emph{full} molecule, and $\Mb'$ be the result of $\Mb$ after any cutting sequence. Then, with $\Mb$, $\Mb'$ and the cutting sequence fixed, we can define an ordering between the components of $\Mb'$:
\begin{enumerate}
\item\label{it.comp_order_1} If the first cutting in the sequence cuts $\Mb$ into $\Mb_1$ with atom set $A$ and $\Mb_2$ with atom set $\Mb\backslash A$, where $A$ is cut as free, then each component of $\Mb'$ with atom set contained in $A$ shall occur in the ordering \emph{before} each component of $\Mb'$ with atom set contained in $\Mb\backslash A$.
\item\label{it.comp_order_2} The ordering between the components with atom set contained in $A$, and the ordering between the components with atom set contained in $\Mb\backslash A$, are then determined inductively, by the cutting sequence after the first cutting.
\end{enumerate}
Denote the above ordering by $\prec_{\mathrm{cut}}$, where $\Xc\prec_{\mathrm{cut}}\Yc$ means $\Xc$ occurs before $\Yc$ in this ordering.
\end{definition}
Finally we introduce the notion of elementary molecules (or components), see Definition 8.9 in \cite{DHM24}. However, here we only consider those formed solely by C-atoms; moreover, in view of the new estimates needed in the current paper, we need to extend the list of elementary molecules in \cite{DHM24} to include two new ones, namely the \{333A\}- and \{334T\}- molecules, see Definition \ref{def.elementary} below.

\begin{definition}\label{def.elementary} We define a molecule $\Mb$ formed by C-atoms to be \textbf{elementary}, if it satisfies one of the followings:
\begin{enumerate}
\item\label{it.elementary_1} This $\Mb$ contains only one atom that has deg 2, 3 or 4; we further require that the two fixed ends are not serial, and are \emph{either both top or both bottom} for C-atom.
\item\label{it.elementary_2} This $\Mb$ contains only two atoms connected by a bond, and their degs are either $(3,3)$ or $(4,4)$.
\item \label{it.elementary_3} This $\Mb$ contains only three atoms, such that two of them have deg 3 and are connected by one bond, and the third atom either has deg 3 and is connected to one of them, or has deg 4 and is connecte to both of them.
\end{enumerate}
For those $\Mb$ in (\ref{it.elementary_1}) and (\ref{it.elementary_2}) we shall denote it by \{2\}-, \{3\}-, \{4\}-, \{33\}- and \{44\}-molecules in the obvious manner; for those $\Mb$ in (\ref{it.elementary_3}) we shall denote it by \{333A\}- or \{334T\}-molecules (where T stands for ``triangle"). Moreover, for a \{33\}-molecule $\Mb$, we denote it by \{33A\} if we can cut one atom as free, such that the other atom has 2 fixed ends that are be both top or both bottom; otherwise denote it by \{33B\} (note that a \{33\} molecule is \{33B\} if and only if there is one top fixed end at the higher atom and one bottom fixed end at the lower atom). For \{333A\}- and \{334T\}-molecules, we also require that the first 2 atoms $\nf_1$ and $\nf_2$ form a \{33A\}-molecule with two fixed ends being both bottom or both top, and that the 2 fixed ends at the third atom $\nf_3$, after cutting $\{\nf_1,\nf_2\}$ as free, must be both top or both bottom. For an illustration see Figure \ref{fig.14}.

Suppose applying a cutting sequence to $\Mb$ generates a number of elementary components, plus the rest of the molecule $\Mb$ whose atoms are not cut by this cutting sequence, then we define $\#_{\{2\}}$ and $\#_{\{33\}}$ etc. to be the number of \{2\}- and \{33\}- etc. components generated in this cutting sequence. We also understand that, when any elementary component is generated in any cutting sequence, this component is automatically tagged as \emph{protected}.
\begin{figure}[h!]
\includegraphics[scale=0.15]{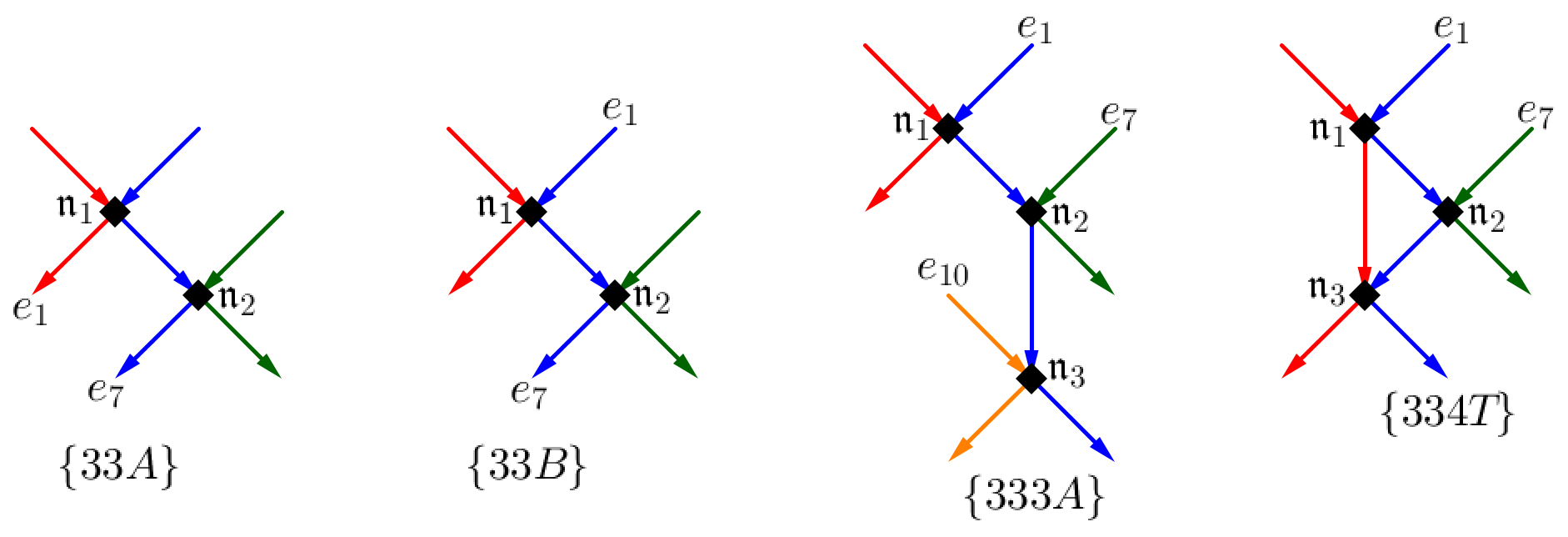}
\caption{An example of \{33A\}, \{33B\}-, \{333A\}- and \{334T\}-molecules (Definition \ref{def.elementary}). Here the ends marked by $e_1$ and $e_7$, and the one marked by $e_{10}$ in the \{333A\} case are fixed, and the other ends are free.}
\label{fig.14}
\end{figure}
\end{definition}
\section{Treating the integral}\label{sec.integral}
\subsection{Integrals for elementary molecules}\label{sec.elem_int} In this subsection we recall the estimates in Section 9.1 in \cite{DHM24} concerning the integral $\Jc(\Mb)$ (Definition \ref{def.associated_int}) for elementary molecules $\Mb$; in addition, we also prove some improvements of these estimates and some new estimates involving the new elementary molecules defined in the current paper. Throughout this section we will assume our molecules only consist of C-atoms; the case of O-atoms will have similar results with simpler proofs, and we will not elaborate on them.

Let $\Mb$ be an elementary molecule as in Definition \ref{def.elementary}, and consider the integral
\begin{equation}\label{eq.intmini}\Jc(\Mb):=\varepsilon^{-(d-1)(|\Ec_*|-2|\Mc|)}\int_{\Tb^{d|\Ec_*|}\times\Rb^{d|\Ec_*|}\times\Rb^{|\Mc|}}\prod_{\nf\in\Mc}\boldsymbol{\Delta}(z_{e_1},z_{e_2},z_{e_1'},z_{e_2'},t_\nf)\cdot Q(\vz_\Ec,\vt_\Mc)\,\mathrm{d}\vz_{\Ec_*}\mathrm{d}\vt_\Mc
\end{equation} as defined in (\ref{eq.associated_int}) in Definition \ref{def.associated_int} (\ref{it.associated_int}). Note that for fixed ends $e$ of $\Mb$, the variable $z_e$ is not integrated in (\ref{eq.intmini}) and acts as a parameter in this integral. The time variable at $\nf_j\,(1\leq j\leq 3)$ in (\ref{eq.intmini}) will be denoted by $t_j$. When $\Mb$ contains only one atom, we will also denote $(e_1',e_2')$ by $(e_3,e_4)$; when $\Mb$ contains two (resp. three) atoms, we may also denote the edges by $e_j$ for $1\leq j\leq 7$ (resp. $1\leq j\leq 10$) depending on the context. The variable $z_{e_j}$ is then abbreviated as $z_j=(x_j,v_j)$.

\begin{remark}\label{rem.parameter} Throughout this paper we will use $C$ to denote a large sonstant depending only on $(\beta,d)$, and $C^*$ be any large quantity such that $C\ll C^*\ll|\log\varepsilon|^{\theta}$, where $0<\theta\ll 1$ is a small constant depending only on $d$. Below we always assume that the $Q$ in (\ref{eq.intmini}) is nonnegative and supported in $|v_e|\leq|\log\varepsilon|^{C^*}$ for each $e\in\Ec$ (thanks to the Maxwellian decay of the Boltzmann density (\ref{eq.boltzmann_decay_1})) and $t_\nf\in[(\ell[\nf]-1)\tau,\ell[\nf]\tau]$ for each $\nf\in\Mc$, where $\ell[\nf]$ is a fixed integer for each $\nf$. We may make other restrictions on the support of $Q$, which will be discussed below depending on different scenarios; such support may also depend on some other external parameters such as $(x^*,v^*,t^*)$, which will be clearly indicated when they occur below.
\end{remark}
\begin{proposition}\label{prop.intmini} Let $\Jc(\Mb)$ be defined as in (\ref{eq.intmini}), where $\Mb$ contains only one atom.
\begin{enumerate}
\item\label{it.intmini_1} If $\Mb$ is a \{2\}-molecule as in Definition \ref{def.elementary}, i.e. the two fixed ends are either both bottom or both top, by symmetry we may assume $(e_1,e_2)$ are fixed and $(e_3,e_4)$ are free. Then we have
\begin{equation}\label{eq.intmini_1}\Jc(\Mb)=\sum_{\substack{t_1:|x_1-x_2+t_1(v_1-v_2)|_\Tb=\varepsilon\\(v_1-v_2)\cdot\omega\leq 0}}Q.
\end{equation} Here in (\ref{eq.intmini_1}), the summation is taken over the (discrete) set of $t_1$ that satisfies the following equation:
\begin{equation}\label{eq.intmini_1+}|x_1-x_2+t_1(v_1-v_2)|_\Tb=\varepsilon,\quad (v_1-v_2)\cdot\omega\leq 0,\end{equation} where $\omega:=\varepsilon^{-1}(x_1-x_2+t_1(v_1-v_2))$ is a unit vector. Moreover $Q$ is the same function in (\ref{eq.intmini}) with input variables $(z_1,z_2,z_3,z_4,t_1)$, where (each choice of) $t_1$ is a function of $(z_1,z_2)$ defined above, and $(z_3,z_4)$ is also some function of $(z_1,z_2)$, which is defined via (\ref{eq.hardsphere}) conjugated by the free transport $(x,v)\leftrightarrow (x-t_1v,v)$.

Note that all the values of $t_1$ occurring in the summation in (\ref{eq.intmini_1}) can be indexed by a variable $m\in\Zb^d$ with $|m|\leq|\log\varepsilon|^{C^*}$, which represents different fundamental domains of $\Tb^d$ within $\Rb^d$. The number of choices of $m$ is $\leq|\log\varepsilon|^{C^*}$, and this allows to decompose
\begin{equation}\label{eq.extradecomp}
\Jc(\Mb)=\sum_{|m|\leq |\log\varepsilon|^{C^*}}\Jc_m(\Mb),\qquad 0\leq \Jc_m(\Mb)\leq Q,
\end{equation} where for each $m$, the variables in $Q$ in (\ref{eq.extradecomp}) is defined as above with $t_1$ being the solution to (\ref{eq.intmini_1}) indexed by $m$.
\item\label{it.intmini_2} If $\Mb$ is a \{3\}-molecule as in Definition \ref{def.elementary}, by symmetry we may assume $e_1$ is fixed and $(e_2,e_3,e_4)$ are free. Then
\begin{equation}\label{eq.intmini_3}
\Jc(\Mb)=\int_{\Rb\times\Sb^{d-1}\times\Rb^d}\big[(v_1-v_2)\cdot\omega\big]_-\cdot Q\,\mathrm{d}t_1\mathrm{d}\omega\mathrm{d}v_2.
\end{equation} Here in (\ref{eq.intmini_3}) the input variables of $Q$ are $(z_1,z_2,z_3,z_4,t_1)$, where (similar to (\ref{eq.hardsphere}))
\begin{equation}\label{eq.intmini_4}
\left\{
\begin{aligned}
x_2&=x_1+t_1(v_1-v_2)-\varepsilon\omega,\\
x_3&=x_1+t_1\big[(v_1-v_2)\cdot\omega\big]\omega,\\
v_3&=v_1-\big[(v_1-v_2)\cdot\omega\big]\omega,\\
x_4&=x_1+t_1\big[v_1-v_2-((v_1-v_2)\cdot\omega)\omega\big]-\varepsilon\omega,\\
v_4&=v_2+[(v_1-v_2)\cdot\omega\big]\omega.
\end{aligned}
\right.
\end{equation}
In the integral (\ref{eq.intmini_4}) the domain of integration can be restricted to $t_1\in [(\ell[\nf]-1)\tau,\ell[\nf]\tau]$ and $|v_2|\leq|\log\varepsilon|^{C^*}$, which has volume $\leq|\log\varepsilon|^{C^*}$. The same logarithmic upper bound also holds for the weight $\big[(v_1-v_2)\cdot\omega\big]_-$.

In addition, assume $Q$ is supported in some set depending on some external parameters $(x^*,t^*)$, and the support satisfies
\begin{equation}\label{eq.intmini_new}\min(|x_i-x_j|_\Tb,|v_i-v_j|)\leq\lambda\quad\mathrm{or}\quad \min(|t_1-t^*|,|v_j-v^*|,|x_j-x^*|_\Tb)\leq\lambda
\end{equation} for some $\varepsilon\lesssim\lambda\lesssim 1$, where $i\neq j$ in the first case in (\ref{eq.intmini_new}) and $j\neq 1$ in the second case. Then in each case, the domain of integration in (\ref{eq.intmini_3}) can be restricted to a set of $(t_1,\omega,v_2)$ that depends on $(x_1,v_1)$ and the external parameters, which has volume $\leq \lambda\cdot|\log\varepsilon|^{C^*}$.
\item\label{it.intmini_3} If $\Mb$ is a \{4\}-molecule as in Definition \ref{def.elementary}, and we still consider the same $\Jc(\Mb)$ in (\ref{eq.intmini}), then
\begin{equation}\label{eq.intmini_6}
\Jc(\Mb)=\varepsilon^{-(d-1)}\int_{\Rb\times\Sb^{d-1}\times(\Rb^d)^3}\big[(v_1-v_2)\cdot\omega\big]_-\cdot Q\,\mathrm{d}t_1\mathrm{d}\omega\mathrm{d}x_1\mathrm{d}v_1\mathrm{d}v_2,
\end{equation} where the inputs $(z_1,z_2,z_3,z_4,t_1)$ of $Q$ are as in (\ref{eq.intmini_4}). In the integral (\ref{eq.intmini_6}) the domain of integration can be restricted to $t_1\in [(\ell[\nf]-1)\tau,\ell[\nf]\tau]$ and $|v_j|\leq|\log\varepsilon|^{C^*}\,(j\in\{1,2\})$, which has volume $\leq|\log\varepsilon|^{C^*}$. The same logarithmic upper bound also holds for the weight $\big[(v_1-v_2)\cdot\omega\big]_-$.

In addition, assume $Q$ is supported in the set (depending on some external parameters $x^*$)
\begin{equation}\label{eq.intmini_6+}
\min(|x_i-x_j|_\Tb,|v_i-v_j|)\leq\lambda\quad\mathrm{or}\quad \min(|t_1-t^*|,|v_j-v^*|,|x_j-x^*|_\Tb)\leq\lambda
\end{equation}
for some $\varepsilon\lesssim\lambda\lesssim 1$ and $i\neq j$, then in this case, the domain of integration in (\ref{eq.intmini_6}) can be restricted to a set of $(t_1,\omega,x_1,v_1,v_2)$ that depends on the external parameters, which has volume $\leq \lambda\cdot|\log\varepsilon|^{C^*}$.
\item\label{it.intmini_4_extra} In (\ref{it.intmini_2}) and (\ref{it.intmini_3}) above, suppose we do not assume (\ref{eq.intmini_new}) or (\ref{eq.intmini_6+}). Instead, we assume that $Q$ is supported in the set where, for the vectors $(x_j,v_j)$ and $(x_{j+1},v_{j+1})$ with some $j\in\{1,3\}$, there exist \emph{at least two} different values $t_1$ such that the equality (\ref{eq.intmini_1+}) holds. Then, we can restrict the domain of integration in (\ref{eq.intmini_3}) (and (\ref{eq.intmini_6})) to a set of $(t_1,\omega,v_2)$ that depends on $(x_1,v_1)$ (or a set of $(t_1,\omega,x_1,v_1,v_2)$) which has volume $\leq\varepsilon^{d-1}\cdot|\log\varepsilon|^{C^*}$.
\item\label{it.intmini_extra_2} In (\ref{it.intmini_4_extra}) above, suppose $\Mb$ is a \{4\}-molecule of one C-atom, and assume that $Q$ is supported in the set where for \emph{both vector pairs $((x_1,v_1),(x_2,v_2))$ and $((x_3,v_3),(x_4,v_4))$}, there exist at least two different values $t_1$ such that the equality (\ref{eq.intmini_1+}) holds. Then we can restrict the domain of integration in (\ref{eq.intmini_6}) to a set of $(t_1,\omega,x_1,v_1,v_2)$) which has volume $\leq\varepsilon^{d}\cdot|\log\varepsilon|^{C^*}$.
\item\label{it.intmini_extra_3} In (\ref{it.intmini_2}) above, suppose we do not assume (\ref{eq.intmini_new}). Instead, we assume that $Q$ is supported in the set (depending on some external parameters $(x^*,v^*)$)\begin{equation}\label{eq.newprop_3_1}\inf_{|t|\leq |\log\varepsilon|^{C^*}}|x_j-x^*-t(v_j-v^*)|_\Tb\leq\mu\end{equation}
for some $\varepsilon\lesssim\mu\lesssim 1$ and $j\neq 1$, and $\max(|x_1-x^*|_\Tb,|v_1-v^*|)\gtrsim\mu'$ for some $\mu\lesssim\mu'\lesssim 1$. Then the domain of integration in (\ref{eq.intmini_3}) can be restricted to a set of $(t_1,\omega,v_2)$ that depends on $(x_1,v_1)$ and the external parameters, which has volume $\leq \mu\cdot(\mu')^{-1}\cdot|\log\varepsilon|^{C^*}$.
\end{enumerate}
\end{proposition}
\begin{proof} (1) By \eqref{eq.intmini}, we have
\begin{equation}\label{eq.proof_intmini_1_1}
    \Jc(\Mb):=\int_{\Tb^{2d}\times\Rb^{2d}\times\Rb}\boldsymbol{\Delta}(z_1,z_2,z_3,z_4,t_1)\cdot Q\,\mathrm{d}z_3\mathrm{d}z_4\mathrm{d}t_1.
\end{equation}
By fixing some $\Rb^d$ representative of $x_j\in\Tb^d\,(1\leq j\leq 2)$ we may assume $x_j\in\Rb^d$, then we have
\begin{equation}\label{eq.proof_intmini_1_2}
    \boldsymbol{\delta}\big(|x_1-x_2+t_1(v_1-v_2)|_\Tb-\varepsilon\big) = \sum_{m\in\Zb^d} \boldsymbol{\delta}\big(|x_1-x_2+t_1(v_1-v_2)-m|-\varepsilon\big),
\end{equation}
where $|\cdot|$ is the norm in $\Rb^d$. Inserting into \eqref{eq.associated_dist_C}, we get
\begin{equation}\label{eq.proof_intmini_1_3}
    \boldsymbol{\Delta}(z_1,z_2,z_3,z_4,t_1)=\sum_{m\in\Zb^d} \boldsymbol{\Delta}_{m}(z_1,z_2,z_3,z_4,t_1),
\end{equation} 
\begin{equation}\label{eq.proof_intmini_1_4}
\begin{aligned}
    \boldsymbol{\Delta}_m(z_1,z_2,z_3,z_4,t_1)&:=\boldsymbol{\delta}\big(x_3-x_1+t_1(v_3-v_1)\big)\cdot \boldsymbol{\delta}\big(x_4-x_2+t_1(v_4-v_2)\big)\\&\times \boldsymbol{\delta}\big(|x_1-x_2+t_1(v_1-v_2) - m|-\varepsilon\big)\cdot \big[(v_1-v_2)\cdot \omega\big]_-\\&\times\boldsymbol{\delta}\big(v_3-v_1+[(v_1-v_2)\cdot\omega]\omega\big)\cdot\boldsymbol{\delta}\big(v_4-v_2-[(v_1-v_2)\cdot\omega]\omega\big).
    \end{aligned}
\end{equation} 

Therefore, we have the same decomposition for $\Jc(\Mb)$:
\begin{equation}\label{eq.proof_intmini_1_5}
    \Jc(\Mb)=\sum_{|m|\leq |\log\varepsilon|^{C^*}}\Jc_m(\Mb),
\end{equation}
where $|m|\leq |\log\varepsilon|^{C^*}$ follows from the bound of $x_j$ and $v_j$ in the integral $\Jc(\Mb)$, and $\Jc_m(\Mb)$ is given by 
\begin{equation}\label{eq.proof_intmini_1_6}
    \Jc_m(\Mb):=\int_{\Tb^{2d}\times\Rb^{2d}\times\Rb}\boldsymbol{\Delta}_m(z_1,z_2,z_3,z_4,t_1)\cdot Q\,\mathrm{d}z_3\mathrm{d}z_4\mathrm{d}t_1.
\end{equation}

Now note that $\Jc_m(\Mb)$ is exactly the same as the $\Jc(\Mb)$ in \cite{DHM24}, up to translation by $m$, see equations (7.1) and (9.1) in \cite{DHM24}. By Proposition 9.1 (1) in \cite{DHM24}, we get
\begin{equation}\label{eq.proof_intmini_1_7}
    \Jc_m(\Mb)=\mathbbm{1}_{\mathrm{col}}\left(z_1, z_2\right) \cdot Q=\sum_{t_{1m}}\mathbbm{1}_{|x_1-x_2+t_{1m}(v_1-v_2) - m|=\varepsilon\,\wedge\,(v_1-v_2)\cdot\omega\leq 0}\cdot Q,
\end{equation} which then implies (\ref{eq.intmini_1})--(\ref{eq.extradecomp}) as desired. This proves (1).

(2) Note that $z_1=(x_1,v_1)$ is fixed, and $|x_1-x_2+t_1(v_1-v_2)|_\Tb=\varepsilon$, so by our convention (see Section \ref{sec.hard_sphere}) we can identify the vector $x_1-x_2+t_1(v_1-v_2)$ with its unique $\Rb^d$ representative that has Euclidean length $\varepsilon$. Then we can introduce the variable $\omega\in \Sb^{d-1}$ such that $x_1-x_2+t_1(v_1-v_2)=\varepsilon\omega$, and make the change of variable $x_2\leftrightarrow\omega$, just as in the proof of Proposition 9.1 (2) in \cite{DHM24}. We carefully note that this substitution changes the integral in $x_2\in\Tb^d$ to the integral in $\omega\in\Sb^{d-1}$, \emph{without introducing the fundamental domain decomposition (\ref{eq.extradecomp}).} The rest of the proof then follows the same arguments as in Proposition 9.1 (2) in \cite{DHM24}, which leads to (\ref{eq.intmini_3})--(\ref{eq.intmini_4}).

The proof of the extra volume bound $\leq \lambda\cdot|\log\varepsilon|^{C^*}$ under the assumption (\ref{eq.intmini_new}) also follows from similar arguments as in Proposition 9.1 (2) in \cite{DHM24}; if $Q$ is supported in $|t_1-t^*|\leq\lambda$ or $|v_j-v^*|\leq\lambda$ or $|v_i-v_j|\leq\lambda$, then we already gain this factor $\lambda$ from the $t_1$ or $v_2$ integral in (\ref{eq.intmini_3}); if $Q$ is supported in $|x_j-x^*|_\Tb\leq\lambda$ (or $|x_i-x_j|_\Tb\leq\lambda$ which is similar) then we get $|(x_1-x^*)+t_1(v_1-v_j)|_\Tb\lesssim\lambda$. By making the fundamental domain decomposition \eqref{eq.extradecomp} and dyadic decompositions for $|t_1|$ (or $|v_1-v_j|$), and gaining from both the $t_1$ and $v_2$ integrals in (\ref{eq.intmini_3}), we get the desired volume bound $\leq \lambda\cdot|\log\varepsilon|^{C^*}$.

(3) This follows from the proof of (2) by adding the extra integrations in $x_1$ and $v_1$.

(4) The two conclusions here are related to \eqref{it.intmini_2} and \eqref{it.intmini_3} respectively. We only prove the one related to \eqref{it.intmini_2}, as the other one follows by adding the extra integrations in $x_1$ and $v_1$.

By \eqref{eq.intmini_3}, we get
\begin{equation}\label{eq.proof_intmini_2_1}
    \Jc(\Mb)=\int_{\Rb\times\Sb^{d-1}\times\Rb^d}\big[(v_1-v_2)\cdot\omega\big]_-\cdot Q\,\mathrm{d}t_1\mathrm{d}\omega\mathrm{d}v_2,
\end{equation} and we need to show that $(t_1,\omega,v_2)$ belongs to a set of volume $\leq\varepsilon^{d-1}\cdot|\log\varepsilon|^{C^*}$. The assumption on the support of $Q$ implies that in the set of $(t_1,\omega,v_2)$, there must exist $t_1\neq t_2$ such that for some $j\in\{1,3\}$, we have
\begin{equation}\label{eq.proof_intmini_2_2}
|x_j-x_{j+1}+t_i(v_j-v_{j+1})-m_i|=\varepsilon,\quad \forall i\in\{1,2\}.
\end{equation} Moreover, since $t_1\neq t_2$ we must also have $m_1\neq m_2$ in (\ref{eq.proof_intmini_2_2}), as under the assumption $(v_1-v_2)\cdot\omega\leq 0$ (or $(v_3-v_4)\cdot\omega\geq 0$), there exists at most one solution $t_i$ to (\ref{eq.proof_intmini_2_2}) with fixed $m_i$.

Now, from (\ref{eq.proof_intmini_2_2}) we deduce that
\begin{equation}\label{eq.proof_intmini_2_3}
|(t_1-t_2)(v_j-v_{j+1})-(m_1-m_2)|\leq 2\varepsilon\Rightarrow |(v_j-v_{j+1})\times (m_1-m_2)|\leq \varepsilon|\log\varepsilon|^{C^*}.
\end{equation} Note that $m_1-m_2$ is a nonzero integer vector (which we may fix up to a $|\log\varepsilon|^{C^*}$ loss). If $j=1$ then (\ref{eq.proof_intmini_2_3}) obviously restricts $v_1-v_2$ (and thus restricts $v_2$ with fixed $v_1$) to a set of volume $\leq\varepsilon^{d-1}\cdot|\log\varepsilon|^{C^*}$. If $j=3$, then $v_3-v_4$ is restricted to a set of volume $\leq\varepsilon^{d-1}\cdot|\log\varepsilon|^{C^*}$, and so is $v_1-v_2$ (with fixed $\omega$) as $v_3-v_4$ is the reflection of $v_1-v_2$ with respect to the orthogonal plane of $\omega$, and this reflection preserves volume. This allows us to bound the volume of the set of $v_2$ by $\varepsilon^{d-1}\cdot|\log\varepsilon|^{C^*}$ for each fixed $(t_1,\omega)$, which proves (4).

(5) By the same arguments as in (4), from the assumptions on the support of $Q$ we obtain that
\begin{equation}\label{eq.proof_intmini_3_1}
|(v_1-v_2)\times m'|\leq \varepsilon|\log\varepsilon|^{C^*},\qquad |(v_3-v_4)\times m''|\leq \varepsilon|\log\varepsilon|^{C^*}
\end{equation} for some nonzero integer vectors $m'$ and $m''$. From the inequality for $v_1-v_2$, we already know that $v_1-v_2$ (and hence $v_2$ with fixed $v_1$) belongs to a set of volume $\leq\varepsilon^{d-1}\cdot|\log\varepsilon|^{C^*}$. Moreover, using the inequality for $v_3-v_4$, and the fact that $v_3-v_4=\Rc_\omega(v_1-v_2)$ is the reflection of $v_1-v_2$ with respect to the orthogonal plane of $\omega$, we get that
\begin{equation}\label{eq.proof_intmini_3_2}
   |\Rc_\omega m''\times (v_1-v_2)|\le\varepsilon|\log\varepsilon|^{C^*}.
\end{equation} Note also that $|v_1-v_2|\geq |\log\varepsilon|^{-1}$ by the first inequality in (\ref{eq.proof_intmini_2_3}) (with $j=1$), we conclude that for fixed values of $m''$ and $v_1-v_2$, the condition (\ref{eq.proof_intmini_3_2}) restricts $\omega$ to a subset of $\Sb^{d-1}$ of volume $\leq\varepsilon\cdot|\log\varepsilon|^{C^*}$. In fact, let $u_1$ and $u_2$ be the unit vector in the directions of $m''$ and $v_1-v_2$ respectively, then (\ref{eq.proof_intmini_2_3}) implies that $|\Rc_\omega u_1-u_2|\leq \varepsilon|\log\varepsilon|^{C^*}$, and thus $|\omega\cdot (u_1+u_2)|\leq \varepsilon|\log\varepsilon|^{C^*}$ and $|\eta\cdot (u_1-u_2)|\leq \varepsilon|\log\varepsilon|^{C^*}$ for any vector $\eta\perp\omega$. The desired volume bound then follows from the first inequality if $|u_1+u_2|\geq 1$ and from the second inequality if $|u_1-u_2|\geq 1$.

By putting together the volume bound for $v_1-v_2$ which is independent of $\omega$, and the volume bound for $\omega$ with fixed $v_1-v_2$, we have proved (5).

(6) We know $j\in\{2,3,4\}$ and $Q$ is supported in the set where (\ref{eq.newprop_3_1}) holds. From (\ref{eq.newprop_3_1}), and using also that $|(x_j+t_1v_j)-(x_1+t_1v_1)|_\Tb=O(\varepsilon)$ due to the collision, it follows that, after shifting the fundamental domain by some $m\in\Zb^d$ which we may fix at a loss of $|\log\varepsilon|^{C^*}$, we get (note that $\varepsilon\lesssim\mu$)
\begin{equation}\label{eq.proof_intmini_4_3}
  |(x_j-x^*)\times(v_j-v^*)|\leq \mu|\log\varepsilon|^{C^*}|v_j-v^*|\Rightarrow  |(x_1-x^*+t_1(v_1-v_j))\times(v_j-v^*)|\leq \mu|\log\varepsilon|^{C^*}|v_j-v^*|.
\end{equation} Let $v_j-v_1:=u_j$, then (\ref{eq.proof_intmini_4_3}) implies that
\begin{equation}\label{eq.proof_intmini_4_4}
   |w\times u_j-\boldsymbol{p}|\leq \mu|\log\varepsilon|^{C^*}|v_j-v^*|\leq \mu|\log\varepsilon|^{C^*},
 \end{equation} where $w:=x_1-x^*+t_1(v_1-v^*)$ and $\boldsymbol{p}$ is a constant 2-form depending only on $(x_1,v_1,x^*,v^*)$. By a dyadic decomposition we may assume $|w|\sim\mu''$ for some dyadic $\mu''\in[\mu,\mu']$ (or $|w|\lesssim\mu$ if $\mu''=\mu$, or $|w|\gtrsim\mu'$ if $\mu''=\mu'$). Note that for fixed $w$, the inequality (\ref{eq.proof_intmini_4_4}) restricts $u_j$ to a tube in $\Rb^d$ with size $\leq|\log\varepsilon|^{C^*}$ in one direction and size $\mu(\mu'')^{-1}|\log\varepsilon|^{C^*}$ in all other $(d-1)$ directions; let this tube by $\Xc$. Denote also $v_2-v_1:=u$, then $u_j\in\{u,(u\cdot\omega)\omega,u-(u\cdot\omega)\omega\}$ (corresponding to $j\in\{2,3,4\}$ respectively).

 We first claim that: for fixed $t_1$ (hence fixed $w$), the condition $u_j\in\Xc$ restricts $(u,\omega)$ to a subset of $\Rb^d\times\Sb^{d-1}$ with measure $\leq \mu(\mu'')^{-1}|\log\varepsilon|^{C^*}$ (of course, with $v_1$ fixed, $(v_2,\omega)$ is then restricted to a set of the same measure). In fact, if $j=2$ this is trivial using the volume of $\Xc$. If $j\in\{3,4\}$, for fixed $\omega$ we can represent $u\in\Rb^d$ by the new coordinates $u_\omega:=u\cdot\omega\in\Rb$ and $u^\perp:=u-u_\omega\cdot\omega\in\Pi_\omega^\perp$ (where $\Pi_\omega^\perp$ is the plane orthogonal to $\omega$). Denote also $u^\circ:=u_\omega\cdot\omega$. If $j=3$, then by a simple Jacobian calculation, we see that the restriction $u^\circ\in\Xc$ implies that $(u_\omega,\omega)$ belongs to a set whose measure is bounded by
 \[\int_\Xc|u^\circ|^{-(d-1)}\,\mathrm{d}u^\circ\lesssim\mu(\mu'')^{-1}|\log\varepsilon|^{C^*}\] with the verification of the last inequality being straightforward (the choice of $u^\perp$ only increases the volume by a factor $|\log\varepsilon|^{C^*}$). Finally, if $j=4$ and dimension $d=2$, then the same argument applies with $\omega$ replaced by $\omega'$ which is the $\pi/2$ rotation of $\omega$; if $j=4$ and dimension $d=3$, we may first fix $\omega$ and $u_\omega$, then $u^\perp\in \Pi_\omega^\perp\cap \Xc$, and it is easy to see that the measure of the two dimensional set $\Pi_\omega^\perp\cap \Xc\subset\Pi_\omega$ is bounded by $\mu(\mu'')^{-1}|\log\varepsilon|^{C^*}$. This proves the first claim in either case.

 We next claim that: for fixed $\mu''$, the condition $|w|\sim\mu''$ restricts $t_1$ to a subset of $\Rb$ with measure $\leq \mu''(\mu')^{-1}\log\varepsilon|^{C^*}$. In fact, recall $w=x_1-x^*+t_1(v_1-v^*)$. if $|v_1-v^*|\geq \mu'|\log\varepsilon|^{-1}$, then the desired bound follows by solving a linear inequality in $t_1$; if $|v_1-v^*|\leq \mu'|\log\varepsilon|^{-1}$. then by the assumption $\max(|x_1-x^*|,|v_1-v^*|)\gtrsim \mu'$, we know that $|x_1-x^*|\gtrsim\mu'$, and hence $\mu''\sim|w|\gtrsim\mu'$, in which case the desired bound becomes obvious (the set of $t_1$ is trivially bounded by $|\log\varepsilon|^{C^*}$). This proves the second claim.

 By putting the above two claims together (and summing over $\mu''$ which loses at most a logarithm and can be absorbed by $|\log\varepsilon|^{C^*}$), we then conclude that $(t_1,\omega,v_2)$ belongs to a set of volume
 \[\leq \sum_{\mu'\leq\mu''\leq\mu}\mu(\mu'')^{-1}\cdot|\log\varepsilon|^{C^*}\cdot \mu''(\mu')^{-1}\cdot|\log\varepsilon|^{C^*}=\mu(\mu')^{-1}\cdot|\log\varepsilon|^{C^*},\] which proves (6). 
\end{proof}
\begin{proposition}\label{prop.intmini_2} Consider the same setting as in Proposition \ref{prop.intmini}, but now $\Mb$ is an elementary molecule with two atoms $(\nf_1,\nf_2)$, where $\nf_1$ is a parent of $\nf_2$, which satisfies one of the following assumptions:
\begin{enumerate}[{(a)}]
\item\label{it.intmini_4} Suppose $\Mb$ is a \{33A\}-molecule with atoms $(\nf_1,\nf_2)$. Let $(e_1,e_7)$ be the two fixed ends at $\nf_1$ and $\nf_2$ respectively. Assume that either (i) $\ell[\nf_1]\neq\ell[\nf_2]$ for the two atoms $\nf_1$ and $\nf_2$, or (ii) $Q$ is supported in the set
\begin{equation}\label{eq.intmini_8-}
\max(|x_1-x_7|_\Tb,|v_1-v_7|)\geq \lambda\quad \mathrm{or}\quad |t_1-t_2|\geq\mu
\end{equation} for some $\epsilon\lesssim\lambda,\mu\lesssim 1$, where $(x_j,v_j)=z_{e_j}$ as above.
\item\label{it.intmini_5} Suppose $\Mb$ is a \{33B\}-molecule. Assume that $Q$ is supported in the set
\begin{equation}\label{eq.intmini_9-}
|t_1-t_2|\geq \lambda,\quad |v_i-v_j|\geq \lambda\quad\textrm{for\ any\ edges\ }e_i\neq e_j\textrm{\ at\ the\ same\ atom},
\end{equation}for some $\varepsilon\lesssim\lambda\lesssim 1$, where again $(x_j,v_j)=z_{e_j}$.
\item\label{it.intmini_6} Suppose $\Mb$ is a \{44\}-molecule with atoms $(\nf_1,\nf_2)$ and $(e_1,e_7)$ being two free ends at $\nf_1$ and $\nf_2$ respectively, such that $\Mb$ becomes a \{33A\}-molecule after turning these two free ends into fixed ends. Moreover assume that $Q$ is supported in the set (where $\varepsilon\lesssim\lambda\lesssim1$)
\begin{equation}\label{it.intmini_10-}
\max(|x_1-x_7|_\Tb,|v_1-v_7|)\leq \lambda.
\end{equation}
\end{enumerate}
Then, in any of the above cases, we can decompose (similar to (\ref{eq.extradecomp})) that $\Jc(\Mb)=\sum_m \Jc_m(\Mb)$, such that for each $m$, we have
\begin{equation}\label{eq.intmini_8}
0\leq\Jc_m(\Mb)\leq\kappa\cdot\int_\Omega Q\,\mathrm{d}w\leq\kappa\cdot\int_{\Omega'}Q\,\mathrm{d}w.
\end{equation} Here $w\in\Omega$ and $\Omega\subseteq\Omega'$ (in fact $\Omega=\Omega'$ in cases \ref{it.intmini_5} and \ref{it.intmini_6}) are opens set in some $\Rb^p$, and $\Omega$ depends on $(x_e,v_e)$ for fixed ends $e$ while $\Omega'$ does not, and the input variables of $Q$ are explicit functions of $w$ and $(x_e,v_e)$ for fixed ends $e$. We also have $|\Omega'|\leq |\log\varepsilon|^{C^*}$ (which is true in case \ref{it.intmini_4} even without assuming (i) or (ii)). As for $\Omega$, we have the folowing estimates:
\begin{enumerate}
\item\label{it.intmini_n1} In case \ref{it.intmini_4} assuming (i), we have $|\kappa|\cdot|\Omega|\leq \varepsilon^{1/2}\cdot|\log\varepsilon|^{C^*}$;
\item\label{it.intmini_n2} In case \ref{it.intmini_4} assuming $\max(|x_1-x_7|_\Tb,|v_1-v_7|)\geq \lambda$ in (ii), we have $|\kappa|\cdot|\Omega|\leq \varepsilon\cdot\lambda^{-1}\cdot|\log\varepsilon|^{C^*}$;
\item\label{it.intmini_n3} In case \ref{it.intmini_4} assuming $|t_1-t_2|\geq\mu$ in (ii), we have $|\kappa|\cdot|\Omega|\leq \varepsilon^{d-1}\cdot\mu^{-d}\cdot|\log\varepsilon|^{C^*}$;
\item\label{it.intmini_n4} In case \ref{it.intmini_5}, we have $|\kappa|\cdot|\Omega|\leq \varepsilon^{d-1}\cdot\lambda^{-2d}\cdot|\log\varepsilon|^{C^*}$;
\item\label{it.intmini_n5} In case \ref{it.intmini_6}, we have $|\kappa|\cdot|\Omega|\leq \varepsilon^{-2(d-1)}\cdot \lambda^{2d}\cdot|\log\varepsilon|^{C^*}$.
\end{enumerate}
\end{proposition}
\begin{proof} First note that the statements other than (\ref{it.intmini_n1})--(\ref{it.intmini_n5}) are obvious. For example, the decomposition $\Jc(\Mb)=\sum_m \Jc_m(\Mb)$ can be proved by inserting \eqref{eq.proof_intmini_1_2} into the expression of $\Jc(\Mb)$. The construction of $\Omega'$ such that $|\Omega'|\leq |\log\varepsilon|^{C^*}$ is the same as the proof of Proposition 9.2 in \cite{DHM24}. Below we focus on the proof of (\ref{it.intmini_n1})--(\ref{it.intmini_n5}). Also we may fix one choice of the fundamental domain $m$, but by a suitable shift (similar to the proof of Proposition \ref{prop.intmini}) we can omit this $m$ and work on $\Rb^d$.

(1) To calculate $\Jc(\Mb)$, we first fix the values of $t_1$ and all $z_f$ for edges $f$ at $\nf_1$, and integrate in $t_2$ and all $z_f$ for all the free ends $f$ at $\nf_2$. By assumption, we know that $\nf_2$ becomes a deg 2 atom with two top fixed ends after cutting $\nf_1$ as free. Let the inner integral be $\Jc(\{\nf_2\})$, then we can apply (\ref{eq.intmini_1}) in Proposition \ref{prop.intmini} to get an explicit expression of $\Jc(\{\nf_2\})$.

Let the bond between $\nf_1$ and $\nf_2$ be $e$,  denote the edges at $\nf_1$ by $(e_1,\cdots,e_4)$ as in Proposition \ref{prop.intmini}, then $e=e_j$ for some $j\in\{2,3,4\}$. After plugging in the above formula for the inner integral $\Jc(\{\nf_2\})$, we can reduce $\Jc(\Mb)$ by an integral of form $\Jc(\{\nf_1\})$ as described in (\ref{it.intmini_2}) of Proposition \ref{prop.intmini}, namely
\begin{equation}\label{eq.proof_intmini2_1_1}
    \Jc(\Mb)=\int_{\Rb\times\Sb^{d-1}\times\Rb^d}\big[(v_1-v_2)\cdot\omega\big]_-\cdot \mathbbm{1}_{\mathrm{col}}(z_j,z_7)\cdot Q\,\mathrm{d}t_1\mathrm{d}\omega\mathrm{d}v_2,
\end{equation} 
where the input variables of $Q$ are explicit functions of $(t_1,\omega,v_2)$ and $(z_1,z_7)$ using Proposition \ref{prop.intmini} and the fact that $z_j$ satisfies (\ref{eq.intmini_4}).

Note that the integral (\ref{eq.proof_intmini2_1_1}) has the same form as (\ref{eq.intmini_3}) in Proposition \ref{prop.intmini} (\ref{it.intmini_2}) and (\ref{it.intmini_extra_3}), and the indicator function $\mathbbm{1}_{\mathrm{col}}(z_j,z_7)$ restricts $(t_1,\omega,v_2)$ to a set which satisfies (\ref{eq.newprop_3_1}) in Proposition \ref{prop.intmini} \ref{it.intmini_extra_3}). with $(x^*,v^*)$ replaced by $(x_7,v_7)$. Then, by the same argument as in the proof of Proposition \ref{prop.intmini} \ref{it.intmini_extra_3}), we can define $w:=x_1-x_7+t_1(v_1-v_7)$ and assume $|w|\sim\mu''\in[\varepsilon,1]$, such that for fixed $t_1$ (hence fixed $w$) the volume of the set of $(\omega,v_2)$ is bounded by $\varepsilon(\mu'')^{-1}|\log\varepsilon|^{C^*}$.

Suppose $|t_1-t_2|\sim\nu$, by assmption (i) in case (a), we know that $|t_1-t^*|\leq |t_1-t_2|\lesssim\nu$ for some fixed value $t^*$ (which is either $(\ell[\nf_1]-1)\tau$ or $\ell[\nf_1]\tau$), so $t_1$ belongs to an interval of length $\lesssim\nu$, which leads to the first upper bound for the volume of the set of $(t_1,\omega,v_2)$, namely $\varepsilon\nu(\mu'')^{-1}|\log\varepsilon|^{C^*}$.

On the other hand, by assumption we have
\begin{equation}\label{eq.intmini_lin}|x_1-x_j+t_1(v_1-v_j)|\leq\varepsilon,\quad |x_7-x_j+t_2(v_7-v_j)|\leq\varepsilon,\quad |x_1-x_7+t_1(v_1-v_7)|=|w|\sim\mu''.\end{equation}By taking a linear combination, this implies that $|t_1-t_2|\cdot|v_7-v_j|\lesssim\mu''$, which means that $v_j$ (and hence $v_j-v_1$) belongs to a fixed ball of radius $\lesssim \mu''\nu^{-1}$. By the same proof as in Proposition \ref{prop.intmini} \ref{it.intmini_extra_3}), we know that this restricts $(\omega,v_2)$ to a set with volume $\leq \mu''\nu^{-1}\cdot |\log\varepsilon|^{C^*}$, which is the second upper bound on the volume of the set of $(t_1,\omega,v_2)$.

Summing up, we then know that $(t_1,\omega,v_2)$ is restricted to a set whose volume is bounded by
\[\min(\varepsilon\nu(\mu'')^{-1}|\log\varepsilon|^{C^*},\mu''\nu^{-1}\cdot |\log\varepsilon|^{C^*})\leq\varepsilon^{1/2}|\log\varepsilon|^{C^*},\] which proves (1).

(2) In this case, all the discussions leading to (\ref{eq.proof_intmini2_1_1}) are the same as in (1), and the inequality (\ref{eq.newprop_3_1}) is also the same as in (1), with $(x^*,v^*)$ replaced by $(x_7,v_7)$. Moreover, we also have $\max(|x_1-x_7|_\Tb,|v_1-v_7|)\gtrsim\lambda$, so by directly applying Proposition \ref{prop.intmini} (\ref{it.intmini_extra_3}), we get that $(t_1,\omega,v_2)$ is restricted to a set whose volume is bounded by $\varepsilon\lambda^{-1}|\log\varepsilon|^{C^*}$. This proves (2).

(3) We adopt the same notations for $e_j$ etc. as in (1), which again leads to (\ref{eq.proof_intmini2_1_1}) as in (1) and (2). Now using the condition $|t_1-t_2|\geq\mu$, we will perform a different change of coordinates. Define $u:=v_2-v_1$, $y=v_3-v_1=(u\cdot\omega)\omega$ and $w=v_4-v_1=u-(u\cdot\omega)\omega$, then from (\ref{eq.proof_intmini2_1_1}) we have
\begin{equation}\label{eq.33Aproofint0}
\Jc(\Mb)\leq\int_{\Rb\times \Sb^{d-1}\times \Rb^d} |y|\cdot \mathbbm 1_{\mathrm{col}}(z_j, z_7)\cdot  Q\,\mathrm{d}t_1\mathrm{d}\omega\mathrm{d}u,
\end{equation} where the values of $(z_j, z_7)$ are determined by the fixed ends and the values of $(t_1, \omega, u)$, and $\mathbbm 1_{\mathrm{col}}(z_j, z_7)$ is the indicator function that the two particles with state $z_j$ and $z_7$ collide. To prove (3), it suffices to show that 
\begin{equation}\label{eq.33Aproofint0+}
    \Jc:=\int |y|\cdot\mathbbm{1}_{\mathrm{col}}(z_j,z_7)\,\mathrm{d}t_1\mathrm{d}\omega\mathrm{d}u \le \varepsilon^{d-1}\cdot\mu^{-d}\cdot|\log\varepsilon|^{C^*},
\end{equation} where all the conditions on the support of $Q$ are imposed in the integral in (\ref{eq.33Aproofint0+}) without explicit mentioning (same below). Here, if necessary, we may perform a dyadic decomposition on the size of $|y|$ to replace it by a constant in order to match the form (\ref{eq.intmini_8}); this leads to at most logarithmic loss which can be absorbed into $|\log\varepsilon|^{C^*}$.

Note that the function $\mathbbm 1_{\mathrm{col}}(z_j, z_7)$ depends on $v_j-v_1$, and $v_j-v_1\in\{u,y,w\}$ depending on the cases of $j\in\{2,3,4\}$. By using polar coordinates in $y$ and a simple Jacobian calculation, we have
\begin{equation}\label{eq.33Aproofint1}|y|\,\mathrm{d}\omega\mathrm{d}u=|y|^{-(d-2)} \mathrm{d}_{\Pi_y^\perp}(w) \mathrm{d}y=|y|^{-(d-3)}\dirac(y\cdot w) \mathrm{d}w \mathrm{d}y,\end{equation} where $\Pi_y^\perp$ is the plane orthogonal to $y$ and $\mathrm{d}_{\Pi_y^\perp}$ is the Hausdorff measure on that plane. Below, if $j=2$ (so $v_j-v_1=u$), we shall keep the $|y|\,\mathrm{d}\omega\mathrm{d}u$ in (\ref{eq.33Aproofint0}); if $j\in\{3,4\}$ (so $v_j-v_1\in\{y,w\}$), we shall substitute $|y|\,\mathrm{d}\omega\mathrm{d}u$ by (\ref{eq.33Aproofint1}). One easy case is when $|v_j-v_1|\lesssim\mu^{-1}\varepsilon$; this implies either $|y|\lesssim\mu^{-1}\varepsilon$ or $|w|\lesssim\mu^{-1}\varepsilon$. In either case, note that dimension $d\in\{2,3\}$, we can prove (\ref{eq.33Aproofint0+}) by direct integration in $y$ and $w$, by using (\ref{eq.33Aproofint1}) and exploiting the symmetry between $y$ and $w$ if necessary.

Next we will assume $|v_j-v_1|\gg \mu^{-1}\varepsilon$ and analyze the factor $\mathbbm 1_{\mathrm{col}}(z_j, z_7)$. In the support of this factor, there exists unique $(t_2, \omega_2)\in \Rb\times \Sb^{d-1}$ such that
\begin{equation}\label{eq.33Aproofint3}
x_j+t_2v_j-(x_7+t_2v_7)=x_1-x_7+t_1(v_1-v_7)+\sigma \varepsilon \omega-(t_1-t_2) (v_j-v_7)=\varepsilon \omega_2,
\end{equation}
where $\sigma\in\{0,1\}$ depending on whether $j=3$ or $j\in\{2,4\}$, and $\omega=y/|y|$. We would like to substitute the variable $v_j-v_1$ by $(t_2,\omega_2)$; indeed, using equations (7.1) and (7.10) in \cite{DHM24}, we have
\begin{equation}\label{eq.33Aproofint4}
\mathbbm 1_{\mathrm{col}}(z_j, z_7)=\int_{\Rb} [(v_j-v_7) \cdot \omega_2]_+ \dirac(|x_1-x_7+t_1(v_1-v_7)+\sigma \varepsilon\omega-(t_1-t_2) (v_j-v_7)|-\varepsilon)\, \mathrm{d}t_2,
\end{equation} where $\omega_2$ is as in (\ref{eq.33Aproofint3}). Now, using polar coordinates in the vector \[v_j-v_1-(t_1-t_2)^{-1}[x_1-x_7+t_2(v_1-v_7)+\sigma \varepsilon\omega]:=-\rho \omega_2,\] we can rewrite the $\dirac$ function in (\ref{eq.33Aproofint4}) as the two Dirac function $\dirac((t_1-t_2)\rho-\varepsilon)$ and $\dirac(|\omega_2|-1)$ (which leads to integration in $\omega_2\in\Sb^{d-1}$). In this way, we get the following upper bounds for (\ref{eq.33Aproofint0+}):
\begin{itemize}
\item If $j=2$, then we directly get \begin{equation}\label{333proofint5}
\Jc\lesssim \varepsilon^{d-1}\int\frac{|(v_2-v_7)\cdot \omega_2|\cdot|y|}{|t_1-t_2|^{d}} \,\mathrm{d}t_1\mathrm{d}t_2 \mathrm{d}\omega\mathrm{d}\omega_2.
\end{equation}
\item If $j=3$, then $\sigma=0$ in (\ref{eq.33Aproofint3}), in particular subsequent expressions do not depend on $\omega$; using the second expression in (\ref{eq.33Aproofint1}), we get
\begin{equation}\label{333proofint5+}
\Jc\lesssim \varepsilon^{d-1}\int\frac{|(v_3-v_7)\cdot \omega_2|}{|y|\cdot|t_1-t_2|^{d}} \,\mathrm{d}t_1\mathrm{d}t_2 \mathrm{d}\omega_2\mathrm{d}_{\Pi_y^\perp}(w),
\end{equation} where $y=v_3-v_1$ is determined by $(t_2,\omega_2)$ as above.
\item If $j=4$, then $\sigma=1$ in (\ref{eq.33Aproofint3}), which complicates things a bit. In this case we have
\begin{equation}\label{eq.333proof_ex}w=w_*+\frac{\varepsilon\omega}{t_1-t_2},\quad w_*:=\frac{x_1-x_7+t_2(v_1-v_7)-\varepsilon\omega_2}{t_1-t_2},\end{equation} and note that $|w|\sim|w_*|\gg \mu^{-1}\varepsilon$ by our assumption. Now we use the second expression in (\ref{eq.33Aproofint1}), and write
\[\dirac(y\cdot w)\,\mathrm{d}y=|w_*|^{-1}\dirac\bigg(\frac{w_*}{|w_*|}\cdot y+\lambda|y|\bigg);\quad \lambda:=\frac{\varepsilon}{(t_1-t_2)\cdot|w_*|}\,(|\lambda|\ll 1).\] Note that the $\dirac$ function is comparable to the Hausdorff measure on the cone $\frac{w_*}{|w_*|}\cdot y+\lambda|y|=0$. We denote this cone (or two rays if $d=2$) by $\Gamma_{w_*}^\perp$ and its Huausdorff measure by $\mathrm{d}_{\Gamma_{w_*}^\perp}(y)$, which is bounded on bounded sets. This leads to
\begin{equation}\label{333proofint5++}
\Jc\lesssim \varepsilon^{d-1}\int\frac{|(v_4-v_7)\cdot \omega_2|}{|w_*|\cdot|t_1-t_2|^{d}} \,\mathrm{d}t_1\mathrm{d}t_2 \mathrm{d}\omega_2\mathrm{d}_{\Gamma_{w_*}^\perp}(y).
\end{equation}
\item Note also that (\ref{333proofint5+}) and (\ref{333proofint5++}) can be written in the unified manner
\begin{equation}\label{333proofint6}
\Jc\lesssim \varepsilon^{d-1}\int\frac{|(v_j-v_7)\cdot \omega_2|}{|z|\cdot|t_1-t_2|^{d}} \,\mathrm{d}t_1\mathrm{d}t_2 \mathrm{d}\omega_2\mathrm{d}_{\widetilde{\Pi}_{z}^\perp}(z'),
\end{equation} where either $(j,z,z')=(3,y,w)$ and $\widetilde{\Pi}_y^\perp=\Pi_z^\perp$, or $(j,z,z')=(4,w_*,y)$ and $\widetilde{\Pi}_{z}^\perp=\Gamma_{w_*}^\perp$.
\end{itemize}

Now we are ready to prove (\ref{eq.33Aproofint0+}). In fact, we may fix $t_1$ and $\omega_2$. The integral in $\omega$ in (\ref{333proofint5}) and the integral in $z'$ in (\ref{333proofint6}) is trivially bounded by $|\log\varepsilon|^{C^*}$. As for the integral in $t_2$, since $|t_1-t_2|\geq\mu$, we only need to worry about the denominator $|z|$ in (\ref{333proofint6}). However the numerator $|v_j-v_7|\lesssim|z|+|v_1-v_7|$, and the $|z|$ term cancels the denominator; as for the $|v_1-v_7|$ factor, note that $w_*$ is given by (\ref{eq.333proof_ex}) (and the same for $y$ when $j=3$), so the $|v_1-v_7|$ cancels the coefficient before $t_2$ and yields
\[\int_{|z|\gtrsim\varepsilon}|z|^{-1}\,\mathrm{d}t_2\leq |\log\varepsilon|\] (for example, we may perform a dyadic decomposition in $|z|$ and use that the measure for the set of $t_2$ satisfying $|z|\sim\nu$ is bounded by $\min(1,\nu|v_1-v_7|^{-1})$). This proves (\ref{eq.33Aproofint0+}).

(4) This is the same as Proposition 9.2 (2) in \cite{DHM24}. The proof is tedious and not much related to the rest of this paper, so we omit it here and refer the reader to \cite{DHM24}.

(5) Note that, if we fix the variables $(x_1,v_1,x_7,v_7)$, then $\Jc(\Mb)$ essentially reduces to the same integral expression but for a \{33A\}-molecule. Therefore, by adding the extra integration over $(x_1,v_1,x_7,v_7)$ in \eqref{eq.proof_intmini2_1_1}, we get
\begin{equation}
    \Jc(\Mb)=\varepsilon^{-2(d-1)}\int \mathrm{d}x_1\mathrm{d}x_7\mathrm{d}v_1\mathrm{d}v_7\int_{\Rb\times\Sb^{d-1}\times\Rb^d}\big[(v_1-v_2)\cdot\omega\big]_-\cdot \mathbbm{1}_{\mathrm{col}}(z_j,z_7)\cdot Q\,\mathrm{d}t_1\mathrm{d}\omega\mathrm{d}v_2.
\end{equation} 
According to the support assmption of $Q$ in \eqref{it.intmini_10-}, we know that 
\begin{equation*}
    \Jc(\Mb)=\varepsilon^{-2(d-1)}\int \mathbbm{1}_{\max(|x_1-x_7|_\Tb,|v_1-v_7|)\leq \lambda} \mathrm{d}x_1\mathrm{d}x_7\mathrm{d}v_1\mathrm{d}v_7\int_{\Rb\times\Sb^{d-1}\times\Rb^d}\big[(v_1-v_2)\cdot\omega\big]_-\cdot \mathbbm{1}_{\mathrm{col}}(z_j,z_7)\cdot Q\,\mathrm{d}t_1\mathrm{d}\omega\mathrm{d}v_2.
\end{equation*} This then proves (5), by noticing that $\kappa=\varepsilon^{-2(d-1)}$ and $\mathrm{vol}\, (\{\max(|x_1-x_7|_\Tb,|v_1-v_7|)\leq \lambda\})\le \lambda^{2d}$.
\end{proof}
\begin{proposition}\label{prop.intmini_3} Consider the same setting as in Proposition \ref{prop.intmini}, but now $\Mb$ is an elementary molecule with three atoms $(\nf_1,\nf_2,\nf_3)$ as in Definition \ref{def.elementary}, where $\nf_1$ is a parent of $\nf_2$, which satisfies one of the following assumptions:
\begin{enumerate}
\item\label{it.intmini_3_1} Suppose $\Mb$ is a \{333A\}-molecule with atoms $(\nf_1,\nf_2,\nf_3)$. Let $t_j=t_{\nf_j}$ and $(e_1,e_7,e_{10})$ be the three fixed ends at $(\nf_1,\nf_2,\nf_3)$ respectively. Assume $Q$ is supported in the set that
\begin{equation}\label{eq.intmini_3_1_1}|t_1-t_2|\gtrsim\mu,\quad \min_{i\neq j\in\{1,7,10\}}(|x_i-x_j|_\Tb,|v_i-v_j|)\gtrsim\mu'
\end{equation} for some $\varepsilon\lesssim\mu,\mu'\lesssim 1$, and that
\begin{equation}\label{eq.intmini_3_1_2}\inf_{|t|\leq |\log\varepsilon|^{C^*}}|x_1-x_7+t(v_1-v_7)|_\Tb\geq\mu'\end{equation} for the same $\mu'$ as above, where $(x_j,v_j)=z_{e_j}$.
\item\label{it.intmini_3_2} Suppose $\Mb$ is a \{334T\}-molecule with atoms $(\nf_1,\nf_2,\nf_3)$. Let $t_j=t_{\nf_j}$ and $(e_1,e_7)$ be the two fixed ends at $\nf_1$ and $\nf_2$ respectively. Assume $Q$ is supported in the set that
\begin{equation}\label{eq.intmini_3_2_1}|t_1-t_2|\gtrsim\mu,\quad \min(|x_1-x_7|_\Tb,|v_1-v_7|)\gtrsim\mu'
\end{equation} for some $\varepsilon\lesssim\mu,\mu'\lesssim 1$, and that \begin{equation}\label{eq.intmini_3_2_2}|x_1-x_7+t_2(v_1-v_7)|_\Tb\geq\mu'\end{equation} for the same $\mu'$ as above, where $(x_j,v_j)=z_{e_j}$.
\end{enumerate}
Then in each case, we can decompose (similar to (\ref{eq.extradecomp})) that $\Jc(\Mb)=\sum_m \Jc_m(\Mb)$, such that for each $m$, we have
\begin{equation}\label{eq.333_2}
0\leq\Jc_m(\Mb)\leq\kappa\cdot\int_\Omega Q\,\mathrm{d}w\leq\kappa\cdot\int_{\Omega'}Q\,\mathrm{d}w.
\end{equation} Here the $\kappa$ and $(\Omega,\Omega')$ are as in Proposition \ref{prop.intmini_2} (so $\Omega$ depends on $(x_e,v_e)$ for fixed ends $e$ while $\Omega'$ does not, etc.), except that the estimate for $\Omega$ should be replaced by $|\kappa|\cdot|\Omega|\leq \varepsilon^{d-1/2}\cdot (\mu\cdot\mu')^{-2d}\cdot|\log\varepsilon|^{C^*}$.
\end{proposition}
\begin{proof} For \{333A\}-molecules, at the expense of a factor of at most $|\log \varepsilon|^{C^*}$, we can bound each single $\Jc_m(\Mb)$ individually, which allows us to lift to the Euclidean covering and assume that the spatial domain is $\Rb^d$. For \{334T\}-molecules, we can still do this for the two atoms $\{\nf_1,\nf_2\}$ which forms the \{33A\}-molecule. We adopt the same convention regarding $e_j$ etc. as in Proposition \ref{prop.intmini_2}; in particular we assume $(x_j,v_j)$ has collision with $(x_7,v_7)$ at atom $\nf_2$ with $j\in\{3,4\}$. Moreover, we can distinguish three cases:
\begin{enumerate}[{(a)}]
\item When $\Mb$ is \{333A\}-molecule with $\nf_3$ adjacent with $\nf_2$; in this case we define $e_k$ (with $k\in\{8,9\}$ and and associated $z_{e_k}=(x_k,v_k)$) to be the bond between $\nf_2$ and $\nf_3$, where $e_8$ is serial with $e_j$ and $e_9$ is not.
\item When $\Mb$ is \{334T\}-molecule; in this case we define $e_k$ (with $k\in\{8,9\}$) as in (a), and define $e_\ell$ (with $\ell\in\{3,4\}$) to be the bond between $\nf_1$ and $\nf_3$. Note that $\{j,\ell\}=\{3,4\}$.
\item When $\Mb$ is \{333A\}-molecule with $\nf_3$ adjacent with $\nf_1$; in this case we define $e_\ell$ (with $\ell\in\{3,4\}$) as in (b).
\end{enumerate}

Arguing as in the proof of Proposition \ref{prop.intmini_2} (3), we only need to prove that
\begin{equation}\label{eq.334proofint0}
    \Jc:=\int |y|\cdot\mathbbm{1}_{\mathrm{col}}(z_j,z_7)\cdot\mathbbm{1}_{\mathrm{col}}(z_*,z_*)\,\mathrm{d}t_1\mathrm{d}\omega\mathrm{d}u \le \varepsilon^{d-1/2}\cdot(\mu\cdot\mu')^{-2d}\cdot|\log\varepsilon|^{C^*},
\end{equation} where $(u,y,w)=(v_2-v_1,v_3-v_1,v_4-v_1)$. The second indicator function equals $\mathbbm{1}_{\mathrm{col}}(z_k,z_{10})$ in case (a), and $\mathbbm{1}_{\mathrm{col}}(z_k,z_\ell)$ in case (b), and $\mathbbm{1}_{\mathrm{col}}(z_\ell,z_{10})$ in case (c). We may assume $\mu\cdot\mu'\gg\varepsilon^{1/2}$, otherwise (\ref{eq.334proofint0}) is trivial. Note that we have $|x_1-x_7+t_2(v_1-v_7)|\geq\mu'$ by either (\ref{eq.intmini_3_1_2}) or (\ref{eq.intmini_3_2_2}); using also (\ref{eq.33Aproofint3}), this implies that $|v_j-v_1|\gtrsim\mu'\gg\mu^{-1}\varepsilon$. Now, following the same arguments as in the proof of Proposition \ref{prop.intmini_2} (3) (i.e. analyzing the \{33A\}-molecule formed by $\nf_1$ and $\nf_2$), we get the following expression (corresponding to (\ref{333proofint6})):
\begin{equation}\label{334proofint1}
\Jc\lesssim \varepsilon^{d-1}\int\frac{|(v_j-v_7)\cdot \omega_2|}{|z|\cdot|t_1-t_2|^{d}} \cdot\mathbbm{1}_{\mathrm{col}}(z_*,z_*)\,\mathrm{d}t_1\mathrm{d}t_2 \mathrm{d}\omega_2\mathrm{d}_{\widetilde{\Pi}_{z}^\perp}(z'),
\end{equation} where either $(j,z,z')=(3,y,w)$ and $\widetilde{\Pi}_y^\perp=\Pi_z^\perp$, or $(j,z,z')=(4,w_*,y)$ and $\widetilde{\Pi}_{z}^\perp=\Gamma_{w_*}^\perp$, with the relvant notations same as in the proof of Proposition \ref{prop.intmini_2} (3). Now we consider the three different cases.

\begin{itemize}

\item \emph{Case (a) (\{333A\}-molecule with $\nf_3$ connected to $\nf_2$)}. In this case we have that $z_{10}=(x_{10},v_{10})$ is fixed. We will fix $t_1$ in (\ref{334proofint1}) (i.e. integrate it after all the other variables) and control the measure of the set of $(t_2,\omega_2)$ (note that $|z|^{-1}\lesssim(\mu')^{-1}$, also the $z'$ integral is uniformly bounded). The factor $1_{\mathrm{col}}(z_k, z_{10})$ implies that there exists $t_3\in\Rb$ such that
\begin{equation}\label{coll3conc}
|x_k+t_3 v_k -x_{10}-t_3 v_{10}|=\varepsilon,
\end{equation}
which combined with the fact that $x_k+t_2v_k=x_7+t_2v_7+ O(\varepsilon)$ (due to the collision $\nf_2$)
gives that 
\begin{equation}\label{v3v5cross}
|(x_7+t_2 v_7-x_{10}-t_2v_{10})+(t_2-t_3)(v_{10}-v_k)|\leq 2\varepsilon. 
\end{equation}
Denote $X:=x_7+t_2 v_7-x_{10}-t_2v_{10}$, then we have $|X\times (v_{10}-v_k)|=O(\varepsilon)$. Note also that \[v_k-v_7\in\{v_j-v_7-((v_j-v_7)\cdot\omega_2)\omega_2,((v_j-v_7)\cdot\omega_2)\omega_2\}\] corresponding to $k\in\{8,9\}$, and that $v_j-v_7=z+v_1-v_7+O(\varepsilon)$. This implies that $v_{10}-v_k=v_{10}-v_7-z''$ up to $O(\varepsilon)$, where
\begin{equation}\label{eq.whatever}z''\in\{z+v_1-v_7-((z+v_1-v_7)\cdot\omega_2)\omega_2,((z+v_1-v_7)\cdot\omega_2)\omega_2\}.\end{equation} Note that by (\ref{eq.33Aproofint3}) and (\ref{eq.intmini_3_1_2}) we have \[|(t_1-t_2)(z+v_1-v_7)|\geq |x_1-v_7+t_1(v_1-v_7)+O(\varepsilon)|\gtrsim\mu'.\]

Let $|X|\sim\sigma$ for some dyadic $\sigma$, then $t_2$ belongs to an interval of length $\min(1,\sigma(\mu')^{-1})|\log\varepsilon|$ since $|v_7-v_{10}|\gtrsim\mu'$. Moreover, with $t_2$ fixed, by the inequality $|X\times (v_{10}-v_k)|=O(\varepsilon)$, we know the $z''$ given by (\ref{eq.whatever}) belongs to a tube $\Xc$ of size $|\log\varepsilon|^{C^*}$ in one dimension and $\min(1,\varepsilon\sigma^{-1})$ in all other $(d-1)$ dimensions. With $t_1$ and $t_2$ fixed, the value of $z$ is also fixed (up to perturbation $O(\varepsilon)$) as in (\ref{eq.333proof_ex}). Note that the $z''$ defined by (\ref{eq.whatever}) belongs to a sphere of radius $\rho=|z+v_1-v_7|/2\gtrsim \mu'|\log\varepsilon|^{-1}$ which is parametrized by $\omega_2$. The intersection of the tube $\Xc$ with this sphere equals a subset of the sphere with diameter $O(\min(1,\varepsilon^{1/2}\sigma^{-1/2}\rho^{1/2}))$, which implies that $\omega_2$ belongs to a set of measure $\lesssim\min(1,\varepsilon^{1/2}\sigma^{-1/2}(\mu')^{-1/2})|\log\varepsilon|$. Putting together the above estimates for $t_2$ and $\omega_2$, we have proved (\ref{eq.334proofint0}). Note the order of integration here is: $z'$, then $\omega_2$, then $t_2$, then $t_1$.
\item \emph{Case (b) (\{334T\}-molecule)}. This case is similar to (a), except that $(x_k,v_k)$ collides with $(x_\ell,v_\ell)$ instead of $(x_{10},v_{10})$. We shall fix $t_2$ and control the measure of the set of $(t_1,\omega_2,z')$. Let $|z+v_1-v_7|=|v_j-v_7|\sim\nu$ for some dyadic $\nu$, then $|x_1-x_7+t_1(v_1-v_7)|_\Tb\lesssim\nu|\log\varepsilon|$, so $t_1$ belongs to an interval of length $\min(1,\nu(\mu')^{-1})|\log\varepsilon|$. Now we fix $t_1$. This also fixes $\widetilde{z}=(x_1-x_7+t_2(v_1-v_7))/(t_1-t_2)$ and we have $z=\widetilde{z}+O(\varepsilon\mu^{-1})$. Note that both $z$ and the set $\widetilde{\Pi}_z^\perp\ni z'$ depend on $\omega_2$, but with fixed $\omega_2$, the $z'\in \widetilde{\Pi}_z^\perp$ can be mapped to some $\widetilde{z}'=z'+O(\varepsilon\mu^{-1})\in\Pi_{\widetilde{z}}^\perp$ at a uniformly bounded Jacobian, and the variable $\widetilde{z}'$ belongs to a set $\Pi_{\widetilde{z}}^\perp$ which is independent of $\omega_2$, so we are allowed to integrate in $\omega_2$ with fixed $\widetilde{z}'$. Then $v_\ell=v_1+\widetilde{z}'+O(\varepsilon\mu^{-1})$ and $x_\ell$ are both fixed up to error $O(\varepsilon)$.

Now we can repeat the discussion in (a) with $(x_\ell,v_\ell)$ in place of $(x_{10},v_{10})$. Note that now $|z+v_1-v_7|=|v_j-v_7|\sim\nu$ in comparison with (a); moreover, the $X$ in (a) should be replaced by $\widetilde{X}=x_7=t_2v_7-x_\ell-t_2v_\ell$ which equals $(t_1-t_2)\widetilde{z}'+X'+O(\varepsilon\mu^{-1})$ where $X'$ depends only on $t_1$ and $t_2$, so if $|X'|\sim\sigma$ as in (a) then $\widetilde{z}'$ belongs to a set of measure $\lesssim\min(1,\sigma\mu^{-1})$ with $t_1$ and $t_2$ fixed. Finally, with $(t_1,t_2,\widetilde{z}')$ fixed, the same proof in (a) yields that $\omega_2$ belongs to a set of measure $\lesssim \min(1,\varepsilon^{1/2}\sigma^{-1/2}\nu^{-1/2})|\log\varepsilon|$. Putting together the above estimates for $t_1$ and $\widetilde{z}'$ and $\omega_2$, we have proved (\ref{eq.334proofint0}). The order of integration here: $\omega_2$, then $\widetilde{z}'$ (mapped from $z'$), then $t_1$, then $t_2$.
\item \emph{Case (c) (\{333A\}-molecule with $\nf_3$ connected to $\nf_1$)}. In this case, we have that $(x_\ell,v_\ell)$ collides with $(x_{10},v_{10})$. We shall fix $\omega_2$ and control the measure of the set of $(t_1,t_2,z')$. By the same calculation as above, we get $|X\times(v_\ell-v_{10})|\lesssim\varepsilon$, where now $X:=x_1-x_{10}+t_1(v_1-v_{10})$. Let $|X|\sim\sigma$, then $t_1$ belongs to an interval of length $\min(1,\sigma(\mu')^{-1})|\log\varepsilon|$. Now we fix $t_1$ so $X$ is also fixed. Then $v_\ell$ (and hence $v_\ell-v_1=z')$ belongs to a fixed tube $\Xc$ (similar to (a) and (b) above) in the direction of $X$, which has thickness $\varepsilon\sigma^{-1}$ in the orthogonal direction of $X$. On the other hand, with $\omega_2$ fixed, we know that $z'$ belongs to $\widetilde{\Pi}_z^\perp$ which is close to $\Pi_{\widetilde{z}}^\perp$ (and is integrated according to the corresponding Hausdorff measure) where $\widetilde{z}=(x_1-x_7+t_2(v_1-v_7))/(t_1-t_2)$. Let $|(X/|X|)\times (\widetilde{z}/|\widetilde{z}|)|\sim\nu$ for some dyadic $\nu$, then with $z_2$ fixed, this forces $z'$ to belong to a set in $\widetilde{\Pi}_z^\perp$ of measure $\leq |\log\varepsilon|^{C^*}\min(1,\varepsilon\sigma^{-1}\nu^{-1})$. Finally, if $|(X/|X|)\times (\widetilde{z}/|\widetilde{z}|)|\sim\nu$, then we have
\[\big|(X/|X|)\times (x_1-x_7)+t_2(X/|X|)\times (v_1-v_7)\big|\lesssim\nu|\log\varepsilon|^2.\] It is then easy to see that $|(X/|X|)\times (v_1-v_7)|\gtrsim(\mu')^2$ (if not, then $|(X/|X|)\times (v_1-v_7)|\ll(\mu')^2$ also, hence $|(x_1-x_7)|\times(v_1-v_7)|\ll(\mu')^2$, contradicting (\ref{eq.intmini_3_1_2})), and thus with $t_1$ fixed, $t_2$ must belong to an interval of length $\min(1,\nu(\mu')^{-2})|\log\varepsilon|^{C^*}$. Putting together all the above estimates for $t_1$, $t_2$ and $z'$, we have proved (\ref{eq.334proofint0}). The order of integration here: $z'$, then $t_2$, then $t_1$, then $\omega_2$.\qedhere
\end{itemize}
\end{proof}

\subsection{Definition of excess}\label{sec.summary} Suppose $\Mb$ is reduced to $\Mb'$ that contains only elementary components, after certain cutting sequence. By the same discussion in Section 9.3 in \cite{DHM24}, we may put certain restrictions on the support of the function $Q$ occurring in (\ref{eq.associated_int}) (and also (\ref{eq.intmini})), by inserting certain indicator functions that form a partition of unity. If such a indicator function is fixed, we may then say that the integral (\ref{eq.associated_int}) (or (\ref{eq.intmini})) is \textbf{restricted} to a certain set, which is specified by indicator function.

Now we recall the definition of \textbf{good}, \textbf{normal} and \textbf{bad} components of $\Mb'$, see Definition 9.5 in \cite{DHM24}. In the current paper, we will need an extension of these definitions, as well as a more quantitative description of the gain for each good component, in the form of \textbf{excess} in Definition \ref{def.excess} below. Note that these definitions depend on both the structure of $\Mb'$, and also the restrictions (i.e. support conditions of $Q$), hence they depend on the specific choice of the indicator functions described above.
\begin{definition}
\label{def.good_normal} The following components are normal:
\begin{enumerate}
\item Any \{2\}-component.
\item Any \{3\}-component, except those described in (\ref{it.good_1}) below.
\item Any \{33A\} component, except those described in (\ref{it.good_2}) below.
\end{enumerate}

The following components are bad:
\begin{enumerate}[resume]
\item Any \{4\}-component, except those described in (\ref{it.good_3}) below. 
\end{enumerate}

The following components are good:
\begin{enumerate}[resume]
\item\label{it.good_1} Any \{3\}-component $\Xc=\{\nf\}$, which occurs in some restriction in the following sense. The integral (\ref{eq.intmini}) is restricted to the set where one of the following two assumptions hold:
\begin{equation}\label{eq.good_normal_1}\min(|v_{e_i}-v_{e_j}|,|x_{e_i}-x_{e_j}|_\Tb)\leq\varepsilon^{1/(8d)}\quad\mathrm{or}\quad \min(|t_\nf-t_{\nf'}|,|x_e-x_{e'}|_\Tb,|v_e-v_{e'}|)\leq\varepsilon^{1/(8d)};
\end{equation}
\begin{multline}\label{eq.good_normal_1+}\qquad\qquad\mathrm{for\ the\ vectors\ }z_{e_i}=(x_{e_i},v_{e_i})\mathrm{\ and\ }z_{e_j}=(x_{e_j},v_{e_j}),\mathrm{\ there\ exists\ at\ least}\\ \mathrm{two\ different\ values\ of\ }
t_1\mathrm{\ such\ that\ the\ equality\ }(\ref{eq.intmini_1+})\mathrm{\ holds}.
\end{multline}
Here in (\ref{eq.good_normal_1}) we assume that (i) $(e_i,e_j)$ are two different ends at $\nf$; (ii) $e$ is a free end at $\nf$, while $\nf'$ is an atom in a component $\Yc$ with $\Yc\prec_{\mathrm{cut}}\Xc$ in the sense of Definition \ref{def.match_ends}, and $e'$ is a free end or bond at $\nf'$. In (\ref{eq.good_normal_1+}) we assume that $(e_i,e_j)$ are two different ends at $\nf$ which are both bottom or both top. Note that if $e_i$ (or $e_j$) is the fixed end, then $z_{e_i}$ equals $z_{e_i'}$ where $e_i'$ is the free end paired with $e_i$ in the sense of Definition \ref{def.cutting}, so (\ref{eq.good_normal_1}) can still be expressed as a condition of $\vz_{\Ec_*}$; same with (\ref{it.good_2})--(\ref{it.good_3}) below.
\item\label{it.good_2} Any \{33A\}-component $\Xc=\{\nf_1,\nf_2\}$, and we assume that either (i) in Proposition \ref{prop.intmini_2} \ref{it.intmini_4} is satisfied, or the integral (\ref{eq.intmini}) is restricted to the set where
\begin{equation}\label{eq.good_normal_2}\max(|x_{e_1}-x_{e_7}|_\Tb,|v_{e_1}-v_{e_7}|)\geq\varepsilon^{1-1/(8d)}\quad\mathrm{or}\quad|t_{\nf_1}-t_{\nf_2}|\geq\varepsilon^{1/(8d)},
\end{equation} where $(e_1,e_7)$ are the two fixed ends at $\nf_1$ and $\nf_2$ respectively.
\item\label{it.good_33B} Any \{33B\}-component $\Xc=\{\nf_1,\nf_2\}$, and the integral (\ref{eq.intmini}) is restricted to the set where
\begin{equation}\label{eq.good_normal_4}|v_{e_i}-v_{e_j}|\geq \varepsilon^{1/(8d)},\quad |t_{\nf_1}-t_{\nf_2}|\geq \varepsilon^{1/(8d)},
\end{equation} for any two distinct edges $(e_i,e_j)$ at the same atom.
\item\label{it.good_44} Any \{44\}-component $\Xc=\{\nf_1,\nf_2\}$, and the integral (\ref{eq.intmini}) is restricted to the set where
\begin{equation}\label{eq.good_normal_3}\max(|x_{e_1}-x_{e_7}|_\Tb,|v_{e_1}-v_{e_7}|)\leq \varepsilon^{1-1/(8d)},
\end{equation} where $(e_1,e_7)$ are two free ends at $\nf_1$ and $\nf_2$ respectively, such that $\Mb$ becomes a \{33A\}-molecule after turning these two free ends into fixed ends.
\item\label{it.good_3} Any \{4\}-component $\Xc=\{\nf\}$, and the integral (\ref{eq.intmini}) is restricted to the set where one of the following two assumptions hold:
\begin{equation}\label{eq.good_normal_4+}\min(|v_{e_i}-v_{e_j}|,|x_{e_i}-x_{e_j}|_\Tb)\leq \varepsilon^{1/(8d)}\quad\textrm{or}\quad \min(|t_\nf-t_{\nf'}|,|x_{e}-x_{e'}|_\Tb,|v_{e}-v_{e'}|)\leq \varepsilon^{1/(8d)};
\end{equation} 
\begin{multline}\label{eq.good_normal_4++}\qquad\qquad\mathrm{for\ the\ vectors\ }z_{e_i}=(x_{e_i},v_{e_i})\mathrm{\ and\ }z_{e_j}=(x_{e_j},v_{e_j}),\mathrm{\ there\ exists\ at\ least}\\ \mathrm{two\ different\ values\ of\ }
t_1\mathrm{\ such\ that\ the\ equality\ }(\ref{eq.intmini_1+})\mathrm{\ holds}.
\end{multline}
Here in (\ref{eq.good_normal_4+}) and (\ref{eq.good_normal_4++}), we make the same assumptions on $(e,e')$ and $(e_i,e_j)$ as in (\ref{eq.good_normal_1}) and (\ref{eq.good_normal_1+}) in (\ref{it.good_1}).
\end{enumerate}
\end{definition}
\begin{definition}\label{def.excess} Let $\Mb$ be any elementary molecule as in Definition \ref{def.elementary}; note that here we are allowed to add various restrictions to the support of $Q$ as in Propositions \ref{prop.intmini}--\ref{prop.intmini_3}. Let $\Jc(\Mb)$ be as in (\ref{eq.intmini}). For dyadic number $\sigma\lesssim 1$, we say $\Mb$ has \textbf{excess} $\sigma$, if $\Jc(\Mb)$ (or each component $\Jc_m(\Mb)$ of $\Jc(\Mb)$ as in Proposition \ref{prop.intmini_2}) can be bounded by an integral of form (\ref{eq.intmini_8}), i.e. 
\begin{equation}\label{eq.excess_1}0\leq\Jc_m(\Mb)\leq\kappa\cdot\int_\Omega Q\,\mathrm{d}w\leq\kappa\cdot\int_{\Omega'}Q\,\mathrm{d}w,\end{equation} Here the $\kappa$ and $(\Omega,\Omega')$ are as in Proposition \ref{prop.intmini_2} (so $\Omega$ depends on $(x_e,v_e)$ for fixed ends $e$ and possibly on the external parameters while $\Omega'$ does not, etc.), except that the estimate for $\Omega$ should be replaced by \begin{equation}\label{eq.excess_2}|\kappa|\cdot|\Omega|\leq \sigma\cdot|\log\varepsilon|^{C^*}.\end{equation} If $\Mb$ is a good \{4\}-molecule, then the inequality (\ref{eq.excess_2}) should be replaced by (cf. (\ref{eq.intmini_6})) \begin{equation}\label{eq.excess_3}|\kappa|\cdot|\Omega|\leq \varepsilon^{-(d-1)}\cdot\sigma\cdot|\log\varepsilon|^{C^*}.\end{equation}
\end{definition}
\begin{proposition}
\label{prop.newprop_3} The following results are true:
\begin{enumerate}
\item\label{it.newprop_0} If $\Mb$ is a good component in the sense of Definition \ref{def.good_normal}, then it has excess $\varepsilon^{1/(8d)}$.
\item\label{it.newprop_1} If $\Mb$ is a \{3\}- or \{4\}-component as in Proposition \ref{prop.intmini} (\ref{it.intmini_2}) and (\ref{it.intmini_3}) that satisfies either (\ref{eq.intmini_new}) or (\ref{eq.intmini_6+}) with some $\varepsilon\lesssim\lambda\lesssim 1$, then $\Mb$ has excess $\lambda$. If $\Mb$ is a \{3\}- or \{4\}-component as in in Proposition \ref{prop.intmini} (\ref{it.intmini_4_extra}), then $\Mb$ has excess $\varepsilon^{d-1}$. If $\Mb$ is a \{4\}-component as in Proposition \ref{prop.intmini} (\ref{it.intmini_extra_2}), then $\Mb$ has excess $\varepsilon^{d}$. If $\Mb$ is a \{3\}-component as in Proposition \ref{prop.intmini} (\ref{it.intmini_extra_3}) that satisfies (\ref{eq.newprop_3_1}) and $\max(|x_1-x^*|_\Tb,|v_1-v^*|)\gtrsim\mu'$ with some $\varepsilon\lesssim\mu\lesssim\mu'\lesssim 1$, then $\Mb$ has excess $\mu\cdot(\mu')^{-1}$.
\item\label{it.newprop_2} If $\Mb$ is a \{33A\}-component as in Proposition \ref{prop.intmini_2} \ref{it.intmini_4} that satisfies $\max(|x_1-x_7|_\Tb,|v_1-v_7|)\geq \lambda$ in (\ref{eq.intmini_8-}) with some $\varepsilon\lesssim\lambda\lesssim 1$, then it has excess $\varepsilon\cdot\lambda^{-1}$. If it satisfies $|t_1-t_2|\geq\mu$ in (\ref{eq.intmini_8-}) with some $\varepsilon\lesssim\mu\lesssim 1$, then it has excess $\varepsilon^{d-1}\cdot\mu^{-d}$.
\item\label{it.newprop_3} If $\Mb$ is a \{333A\} or \{334T\}-component as in Proposition \ref{prop.intmini_3} that satisfies either (\ref{eq.intmini_3_1_1})--(\ref{eq.intmini_3_1_2}) or (\ref{eq.intmini_3_2_1})--(\ref{eq.intmini_3_2_2}) with some $\varepsilon\lesssim\mu,\mu'\lesssim 1$, then it has excess $\varepsilon^{d-1/2}\cdot (\mu\cdot\mu')^{-2d}$.
\end{enumerate}
\end{proposition}
\begin{proof} By Definition \ref{def.excess}, we see that (1) follows from Definition \ref{def.good_normal} using Propositions \ref{prop.intmini}--\ref{prop.intmini_2}, (2) follows from Proposition \ref{prop.intmini} (\ref{it.intmini_2})--(\ref{it.intmini_extra_3}), (3) follows from Proposition \ref{prop.intmini_2} (2) and (3), and (4) follows from Proposition \ref{prop.intmini_3}.
\end{proof}
\subsection{The main combinatorial proposition} Now we can finally state the main technical (combinatorial) ingredient of the current paper, namely Proposition \ref{prop.comb_est_extra} below. This proposition plays the role of Proposition 13.1 in \cite{DHM24}, and is in fact a refinement of the latter with the extra proofs needed in the torus case. It is also a major component of the proof of Theorem \ref{th.main} in Section \ref{sec.proof_main}, see \textbf{Part 6} in Section \ref{sec.proof_main}.
\begin{proposition}\label{prop.comb_est_extra} Let $\Gamma$ be a large absolute constant depending only on the dimension $d$. Let $\Mb$ be a full molecule of C-atoms with all atoms in the same layer $\ell$ (see Definition \ref{def.molecule} for layers), such that all its bonds form a tree of $\leq|\log\varepsilon|^{C^*}$ atoms plus exactly $\gamma$ bonds where $\Gamma<\gamma<2\Gamma$. 

Then there exists a cutting sequence that cuts $\Mb$ into elementary components (Definition \ref{def.elementary}), such that (after decomposing 1 into at most $C^{|\Mc|}|\log\varepsilon|^{C^*}$ indicator functions, where $\Mc$ is the set of atoms of $\Mb$) one of the followings happen:
\begin{enumerate}
\item There is at least one good \{44\}-component, and also $\#_{\{33B\}}=\#_{\{4\}}=0$.
\item All \{33B\}- and \{44\}- components are good, and 
\begin{equation}
\label{eq.extra_case_1}\frac{1}{10d}\cdot(\#_{\mathrm{good}})-d\cdot (\#_{\{4\}})\geq 100d^2.
\end{equation}
\item There is exactly one \{4\}-component, and all the other components are \{2\}-, \{3\}-, \{33A\}-, \{333A\}- and \{334T\}-components. Moreover there exists at most 10 components, each having excess $\sigma_j$ as in Definition \ref{def.excess}, such that $\prod_j\sigma_j\leq\varepsilon^{d-1+1/(15d)}$.
\end{enumerate}
\end{proposition}
\section{Proof of Proposition \ref{prop.comb_est_extra}}\label{sec.error} From now on we focus on the proof of Proposition \ref{prop.comb_est_extra}. We start by recalling the basic algorithm \textbf{UP} in \cite{DHM24} and its properties (Definition 10.1 and Proposition 10.2 in \cite{DHM24}).
\begin{definition}[The algorithm \textbf{UP}]\label{def.alg_up} Define the following cutting algorithm \textbf{UP}. It takes as input any molecule $\Mb$ of C-atoms which has no top fixed end. For any such $\Mb$, define the cutting sequence as follows:
\begin{enumerate}
\item\label{it.alg_up_1} If $\Mb$ contains any deg $2$ atom $\nf$, then cut it as free, and repeat until there is no deg $2$ atom left.
\item\label{it.alg_up_2} Choose a \emph{lowest} atom $\nf$ in the set of all deg 3 atoms in $\Mb$ (or a lowest atom in $\Mb$, if $\Mb$ only contains deg 4 atoms). Let $S_\nf$ be the set of descendants of $\nf$ (Definition \ref{def.molecule_more}).
\item \label{it.alg_up_3} Starting from $\nf$, choose a \emph{highest} atom $\mf$ in $S_\nf$ that has not been cut. If $\mf$ has deg 3, and either a parent $\mf^+$ or a child $\mf^-$ of $\mf$ (Definition \ref{def.molecule_more}) also has deg 3, then cut $\{\mf,\mf^+\}$ or $\{\mf,\mf^-\}$ as free; otherwise cut $\mf$ as free. Repeat until all atoms in $S_\nf$ have been cut. Then go to (\ref{it.alg_up_1}).
\end{enumerate}

We also define the dual algorithm \textbf{DOWN}, by reversing the notions of parent/child, lowest/highest, top/bottom etc. It applies to any molecule of C-atoms that has no bottom fixed end.
\end{definition}
\begin{proposition}
\label{prop.alg_up} Let $\Mb$ be any connected molecule as in Definition \ref{def.alg_up}. Then after applying algorithm \textbf{UP} to $\Mb$ (and same for \textbf{DOWN}), it becomes $\Mb'$ which contains only elementary components. Moreover, among these elementary components:
\begin{enumerate}
\item We have $\#_{\{33B\}}=\#_{\{44\}}=0$ and $\#_{\{4\}}\leq 1$, and $\#_{\{4\}}=1$ if and only if $\Mb$ is full. If $\Mb$ has no deg $2$ atom, and either contains a cycle or contains at least two deg $3$ atoms, then either $\#_{\{33A\}}\geq 1$ or $\Mb'$ contains a good component (in the sense of Definition \ref{def.good_normal}). If $\Mb$ has no deg 2 atom and at most one deg 3 atom, and is also a tree, then $\Mb'$ contains (at most one) \{4\}-component and others are all \{3\}-components.
\item If $\Mb$ has no deg 2 atom and at most one deg 3 atom, and contains a cycle, then we can decompose 1 into $\leq C$ indicator functions, such that for each indicator function, $\Mb'$ must contain a good component (possibly after cutting the \{33A\}-component into a \{3\}- and a \{2\}-component).
\end{enumerate}
\end{proposition}
\begin{proof} (1) The proof is basically the same as Proposition 10.2 in \cite{DHM24}. The only adjustment needed is due to the possibility of $\Mb$ having double bonds, which only affects the part of the proof involving \{33A\}-components.

More precisely, assuming $\Mb$ has no deg $2$ atom, and either contains a cycle or contains at least two deg $3$ atoms, we need to prove that $\Mb'$ has one \{33A\}- or good component. Suppose not, then we will show that $\Mb'$ also cannot contain any \{2\}-component; this would then lead to a contradiction by repeating the same proof of Proposition 10.2 in \cite{DHM24}.

Now assume $\Mb'$ has a \{2\}-component, let the first atom $\mf$ that has deg 2 when we cut it, and consider the cutting operation that turns $\mf$ into deg 2, which must be cutting some atom $\pf$ as free. This cutting creates a fixed end at $\mf$, so by Definition \ref{def.cutting}, this $\pf$ must be a parent or child of $\mf$. At the time of cutting $\pf$, if $\mf$ has deg 3, then $\pf$ must also have deg 3, which violates Definition \ref{def.alg_up} (\ref{it.alg_up_3}). Therefore the operation of cutting $\pf$ must turn the deg 4 atom $\mf$ into deg 2. This means that
\begin{itemize}[$(\bigstar)$]
\item $\pf$ and $\mf$ are \emph{connected by a double bond}.
\end{itemize}
However, with $(\bigstar)$, it is easy to see that for the two bonds $(e_i,e_j)$ connecting $\pf$ and $\mf$ (which are both bottom or both top), and the corresponding vectors $z_{e_i}=(x_{e_i},v_{e_1})$ and $z_{e_j}=(x_{e_j},v_{e_j})$, there must exist two different values $t_1$ that satisfy the equality (\ref{eq.intmini_1+}). Since $\{\pf\}$ is a \{3\}- or \{4\}-component in $\Mb'$, we know that it must be a good component due to Definition \ref{def.good_normal} (\ref{it.good_1}) or (\ref{it.good_3}) (i.e. (\ref{eq.good_normal_1+}) or (\ref{eq.good_normal_4++})), which leads to a contradiction.

(2) The proof is basically the same as Proposition 10.2 in \cite{DHM24}, with the same adjustments made as in (1) above in the case of double bonds. We omit the details.
\end{proof}
Next, we recall the notions of strong and weak degeneracies (Propositions 10.5 and 10.16 in \cite{DHM24}), and properness and primitivity (Definition 8.2 and Proposition 10.4 in \cite{DHM24}).
\begin{definition}
\label{def.degen} We define the following notions.
\begin{enumerate}
\item \emph{Strong degeneracy.} Define a pair of atoms $\{\nf,\nf'\}$ to be \textbf{strongly degenerate} if they are adjacent, and there exist edges $(e,e')$ at $\nf$ and $\nf'$ respectively, such that cutting $\{\nf,\nf'\}$ as free in $\Mb$ and turning $(e,e')$ into fixed ends results in a \{33A\}-component, and we have the restriction (by indicator functions) that $|x_{e}-x_{e'}|\leq\varepsilon^{1-1/(8d)}$ and $|v_{e}-v_{e'}|\leq\varepsilon^{1-1/(8d)}$.
\item\emph{Weak degeneracy}. Define the atom pair $\{\nf,\nf'\}$ to be \textbf{weakly degenerate} if they are adjacent, and we have the restriction (by indicator functions) that $|t_{\nf}-t_{\nf'}|\leq \varepsilon^{1/(8d)}$. Define also one atom $\nf$ to be weakly degenerate if we have the restriction $|v_e-v_{e'}|\leq\varepsilon^{1/(8d)}$ for two distinct edges $(e,e')$ at $\nf$.
\item \emph{Primitivity.} Define an atom pair $\{\pf,\pf'\}$ in $\Mb$ to be \textbf{primitive} if $\pf$ is a parent of $\pf'$ (but they are not connectedby a double bond), and the other parent of $\pf'$ is not a descendant of $\pf$.
\item \emph{Properness.} Define a molecule $\Mb$ to be \textbf{proper}, if it is a forest, contains only deg 3 and 4 atoms, with the distance (i.e. length of shortest undirected path) between any two deg 3 atoms being at least 3.
\end{enumerate}
\end{definition}
\subsection{The cutting algorithm: general case}\label{sec.cutting_final}
A major part of the proof of Proposition \ref{prop.comb_est_extra} is the same as in Section 13.3 in \cite{DHM24}, but the torus case involves one new difficulty, namely that the number of collisions between a fixed number of particles now has no absolute upper bound, unlike the Euclidean case due to \cite{BFK98}. This difficulty then requires one new ingredient, which is treated in Proposition \ref{prop.newcase} below.
\begin{proposition}\label{prop.comb_est_extra_2} Define the function 
\begin{equation}\label{defG}G(q)=(2q+1)^{10}\cdot\binom{q}{2}\cdot (32q^{3/2})^{q^2}\end{equation} for positive integers $q$. Suppose $\Mb$ satisfies the following properties:
\begin{enumerate}
\item $\Mb$ is a full molecule, has at most $q$ bottom free ends and at least $G(q)$ atoms;
\item $\Mb$ does not contain double bonds;
\item $\Mb$ does not contain any strongly degenerate primitive pair as defined in Definition \ref{def.degen};
\end{enumerate}
Then, for this $\Mb$, there exists a cutting sequence such that Proposition \ref{prop.comb_est_extra} (3) holds true.
\label{prop.newcase}
\end{proposition}
The proof of Proposition \ref{prop.newcase} is contained in Section \ref{sec.newcase} below. In the rest of this subsection we shall prove Proposition \ref{prop.comb_est_extra} under the assumption of Proposition \ref{prop.newcase}. The proof is essentially the same as in Section 13.3 in \cite{DHM24}; for the reader's convenience we still include it here.
\begin{definition}
\label{def.trans} Let $\Mb$ be a full molecule of C-atoms, and $A\subseteq \Mb$ is an atom set. We say $A$ is transversal, if we can decompose $\Mb\backslash A$ into two disjoint subsets $A^+$ and $A^-$, such that no atom in $A^-$ is parent of any atom in $A\cup A^+$, and no atom in $A$ is parent of any atom in $A^+$.

For any transversal set $A$, define the set $X_0(A)$  such that, an atom $\nf\in X_0(A)$ if and only if $\nf\not\in A$ and $\nf$ has two bonds connected to two atoms in  $A$. Then define the set $X(A)$ inductively as follows: an atom $\nf\in X(A)$ if and only if $\nf\not\in A$ and $\nf$ has two bonds connected to two atoms in $X(A)\cup A$.
\end{definition}
\begin{proposition}
\label{prop.trans}
Let $A$ be a connected transversal subset and $\Mb$ is a connected molecule.
\begin{enumerate}
\item We can always choose $(A^+,A^-)$ in Definition \ref{def.trans} such that each component of $A^+$ has at least one bond connected to $A$.
\item Let $X_0^\pm (A)=X_0(A)\cap A^\pm$ and $X^\pm (A)=X(A)\cap A^\pm$, then $\nf\in X_0^\pm(A)$ if and only if $\nf\not\in A$ and $\nf$ has two \emph{children} (or two \emph{parents}) in $A$. Similarly, the sets $X^\pm(A)$ can be defined in the same way as $X(A)$, but with the sentence ``$\nf$ has two bonds connected to two atoms in $X(A)\cup A$" replaced by ``$\nf$ has two \emph{children} (or \emph{two parents}) in $X^\pm (A)\cup A$''.
\item The set $X(A)\cup A$ is also connected and transversal. Moreover any atom in $\Mb\backslash (X(A)\cup A)$ has at most one bond with atoms in $X(A)\cup A$, so $X_0(X(A)\cup A)=\varnothing$.
\item Recall $\rho(A)$ defined in Definition \ref{def.molecule_more}. Then there exists a connected transversal set $B\supseteq A$, such that $\rho(B)=\rho(A)$, and either $B=\Mb$ or $X_0(B)\neq\varnothing$.
\end{enumerate}
\end{proposition}
\begin{proof} (1) Suppose $A^+$ has a component $U$ that is not connected to $A$ by a bond, then we replace $A^+$ by $A^+\backslash U$ and $A^-$ by $A^-\cup U$, which still satisfies the requirements for $(A^+,A^-)$ in Definition \ref{def.trans}, but the number of components of $A^+$ decreases by $1$. Repeat until component of $A^+$ is connected to $A$ by a bond (or $A^+$ becomes empty).

(2) If an atom $\nf\in X_0(A)$ belongs to $A^\pm$, then $\nf$ must have two children (or two parents) in $A$. Now suppose $\nf\in X(A)$ has two adjacent atoms $\nf_1$ and $\nf_2$ in $X_0(A)\cup A$. If $\nf_1$ belongs to $A^\pm$, then $\nf$ must be a parent (or child) of $\nf_1$, and the same holds for $\nf_2$. This means that either both $\nf_1,\nf_2\in A^+\cup A$ and both are children of $\nf$ (so $\nf\in A^+$), or both $\nf_1,\nf_2\in A^-\cup A$ and both are parents of $\nf$. Repeating this discussion, we get the desired splitting $X(A)=X^+(A)\cup X^-(A)$.

(3) Connectedness is obvious by definition. To prove transversality, we simply decompose $\Mb\backslash (A\cup X(A))=B^+\cup B^-$, where $B=A^+\backslash X^+(A)$ and $B^-=A^-\backslash X^-(A)$. This then satisfies the requirements (using the fact that any child (or parent) of any atom $\nf\in X^\pm(A)$ must belong to $X^\pm(A)\cup A$). The second statement follows from the definition of $X(A)$ in Definition \ref{def.trans}.

(4) Suppose $A$ is connected and transversal; we may assume $A\neq\Mb$ and $X_0(A)=\varnothing$ (otherwise choose $B=A$). Since $\Mb$ is connected, there must exist an atom $\nf\in \Mb\backslash A$ that is either a parent or child of an atom in $A$. We may assume it is a parent, and then choose a lowest atom $\nf\in \Mb\backslash A$ among these parents. Since $X_0(A)=\varnothing$, we know that $\nf$ has only one bond with atoms in $A$; let $A_1=A\cup\{\nf\}$, then $A_1$ is connected and $\rho(A_1)=\rho(A)$.

Now we claim that $A_1$ is transversal. To see this, define $C$ to be the set of atoms $\mf\in A^+\backslash\{\nf\}$ that are descendants of $\nf$, and decompose $\Mb\backslash A_1=A_1^+\cup A_1^-$, where $A_1^+=A^+\backslash (C\cup\{\nf\})$ and $A_1^-=A^-\cup C$. Clearly no atom in $A_1$ can be parent of any atom in $A_1^+$, and no atom in $A^-$ can be parent of any atom in $A_1\cup A_1^+$. If an atom $\mf\in C$ is parent of some atom $\pf\in A_1^+$, then $\pf\in A^+\backslash C$, but $\mf\in C$ and $\pf$ is child of $\mf$, so we also have $\pf\in C$, contradiction. Finally if $\mf\in C$ is parent of $\pf\in A_1=A\cup \{\nf\}$, clearly $\pf\in A$, but then $\mf$ is also parent of an atom in $A$, contradicting the lowest assumption of $\nf$.

Now we know that $A_1$ is also connected transversal; replacing $A$ by $A_1$ and repeating the above discussion, we eventually will reach some $B$ such that either $B=\Mb$ or $X_0(B)\neq\varnothing$, as desired.
\end{proof}
\begin{definition}[Algorithm \textbf{TRANSUP}]
\label{def.alg_transup}
Suppose $\Mb$ is connected full molecule and $A$ is a connected transversal subset. We may choose $(A^+,A^-)$ as in Proposition \ref{prop.trans} (1). Define the following algorithm:
\begin{enumerate}
\item\label{it.alg_transup_1} If $A$ contains any deg $2$ atom $\nf$, then cut it as free, and repeat until there is no deg $2$ atom left.
\item\label{it.alg_transup_2} Choose a \emph{lowest} atom $\nf$ in the set of all deg 3 atoms in $A$ (or a lowest atom in $A$, if $A$ only contains deg 4 atoms). Let $S_\nf$ be the set of descendants of $\nf$ in $A$.
\item \label{it.alg_transup_3} Starting from $\nf$, choose a \emph{highest} atom $\mf$ in $S_\nf$ that has not been cut. If $\mf$ has deg 3, and is adjacent to an atom $\pf\in X_0^+(A)$ that also has deg 3, then cut $\{\mf,\pf\}$ as free; otherwise cut $\mf$ as free. Repeat until all atoms in $S_\nf$ have been cut. Then go to (\ref{it.alg_up_1}).
\item \label{it.alg_transup_4} Repeat (\ref{it.alg_transup_1})--(\ref{it.alg_transup_3}) until all atoms in $A$ have been cut. Then cut the remaining part of $A^+$ as free and cut it into elementary components using \textbf{UP}, then cut $A^-$ into elementary components using \textbf{DOWN}.
\end{enumerate}
We also define the dual algorithm \textbf{TRANSDN} by reversing the notions of parent/child etc. Note that we replace $A^+$ and $X_0^+(A)$ by $A^-$ and $X_0^-(A)$, and also replace $A^+$ by $A^-$ in Proposition \ref{prop.trans} (1).
\end{definition}
\begin{proposition}
\label{prop.alg_transup} Suppose $\Mb$ has no double bond. For the algorithm \textbf{TRANSUP} (and same for \textbf{TRANSDN}), we have $\#_{\{33B\}}=\#_{\{44\}}=0$ and $\#_{\{4\}}=1$, and $\#_{\{33A\}}\geq |X_0^+(A)|/2-\rho(A)-1$.
\end{proposition}
\begin{proof} It is clear, in the same way as in the proof of Proposition \ref{prop.alg_up}, that during the process, there is no top fixed end in $(A\backslash S_\nf)\cup\{\nf\}$ or in $A^+$, and there is no bottom fixed end in $S_\nf\backslash\{\nf\}$ or in $A^-$, so $\#_{\{33B\}}=\#_{\{44\}}=0$. Moreover $\#_{\{4\}}=1$ because no component of the remaining part of $A^+$ after Definition \ref{def.alg_transup} (\ref{it.alg_transup_3}) can be full when it is first cut in Definition \ref{def.alg_transup} (\ref{it.alg_transup_4}) (thanks to Proposition \ref{prop.trans} (1)), and no component of $A^-$ can be full when it is first cut in Definition \ref{def.alg_transup} (\ref{it.alg_transup_4}) (thanks to $\Mb$ being connected).

To prove the lower bound for $\#_{\{33A\}}$, note that only one atom in $A$ belongs to a \{4\}-component (as $A$ is connected). Let the number of atoms in $A$ that belongs to a \{2\}-component be $q$, then by using the invariance of $|\Ec_*|-3|\Mc|$ during cutting operations on $A$ (where $\Ec_*$ is the set of all free ends and bonds at atoms in $A$, including in all the elementary components), we deduce that $q\leq\rho(A)$. To see this, just note that  initially $|\Ec_*|-3|\Mc|$ is $-\rho(A)+1$ and is finally $-q+1$. Then, for each $\nf\in X_0^+(A)$, which is connected to two atoms $\nf_1,\nf_2\in A$ by two bonds, assume say $\nf_2$ is cut after $\nf_1$. The total number of such $\nf_2$ is at least $|X_0^+(A)|/2$ as each $\nf_2$ can be obtained from at most two $\nf$. For each such $\nf_2$, the corresponding $\nf$ must have deg 3 when it is cut, so by Definition \ref{def.alg_transup} (\ref{it.alg_transup_3}), it must belong to either \{4\}-, or \{2\}-, or \{33A\}-component. Using the upper bound for the number of \{4\}- and \{2\}-components, this completes the proof.
\end{proof}
\begin{definition}[The function \textbf{SELECT2}]
\label{def.func_select2} Let $A$ be a connected molecule with only C-atoms, no bottom fixed end, and no deg 2 atoms. Let $Z$ be the set of deg 3 atoms in $A$, and let $Y$ be a subset of atoms in $A$ such that $A$ becomes a forest after removing the atoms in $Y$. We define the function $\textbf{SELECT2}=\textbf{SELECT2}(A,Z,Y)$ as follows.
\begin{enumerate}
\item\label{it.func_select2_1} Consider all the components of $Z\cup Y$ in $A$, which is a finite collection of disjoint subsets of $A$.
\item\label{it.func_select2_2} If any two of the subsets in (\ref{it.func_select2_1}), say $U$ and $V$, have the shortest distance (within $A$) which is at most 4, then we choose one shortest path between an atom in $U$ and an atom in $V$, and let the atoms on this path be $\nf_j$. Then replace the two sets $U$ and $V$ by a single set which is $U\cup V\cup \{\nf_j\}$.
\item\label{it.func_select2_4} Repeat (\ref{it.func_select2_2}) until this can no longer be done. Next, if a single subset $U$ contains two atoms (which may be the same) that are connected by a path of length at most 4 with none of the intermediate atoms belonging to $U$ or any other subset, then add the atoms on this path to $U$. If this causes the shortest distance between two subsets to be $\leq 4$, then proceed to (\ref{it.func_select2_2}) and repeat it as above.
\item\label{it.func_select2_5} When no scenario in (\ref{it.func_select2_2}) or (\ref{it.func_select2_4}) occurs, we output $S:=\textbf{SELECT2}(A,Z,Y)$ as the union of all the current sets.
\end{enumerate}
\end{definition}
\begin{proposition}\label{prop.func_select2} The molecule $A$ becomes a proper forest after cutting $S$ as free, and we also have $|S|\leq 10(|Y|+|Z|+\rho(A))$.
\end{proposition}
\begin{proof} First $A$ becomes a forest after cutting $S$ as free, because $Y\subseteq S$. Assume $A$ is not proper (Definition \ref{def.degen}) after cutting $S$ as free, say there exist two deg 3 atoms $\nf$ and $\nf'$ of distance at most $2$. As $Z\subset S$, we know $\nf$ and $\nf'$ must both have deg 4 in $A$, thus each of them must be adjacent to an atom in $S$, say $\mf$ and $\mf'$, so the distance betwewn $\mf$ and $\mf'$ is at most $4$. Now consider the sets in Definition \ref{def.func_select2} (\ref{it.func_select2_5}) that form $S$. If $\mf$ and $\mf'$ belong to two different sets among them, then we get a contradiction with the absence of the scenario in Definition \ref{def.func_select2} (\ref{it.func_select2_2}); if $\mf$ and $\mf'$ belong to the same set among them, then we get a contradiction with the absence of the scenario in Definition \ref{def.func_select2} (\ref{it.func_select2_4}).

In summary, we know $A$ is a proper forest after cutting $S$ as free. As for the upper bound for $|S|$, simply note that each single step in Definition \ref{def.func_select2} (\ref{it.func_select2_2})--(\ref{it.func_select2_4}) adds at most $3$ atoms to $S$. Moreover each single step in Definition \ref{def.func_select2} (\ref{it.func_select2_2}) decreases the number of subsets by 1 and does not decrease $\rho(S)$, while each single step in Definition \ref{def.func_select2} (\ref{it.func_select2_4}) does not change the number of subsets, but \emph{adds one more cycle to $S$}. In particular adding this cycle increases the value of $\rho(S)$ by at least 1, while we always have $\rho(S)\leq \rho(A)$ because $S\subseteq A$, so the number of steps in Definition \ref{def.func_select2} (\ref{it.func_select2_4}) is bounded by a constant multiple of $\rho(A)$, hence the result.
\end{proof}
\begin{definition}[Algorithm \textbf{MAINTRUP}] 
\label{def.alg_maintrup}
Suppose $\Mb$ is a connected full molecule and $A$ is a connected transversal subset. We may choose $(A^+,A^-)$ as in Proposition \ref{prop.trans} (1). Let the number of bonds connecting an atom in $X^+(A)$ to an atom in $A$ be $\#_{\mathrm{conn}}^+$. It is easy to prove that there exists a set $Y_0\subseteq A$ of at most $\rho(A)$ atoms, such that $A$ becomes a forest after cutting atoms in $Y_0$ as free. Moreover, by decomposing 1 into at most $C^{|A|+\rho(A)}$ indicator functions, we can identify a set of weakly degenerate atoms and atom pairs in $A$ in the sense of Definition \ref{def.degen}; let $Y_1\subseteq A$ be the atoms involved in these weak degeneracies. We define the following algorithm, which contains two \emph{Options} that we can choose at the beginning.

In \emph{Option 1} we do the followings:
\begin{enumerate}
\item Cut $A$ as free, then cut it into elementary components using \textbf{UP}.
\item Then cut $A^+$ as free and cut it into elementary components using \textbf{UP}, then cut $A^-$ into elementary components using \textbf{DOWN}.
\end{enumerate}

In \emph{Option 2} we do the followings:
\begin{enumerate}
\item\label{it.alg_maintrup_1}  Cut all atoms in $X_0^+(A)$ as free. If any atom in $A$ becomes deg 2, also cut it as free until $A$ has no deg 2 atom.
\item\label{it.alg_maintrup_2} If $A$ remains connected after the above step, let $Z$ be the set of deg 3 atoms in $A$, and $Y$ be those atoms in $Y_0\cup Y_1$ that have not been cut. Define $S=\textbf{SELECT2}(A,Z,Y)$ as in Definition \ref{def.func_select2}, then cut $S$ as free and cut it into elementary components using \textbf{DOWN}. If $A$ is not connected, apply this step to each connected component of $A$.
\item\label{it.alg_maintrup_3} If not all atoms in $X^+(A)$ have been cut, then choose a lowest atom $\nf$ in $X^+(A)$ that has not been cut. If $\nf$ is adjacent to an atom $\pf\in A$ that has deg 3, then cut $\{\nf,\pf\}$ as free; otherwise cut $\nf$ as free.
\item\label{it.alg_maintrup_3+} If any two deg 3 atoms $(\rf,\rf')$ in $A$ becomes adjacent, then cut $\{\rf,\rf'\}$ as free. Repeat until no such instances exist.
\item\label{it.alg_maintrup_3++} If the distance of any two deg 3 atoms $(\rf,\rf')$ in $A$ becomes 2, say $\rf$ and $\rf'$ are both adjacent to some $\rf''\in A$, then cut $\{\rf,\rf',\rf''\}$ as free and then cut $\rf$ as free from $\{\rf,\rf',\rf''\}$. Go to (\ref{it.alg_maintrup_3+}). Repeat until $A$ becomes proper again.
\item Repeat (\ref{it.alg_maintrup_3}) until all atoms in $X^+(A)$ have been cut. Then choose any deg 3 atom in $A$ and cut it as free. If the distance between any two deg 3 atoms in $A$ becomes at most 2, go to (\ref{it.alg_maintrup_3+}) and repeat (\ref{it.alg_maintrup_3+})--(\ref{it.alg_maintrup_3++}) until $A$ becomes proper again. Then choose the next deg 3 atom in $A$ and cut it as free, and repeat until all atoms in $A$ have been cut.
\item\label{it.alg_maintrup_5} Finally, cut (the remaining parts of) $A^+$ as free and cut it into elementary components using \textbf{UP}, and then cut $A^-$ into elementary components using \textbf{DOWN}.
\end{enumerate}
We define the dual algorithm \textbf{MAINTRDN} in the same way (so $\#_{\mathrm{conn}}^+$ is replaced by $\#_{\mathrm{conn}}^-$ etc.).
\end{definition}
\begin{proposition}
\label{prop.alg_maintrup} Suppose $\Mb$ has no double bond. In \emph{Option 1} of \textbf{MAINTRUP} in Definition \ref{def.alg_maintrup} (same for \textbf{MAINTRDN}), we have
\begin{equation}\label{eq.alg_maintrup_1}
\#_{\{33B\}}=\#_{\{44\}}=0,\quad \#_{\{4\}}=1;\qquad \#_{\mathrm{good}}\geq\frac{1}{10}\cdot |Y_1|-\rho(A)-1.
\end{equation}
 In \emph{Option 2} of of \textbf{MAINTRUP} in Definition \ref{def.alg_maintrup}, we have $\#_{\{44\}}=0$ and all \{33B\}-components are good, and moreover
\begin{equation}\label{eq.alg_maintrup_2}
\begin{aligned}
\#_{\{33A\}}+\#_{\{33B\}}&\geq\frac{1}{10}\cdot\big(\#_{\mathrm{conn}}^+-10^5(|Y_1|+\rho(A)+|X_0^+(A)|)\big),\\
\#_{\{4\}}&\leq |Y_1|+\rho(A)+|X_0^+(A)|.
\end{aligned}
\end{equation}
\end{proposition}
\begin{proof} The case of \emph{Option 1} is easy; $\#_{\{33B\}}=\#_{\{44\}}=0$ is obvious by definition, and $\#_{\{4\}}=1$ because $A$ is connected and $A^+$ satisfies the assumption in Proposition \ref{prop.trans} (1), in the same way as in the proof of Proposition \ref{prop.alg_transup}. Moreover, each atom in $A$ belongs to a \{4\}- or \{3\}- or \{2\}- or \{33A\}-component, and the total number of \{2\}- and \{33A\}-components is equal to $\rho(A)$. For any atom $\nf$ that belongs to a  \{3\}-component, if it is a weakly degenerate atom, or if it is part of a weakly degenerate pair and is cut after the other atom $\nf'$ of the pair, then $\{\nf\}$ must be a good component by Definition \ref{def.good_normal} (\ref{it.good_1}). This gives at least $|Y_1|/10-\rho(A)-1$ good components, as desired.

In the case of \emph{Option 2}, note that in the whole process, there is no top fixed end in $A^+$ and no bottom fixed end in $A^-$; this is because for the lowest atom $\nf$ in $X^+(A)$ chosen in Definition \ref{def.alg_maintrup} (\ref{it.alg_maintrup_3}), any child of $\nf$ must either belong to $A$ or belong to $X^+(A)$ (and thus will have already been cut). Moreover, any $\nf$ chosen in this step either has deg 2 and no bond connecting to $A$ or has deg 3 and exactly one bond connecting to $A$ (because all atoms in $X_0^+(A)$ have been cut in Definition \ref{def.alg_maintrup} (\ref{it.alg_maintrup_1})). Also any \{33\}-component $\{\nf,\pf\}$ cut in this way must be \{33A\}-component, and no full component can be cut (hence no \{4\}-component created) in Definition \ref{def.alg_maintrup} (\ref{it.alg_maintrup_5}). Therefore, all \{33B\}-component must have both atoms in $A$, and the only \{4\}-components created are those created in Definition \ref{def.alg_maintrup} (\ref{it.alg_maintrup_1}) (contributing at most $|X_0^+(A)|$ many \{4\}-components) and Definition \ref{def.alg_maintrup} (\ref{it.alg_maintrup_2}) (contributing at most $|Y_1|+\rho(A)$ many \{4\}-components), hence the upper bound $\#_{\{4\}}\leq|Y_1|+\rho(A)+|X_0^+(A)|$.

Now we prove the lower bound on $\#_{\{33\}}$. Note that after Definition \ref{def.alg_maintrup} (\ref{it.alg_maintrup_1}), there is no bottom fixed end in $A$ (and hence none in $S$), so we can cut $S$ into elementary components using \textbf{DOWN} as in Definition \ref{def.alg_maintrup} (\ref{it.alg_maintrup_2}). By definition of $Y_0$ and $Y_1$, and by Proposition \ref{prop.func_select2}, we know that after Definition \ref{def.alg_maintrup} (\ref{it.alg_maintrup_2}) is finished, $A$ will become a forest which is proper, and contains no weakly degenerate atoms or atom pairs. Next we prove an upper bound on $|S|$; let $Z_0$ be the set of atoms in $A$ cut in Definition \ref{def.alg_maintrup} (\ref{it.alg_maintrup_1}), then $|Z|\leq 4(|Z_0|+|X_0^+(A)|)$ and $|S|\leq 10(4|Z_0|+4|X_0^+(A)|+|Y_1|+2\rho(A))$ by Proposition \ref{prop.func_select2}. By Definition \ref{def.alg_maintrup} (\ref{it.alg_maintrup_1}), we know that each parent of each atom in $Z_0$ must be in either  $Z_0$ or $X_0^+(A)$. This then leads to $\rho(Z_0)\geq |Z_0|-2|X_0^+(A)|$, but also $\rho(Z_0)\leq\rho(A)$, so $|Z_0|\leq \rho(A)+2|X_0^+(A)|$, thus
\[|S|\leq 2\cdot 10^2(\rho(A)+|Y_1|+|X_0^+(A)|).\]

Now, after cutting $X_0^+(A)\cup Z_0\cup S$ as free, the number of bonds connecting an atom in $A$ to an atom in $X^+(A)$ is still at least \[(\#_{\mathrm{conn}}^+)'=\#_{\mathrm{conn}}^+-2(|X_0^+(A)|+|Z_0|+|S|)\geq\#_{\mathrm{conn}}^+-10^4(\rho(A)+|Y_1|+|X_0^+(A)|).\] At this point we can apply the same arguments as in the proof of Proposition 11.11 in \cite{DHM24} to show that
\begin{equation}\label{eq.alg_maintrup_3}\#_{\{33A\}}+\#_{\{33B\}}\geq\frac{1}{10}\cdot \big(\#_{\mathrm{conn}}^+-10^5(\rho(A)+|Y_1|+|X_0^+(A)|)\big).\end{equation} Here the role of UD connections is played by the $(\#_{\mathrm{conn}}^+)'$ bonds connecting an atom in $A$ to an atom in $X^+(A)$. Note that after cutting $S$ as free, the number of components of $A$ is at most
\[\#_{\mathrm{comp}(A)}\leq 1+4(|X_0^+(A)|+|S|+|Z_0|)\leq 10^5(\rho(A)+|Y_1|+|X_0^+(A)|),\] so among these $(\#_{\mathrm{conn}}^+)'$ bonds (we denote this set by $Q$), there exists a subset $Q'\subseteq Q$ of at least $(\#_{\mathrm{conn}}^+)'-\#_{\mathrm{comp}(A)}$ bonds, such that each bond in $Q'$ is connected to some other bond in $Q$ via $A$. The same proof for Proposition 11.11 in \cite{DHM24} then applies, which leads to (\ref{eq.alg_maintrup_3}). This completes the proof.
\end{proof}
Now we can finish the proof of Proposition \ref{prop.comb_est_extra} under the assumption of Proposition \ref{prop.newcase}.
\begin{proof}[Proof of Proposition \ref{prop.comb_est_extra} assuming Proposition \ref{prop.newcase}] We start with two simple cases.

\emph{Simple case 1}: assume $\Mb$ contains at least one pair of strongly degenerate and primitive atoms $(\mf,\nf)$. In this case we cut $\{\mf,\nf\}$ as free and cut the rest of $\Mb$ into elementary components using either \textbf{UP} or \textbf{DOWN} (similar to the proof of Proposition 10.5 in \cite{DHM24}). This creates one good \{44\}-component and no \{4\}-component (as $\Mb$ is connected), which meets the requirement of Proposition \ref{prop.comb_est_extra} (1). 

Below we shall assume $\Mb$ contains no pair of strongly degenerate and primitive atoms; by the same proof as Proposition 10.4 in \cite{DHM24}, we know that that each \{33A\}-component is either good or can be cut into one \{2\}- and one good \{3\}-atom. Therefore, from now on, we will treat \{33A\}-components as good.

\emph{Simple case 2}: assume $\Mb$ contains a double bond between two atoms $\nf_1$ and $\nf_2$. First assume that (say) $\nf_2$ is connected to another atom $\nf_3$ by another double bond; we may assume that $\nf_1\rightarrow\nf_2\rightarrow\nf_3$, where $\nf\rightarrow\mf$ means $\nf$ is the parent of $\mf$. Then, we apply the following variant of the algorithm \textbf{UP} to $\Mb$ (Definition \ref{def.alg_up}). We start by choosing the deg 4 atom $\nf_2$ as the $\nf$ in Definition \ref{def.alg_up} (\ref{it.alg_up_2}), and let $S_{\nf_2}$ be the set of descendants of $\nf_2$. We then cut all the atoms in $S_{\nf_2}$ as in Definition \ref{def.alg_up} (\ref{it.alg_up_3}) with $\nf_2$ being the first and $\nf_3$ being the second, and then repeat the steps in Definition \ref{def.alg_up} to cut the rest of $\Mb$ (note in particular we cut $\nf_1$ as a \{2\}-component). By slightly adjusting the proof of Proposition \ref{prop.alg_up}, we can show that $\Mb$ is indeed cut into elementary components with $\{\nf_2\}$ being a \{4\}-component. Since $\nf_2$ has two double bonds, it is easy to see that $\{\nf_2\}$ satisfies the condition in Proposition \ref{prop.intmini} (\ref{it.intmini_extra_2}), so it has excess $\varepsilon^{d-1/2}$ by Proposition \ref{prop.newprop_3} (\ref{it.newprop_1}). This meets the requirement of Proposition \ref{prop.comb_est_extra} (3).

Now, assume neither $\nf_1$ nor $\nf_2$ has another double bond. Assume say $\nf_1\rightarrow\nf_2$, then we apply the variant of \textbf{UP} defined above, but we first choose $\nf_1$ (instead of $\nf_2$) as the $\nf$ in Definition \ref{def.alg_up} (\ref{it.alg_up_2}). In this cutting sequence, we first cut $\nf_1$ and next $\nf_2$, and then repeat the remaining steps in Definition \ref{def.alg_up}, to cut $\Mb$ into elementary components. Note that after $\nf_1$ and $\nf_2$ have been cut, there is no deg 2 atom in $\Mb$; then by Proposition \ref{prop.alg_up}, the subsequent application of \textbf{UP} must produce either another good component, or another \{33A\}-component which is also good. Note that this good component has excess $\leq\varepsilon^{1/(8d)}$ by Definition \ref{def.good_normal} and Proposition \ref{prop.newprop_3}; moreover, since $\nf_1$ has a double bond, it must satisfy condition in Proposition \ref{prop.intmini} (\ref{it.intmini_4_extra}), so $\{\nf_1\}$ is also good and has excess $\varepsilon^{d-1}$ by Proposition \ref{prop.newprop_3} (\ref{it.newprop_1}). This meets the requirement of Proposition \ref{prop.comb_est_extra} (3).

\smallskip
Below we assume $\Mb$ has no double bond. We then consider the two main cases.

\emph{Main case 1}: assume there exists a connected transversal set $A$ (Definition \ref{def.trans}), such that $|X^+(A)|\geq G(\#_{\mathrm{conn}}^+(A))$ (or $|X^-(A)|\geq G(\#_{\mathrm{conn}}^-(A))$), where $X^\pm(A)$ is defined in Definition \ref{def.trans} and Proposition \ref{prop.trans}, the function $G$ is as (\ref{defG}), and $\#_{\mathrm{conn}}^\pm(A)$ is defined in Definition \ref{def.alg_maintrup}. We may assume $\#_{\mathrm{conn}}^+(A)=q$ and $|X^+(A)|\geq G(q)$. Now we cut $X^+(A)$ as free from $\Mb$, and subsequently cut $X^+(A)$ and $\Mb\backslash X^+(A)$ separately.
\begin{itemize}
\item Concerning $X^+(A)$: it is clear from the definition of $X^+(A)$ that, after cutting $X^+(A)$ as free, every \emph{bottom} free end of $X^+(A)$ must originally be a bond between an atom in $X^+(A)$ and an atom in $A$, so the number of these bottom free ends does not exceed $\#_{\mathrm{conn}}^+(A)=q$. Since $|X^+(A)|\geq G(q)$, and $X^+(A)$ does not contain any double bond or strongly degenerate primitive pair, by Proposition \ref{prop.newcase} we can cut $X^+(A)$ into elementary components such that Proposition \ref{prop.comb_est_extra} (3) holds true.
\item Concerning $\Mb\backslash X^+(A)$: with Proposition \ref{prop.comb_est_extra} (3) for the molecule $X^+(A)$ alone already providing enough excess ($\leq \varepsilon^{d-1+1/(15d)}$), we now only need to cut $\Mb\backslash X^+(A)$ into \{2\}-, \{3\}- and \{33A\}-components. Note that by Proposition \ref{prop.trans} (1) we may assume that each component of $A^+$ has at least one bond connected to $A$.

Now, after cutting $X^+(A)$ as free, we know $A$ is still connected and has at least one fixed end, but has no \emph{bottom} fixed end. We can then cut $A$ into the needed components using \textbf{DOWN}. Then, each component in the remaining part of $A^+$ must have at least one fixed end, but has no \emph{top} fixed end, so we can cut the remaining part of $A^+$ into the needed components using \textbf{UP}. Finally, after $A\cup A^+$ has been cut, each component of $A^-$ will again have at least one fixed end but has no \emph{bottom} fixed end, so we can cut $A^-$ into the needed components using \textbf{DOWN}. This completes the proof in \emph{Main case 1} by meeting the requirement of Proposition \ref{prop.comb_est_extra} (3).
\end{itemize}

\emph{Main case 2}: assume there does not exist a set $A$ as in \emph{Main case 1} above. We start by choosing a connected transversal subset $A_1$ of $\Mb$ which is a tree (i.e. $\rho(A_1)=0$) and either $A_1=\Mb$ or $X_0(A_1)\neq\varnothing$; the existence of such $A_1$ follows in the same way as in Proposition \ref{prop.trans} (4) starting from a single atom with no children. If $A_1\neq\Mb$, let $X_1=X(A_1)$, then $A_1\cup X_1$ is connected transversal by Proposition \ref{prop.trans} (3), so we can find $A_2\supseteq A_1\cup X_1$ starting from $A_1\cup X_1$, using Proposition \ref{prop.trans} (4). If $A_2\neq \Mb$, then let $X_2=X(A_2)$ and find $A_3$ by Proposition \ref{prop.trans} (4), then define $X_3=X(A_3)$ and so on.

Recall the function $G(q)$ in (\ref{defG}). Define $D=K_1:=(60d)^{60d}$ and $K_{j+1}=(60d\cdot G((60dK_{j})^{60d}))^{60d}$ for $j\geq 1$. We will assume $\Gamma>(60d)^{60d} K_D$ below, and consider several different cases.

(1) Suppose $A_j$ exists and $X_0(A_j)\neq\varnothing$ for all $j\leq D$. By construction $\rho(A_j)=\rho(A_{j-1}\cup X_{j-1})$, so we can list the elements of $B_j:=A_j\backslash (A_{j-1}\cup X_{j-1})$ as $\nf_1^j,\cdots ,\nf_{|B_j|}^j$ such that $\nf_1^j$ has only one bond with $A_{j-1}\cup X_{j-1}$ and $\nf_i^j$ has only one bond with $A_{j-1}\cup X_{j-1}\cup\{\nf_1^j,\cdots,\nf_{i-1}^j\}$. Similarly we can list the atoms in $B_1:=A_1$. Choose also an element $\mf_j\in X_0(A_j)$ for each $j$, then this $\mf_j$ is adjacent to two atoms in $A_j$, with at least one being in $B_j$. We then perform the following cutting sequence:

Starting from $j=1$, for each $j$, cut the atoms $(\nf_i^j:i\leq |B_j|)$ as free in the increasing order of $i$. However, if $\nf_i^j$ is adjacent to $\mf_{j}$ and $\mf_{j}$ has deg 3 when we cut $\nf_i^j$, then we cut $\{\nf_i^j,\mf_{j}\}$ instead of $\nf_i^j$. After all atoms in $B_j$ have been cut, the atom $\mf_{j}$ will also have been cut; we then cut the remaining atoms in $X_j$ as free (starting from the lowest ones; each will have deg 2 when we cut it), and proceed with $B_{j+1}=A_{j+1}\backslash (A_j\cup X_j)$ and so on.

In this way, it is clear that each $\mf_j$ must belong to a \{33\}-component, which also has to be \{33A\}-component because $\mf_j$ will have two top (or two bottom) free ends when it is cut, assuming $\mf_j\in X_0^+(A_j)$ (or $\mf_j\in X_0^-(A_j)$). This produces at least $D$ many \{33A\}-components (which are treated as good), while $\#_{\{4\}}=1$, which meets the requirement of Proposition \ref{prop.comb_est_extra} (2).

(2) Suppose $A_j=\Mb$ for some $j<D$, then $\rho(A_j)\geq\Gamma$ for this $j$. Therefore we may choose the smallest $j$ such that $\rho(A_{j+1})\geq K_{j+1}$. Let $A=A_j$, then we have
\[\rho(A)<K_j,\quad K_{j+1}\leq \rho(A_{j+1})=\rho(A\cup X(A))\leq \rho(A)+10|X(A)|,\] so $|X(A)|\geq (20)^{-1}\cdot K_{j+1}$ (note also that $|X^+(A)|$ and $\rho(A_{j+1})$ etc. are all bounded above by $\rho(\Mb)\leq 2\Gamma$, a constant depending only on $d$). We may assume $|X^+(A)|\geq (40)^{-1}\cdot K_{j+1}$; consider the set $X^+(A)$, by the assumption of \emph{Main case 2} we then know that $\#_{\mathrm{conn}}^+>(60d)^{60d}\cdot K_j$ in Definition \ref{def.alg_maintrup}. We then
\begin{itemize} 
\item Run algorithm \textbf{TRANSUP} and apply Proposition \ref{prop.alg_transup} if $|X_0^+(A)|\geq (30d)^{30d}\cdot K_j$;
\item Run algorithm \textbf{MAINTRUP}, \emph{Option 1} and apply Proposition \ref{prop.alg_maintrup} if $|Y_1|\geq (30d)^{30d}\cdot K_j$;
\item Run algorithm \textbf{MAINTRUP}, \emph{Option 2} and apply Proposition \ref{prop.alg_maintrup}, if $|X_0^+(A)|\leq (30d)^{30d}\cdot K_j$ and $|Y_1|\leq (30d)^{30d}\cdot K_j$.
\end{itemize} In any case we have met the requirements of Proposition \ref{prop.comb_est_extra} (2), so the proof is complete.
\end{proof}
\subsection{The cutting algorithm: special case}\label{sec.newcase} In this subsection we prove Proposition \ref{prop.newcase}. Fix a quantity $\varepsilon^*:=\exp(-|\log\varepsilon|^{1/2})$; note that as $\varepsilon\to 0$, this $\varepsilon^*$ vanishes slower than any power of $\varepsilon$ and faster than any $|\log\varepsilon|^{-C^*}$. In the proof below we will use the notion of \textbf{long bonds}, where a long bond is a bond between two atoms $\nf_1$ and $\nf_2$ such that we have the restriction (by indicator functions) that $|t_{\nf_1}-t_{\nf_2}|\geq\varepsilon^*$.
\begin{lemma} Proposition \ref{prop.newcase} is true, if $\Mb$ contains a triangle with a long bond.
\label{lem.newlem}
\end{lemma}
\begin{proof} Assume $\Mb$ contains a triangle with atoms $\nf_1\rightarrow\nf_2\rightarrow\nf_3$, where (say) $|t_2-t_3|\geq\varepsilon^*$. We first consider a possible atom $\mf_1$ that is common parent of $\nf_1$ and $\nf_2$, i.e. $\mf_1\rightarrow \nf_1$ and $\mf_1\rightarrow\nf_2$; if $\mf_1$ exists we may then consider possible common parent $\mf_2$ of $\mf_1$ and $\nf_1$, and common parent $\mf_3$ of $\mf_2$ and $\mf_1$, and so on. Assume this stops at $\mf_r$. If $r\geq 20$, we consider all the bonds between $\mf_j$ and $\mf_{j-1}$ for $r-15\leq j\leq r-5$; note that all these atoms belong to at most 3 particle lines. Since we may assume that all the collisions represented by atoms in $\Mb$ can actually occur for some initial configurations, we know that for some initial configuration on $\Tb^{3d}\times\Rb^{3d}$, these 3 particles collide at least 10 times as represented by $\mf_j\,(r-15\leq j\leq r-5)$ on $\Tb^d$. However, we know (see for example \cite{CM93}) that 3 particles can only collide 4 times in the $\Rb^d$ dynamics, so the trajectory of these particles among these 10 collisions \emph{cannot be contained in a single fundamental domain of $\Tb^d$ in $\Rb^d$} (otherwise this portion of the $\Tb^d$ dynamics would coincide with the $\Rb^d$ dynamics and cannot support this many collisions). Since all velocities $|v_e|\leq|\log\varepsilon|^{C^*}$ (see Remark \ref{rem.parameter}), we conclude that $|t_{\mf_j}-t_{\mf_{j-1}}|\geq \varepsilon^*$ for at least one such $j$, so we can choose the triangle with atoms $(\mf_{j+1},\mf_j,\mf_{j-1})$ which has a long bond, and the corresponding $r$ value becomes $r'\leq r-j\leq 15$.

Therefore, we can find a triangle of atoms $(\nf_1,\nf_2,\nf_3)$ that contains a long bond between $\nf_2$ and $\nf_3$, and satisfies $r\leq 20$ with $r$ defined as above. Now we apply the following variant of \textbf{UP} to $\Mb$: first choose the deg 4 atom $\nf_1$ as the $\nf$ in Definition \ref{def.alg_up} (\ref{it.alg_up_2}), and let $S_{\nf_1}$ be the set of descendants of $\nf_1$. Then cut $\{\nf_1\}$ as a \{4\}-component; if $\nf_2$ and $\nf_3$ have a common child atom $\nf_4$ then cut $\{\nf_2,\nf_3,\nf_4\}$ as a \{334T\}-component, otherwise cut $\{\nf_2,\nf_3\}$ as a \{33A\}-component. Then we subsequently cut each $\mf_j$ as a \{2\}-component \emph{before returning to} $S_{\nf_1}$ and carrying out the steps in Definition \ref{def.alg_up} (\ref{it.alg_up_3}). Note that at this time there is still no top fixed end in $\Mb\backslash S_{\nf_1}$ and no bottom fixed end in $S_{\nf_1}$, which guarantees that we can return to algorithm \textbf{UP} and follow the steps in Definition \ref{def.alg_up} to complete the following cuttings; moreover, we either have a \{334T\}-component or have a \{33A\}-component.

\emph{If we get a \{334T\}-component}: in this case, let $\mu':=\varepsilon^{1/(8d)}$ and $(x_j,v_j)=z_{e_j}$ be as in Proposition \ref{prop.intmini_3} (\ref{it.intmini_3_2}) associated with this \{334T\}-component. Note that $e_1$ and $e_7$ are the two fixed ends of the \{334T\}-component, and also correspond to the two bonds at $\nf_1$.

If $|x_1-x_7|_\Tb\leq \mu'$ or $|v_1-v_7|\leq \mu'$, then $\{\nf_1\}$ is a good component with excess $\varepsilon^{1/(8d)}$ by Definition \ref{def.good_normal} (\ref{it.good_3}), which meets the requirement in Proposition \ref{prop.comb_est_extra} (3) upon cutting $\{\nf_2,\nf_3,\nf_4\}$ into the \{33A\}-component $\{\nf_2,\nf_3\}$ and the 2-component $\{\nf_4\}$. If $\min(|x_1-x_7|_\Tb,|v_1-v_7|)\geq\mu'$ and $|x_1-x_7+t_2(v_1-v_7)|_\Tb\geq \mu'$, then this \{334T\}-component satisfies (\ref{eq.intmini_3_2_1})--(\ref{eq.intmini_3_2_2}), so by Proposition \ref{prop.newprop_3} (\ref{it.newprop_3}) we get excess $\varepsilon^{d-3/4}(\varepsilon^*)^{-2d}$, which also meets the requirement in Proposition \ref{prop.comb_est_extra} (3). Finally, if $|x_1-x_7+t_2(v_1-v_7)|_\Tb\leq \mu'$, note also that $|x_1-x_7+t_{\nf_1}(v_1-v_7)|_\Tb\leq \mu'$ due to the collision $\nf_1$, and that $t_2=t_{\nf_3}$ in the  \{334T\}-component we get, and that $t_{\nf_1}-t_2\geq\varepsilon^*$ and $|v_1-v_7|\geq\mu'$. These conditions then imply that $|(v_1-v_7)\times m|\leq\mu'|\log\varepsilon|^{C^*}$ for some nonzero integer vector $m$ with $|m|\leq|\log\varepsilon|^{C^*}$; therefore, by the same arguments as in the proof of Proposition \ref{prop.intmini} (4), we conclude that $\{\nf_1\}$ has excess $\varepsilon^{1/(8d)}$, which again meets the requirement in Proposition \ref{prop.comb_est_extra} (3).

\emph{If we get a \{33A\}-component}: in this case, since $\nf_4$ does not exist, it can be shown that after cutting $\{\nf_1,\nf_2,\nf_3\}$ and the possible $\mf_j$, there is no deg 2 atom in $\Mb$ but there still exists cycles in $\Mb$ (because the total number of cycles in $\Mb$ is $\gg 20\geq r$ by assumption (1) in Proposition \ref{prop.newcase}). Therefore, by the same proof of Proposition \ref{prop.alg_up}, the remaining cutting process following the steps in Definition \ref{def.alg_up} must produce at least one more \{33A\}-component, which has excess $\varepsilon^{1/(8d)}$. Since the first \{33A\}-component has excess $\varepsilon^{d-1}(\varepsilon^*)^{-d}$ by Proposition \ref{prop.newprop_3} (\ref{it.newprop_2}) and the long bond assumption, we get total excess $\varepsilon^{d-1+1/(8d)}$ which meets the requirement in Proposition \ref{prop.comb_est_extra} (3). This completes the proof.
\end{proof}
From now on, we may assume that the $\Mb$ in Proposition \ref{prop.newcase} does not contain any double bond, any strongly degenerate primitive pair, or any triangle with a long bond. Also, throughout the proof, both $q$ and $|\Mc|$ (the number of atoms in $\Mb$) will be bounded by an absolute constant depending only on the dimension $d$, as this is all we need in the application of Proposition \ref{prop.newcase} in the proof of Proposition \ref{prop.comb_est_extra} in Section \ref{sec.cutting_final}. As such, we may fix a time ordering of all the time variables $t_\nf$ at a cost of a constant factor. The next step in the proof is to reduce Proposition \ref{prop.newcase} to the case of a molecule with two sub-layers, namely the following proposition.
\begin{proposition}
\label{prop.new2layer} Suppose $\Mb$ is as in Proposition \ref{prop.newcase} and does not contain any double bond, any strongly degenerate primitive pair, or any triangle with a long bond. Moreover, assume $\Mb$ is divided into two subsets $\Mb=\Mb_U\cup\Mb_D$, such that:
\begin{enumerate}
\item No atom in $\Mb_D$ is parent of any atom in $\Mb_U$;
\item Each particle line intersects both $\Mb_D$ and $\Mb_U$, and each of $\Mb_D$ and $\Mb_U$ is connected by itself;
\item For $\nf\in\Mb_U$ and $\mf\in\Mb_D$, we have the restriction (by indicator functions) that $|t_{\nf}-t_{\mf}|\geq\varepsilon^*$.
\end{enumerate} Then Proposition \ref{prop.newcase} holds true for this $\Mb$.
\end{proposition}
Next we prove Proposition \ref{prop.newcase} assuming Proposition \ref{prop.new2layer}.
\begin{proof}[Proof of Proposition \ref{prop.newcase} assuming Proposition \ref{prop.new2layer}] Note that $\Mb$ contains at most $q$ particle lines by assumption; define $G_0(q)=(2q+1)^{5}\cdot\binom{q}{2}\cdot (32q^{3/2})^{q^2}$. We will prove Proposition \ref{prop.newcase} by induction on $q$. The base case is easy. Suppose Proposition \ref{prop.newcase} is true for values $<q$, now consider the case of $q$. Using the fixed time ordering of all collisions, we can define the set of ``sublayers" of atoms $\Mb_j$ as follows: suppose $\Mb_{j-1}$ is defined (or $j=0$), then we add the next $G_0(q)$ atoms (in the fixed time order as above) into $\Mb_j$. After this, we keep adding the next atom $\nf$ (in the fixed time order) if and only if the time separation between $\nf$ and its successor $\nf'$ satisfies $t_\nf-t_{\nf'}<\varepsilon^*$. Repeat until all the atoms have been exhausted, so $\Mb$ is divided into subsets $\Mb_j$. We then have the followings:
\begin{enumerate}
\item For each atom $\nf\in\Mb_j$ and $\nf'\in\Mb_{j'}\,(j'<j)$ we must have $t_\nf-t_{\nf'}\geq\varepsilon^*$, in particular $\nf'$ can never be parent of $\nf$.
\item We have $G_0(q)\leq |\Mb_j|\leq 2G_0(q)$ for each $j$, so in particular the number of sublayers $\Mb_j$ must be $\gg 1$. The left hand side is obvious. For the right hand side, just notice that one cannot have $G_0(q)$ collisions (represented by atoms) that all happen within a time interval of length $G_0(q)\cdot \varepsilon^*$; in fact, if this happens, then the trajectories of all the (at most $q$) particles are contained in a ball of radius $G_0(q)\cdot \varepsilon^*\cdot |\log\varepsilon|^{C^*}\ll 1$ due to the upper bound $|v_e|\leq |\log\varepsilon|^{C^*}$, so these trajectories are contained in a single fundamental domain of $\Tb^d$ in $\Rb^d$, in which the $\Tb^d$ dynamics coincides with the $\Rb^d$ dynamics. Since we may assume that all the collisions represented by atoms in $\Mb$ can actually occur for some initial configurations, we would get an initial configuration that leads to $\geq G_0(q)$ collisions in the $\Rb^d$ dynamics, which contradicts the result of \cite{BFK98}.
\end{enumerate}

Note that by our construction, if $\nf_1\prec\nf_2\prec\nf_3$ belong to the same particle line (where $\nf_j$ is a descendant of $\nf_{j+1}$), and $\nf_1$ and $\nf_3$ belong to the same sublayer $\Mb_j$, then we must also have $\nf_2\in\Mb_j$. Therefore, if we cut $\Mb_j$ as free, it will produce exactly one top free end and one bottom free end, along each particle line that intersects $\Mb_j$. Now we consider two cases.

In the first case, assume that there exists $j$ such that $\Mb_j$ is not connected, or not all particle lines of $\Mb$ intersect $\Mb_j$. In this case, choose a component $\Mb'$ of $\Mb_j$ with largest number of atoms, then $|\Mb'|\geq q^{-1}\cdot G_0(q)\gg G(q-1)$ (note that there are at most $q$ components). Moreover this $\Mb'$ intersects at most $q-1$ particle lines. We cut $\Mb'$ as free from $\Mb$, then this $\Mb'$ satisfies all the assmptions of Proposition \ref{prop.newcase} with $q$ replaced by $q-1$, so we can apply induction hypothesis to cut $\Mb'$ and meet the requirement of Proposition \ref{prop.comb_est_extra} (3). Then we simply need to cut $\Mb\backslash \Mb'$ into \{2\}-, \{3\}- and \{33A\}-components. For this, define $\Mb''$ to be the set of components of $\cup_{j'<j}\Mb_{j'}$ that are connected to $\Mb'$ by a bond; after cutting $\Mb'$ as free, this $\Mb''$ has no bottom fixed end and no full component, so we can cut it using \textbf{DOWN}. After this, the set $\Mb\backslash (\Mb'\cup\Mb'')$ will have no top fixed end and no full component (as originally $\Mb$ is connected), so we can cut it using \textbf{UP}. This allows us to prove the first case.

In the second case, each $\Mb_j$ is connected and intersects each particle line of $\Mb$. Now let $\Mb_1\cup\Mb_2:=\Mb'$, then this $\Mb'$ clearly satisfies all the assumptions of Proposition \ref{prop.new2layer} (with $\Mb_D=\Mb_1$ and $\Mb_U=\Mb_2$). We then cut $\Mb'$ as free from $\Mb$, and cut $\Mb'$ under the requirement of Proposition \ref{prop.comb_est_extra} (3) by Proposition \ref{prop.new2layer}. Then, since $\Mb\backslash\Mb'$ has no top fixed end and no full component, we can cut it into \{2\}-, \{3\}- and \{33A\}-components using \textbf{UP}. This proves the second case and finishes the proof of Proposition \ref{prop.newcase}. 
\end{proof}

For the rest of this subsection we will prove Proposition \ref{prop.new2layer}. This will rely on the following procedure:
\begin{proposition}\label{prop.2cluster} Consider an atom $\mf_0\in\Mb_D$ that does not have a parent in $\Mb_D$. Let $\pb_1$ and $\pb_2$ be the two particle lines containing $\mf_0$. Now choose the first atom in $\nf_0\in\Mb_U$ (in the fixed time order) such that $\pb_1$ and $\pb_2$ are connected by a path that contains only atoms in $\Mb_U$ \emph{before and including} $\nf_0$ (in the fixed time order). Note that this corresponds to the notion of cluster (i.e. the particles representing $\pb_1$ and $\pb_2$ belong to the same cluster) if we only account for the collisions in $\Mb_U$ before and including $\nf_0$.

Now, let $P_j$ be the set of particle lines that are connected to $\pb_j$ by a path that contains only atoms in $\Mb_U$ \emph{before and not including} $\nf_0$. Then we have $\pb_j\in P_j$ and $P_1\cap P_2=\varnothing$. Let $A_j$ be the set of atoms in $\Mb_U$ before and not including $\nf_0$ that belongs to a particle line in $P_j$, then each $A_j$ is connected, any child of an atom in $A_j$ that belongs to $\Mb_U$ must still be in $A_j$, and $A_1$ and $A_2$ are disjoint and no bond exists between them. Moreover $A_j=\varnothing$ if and only if $P_j=\{\pb_j\}$, in which case $\nf_0$ is connected to $\mf_0$ by a bond along $\pb_j$; if $A_j\neq\varnothing$ for each $j$, then $\mf_0$ has one parent $\mf_j^+$ in each $A_j$, and $\nf_0$ has one children $\nf_j^-$ in each $A_j$.
\end{proposition}
\begin{proof} This is straightforward from the definitions of $\nf_0$ and $(P_j,A_j)$. For example, for any $\pf,\pf'\in A_1$, we may assume $\pf\in \pb$ and $\pf'\in\pb'$ where $\pb,\pb'\in P_1$; by definition $\pb$ and $\pb'$ are connected by a path containing only atoms in $\Mb_U$ before and not including $\nf_0$. By definition each of these bonds in this path belongs to a particle line in $P_1$, and each of these atoms in this path belongs to $A_1$. Since also $\pf\in \pb$ and $\pf'\in\pb'$, we know that $\pf$ and $\pf'$ are connected by a path containing only atoms in $A_1$, which proves that $A_1$ is connected. Moreover, if $\pf\in A_j$ and $\qf\in \Mb_U$ is a child of $A_j$, then the particle line containing the bond between $\pf$ and $\qf$ must belong to $P_j$ (as this particle line contains $\pf$), so $\qf\in A_j$ by definition. The other statements are proved in similar ways.
\end{proof}
Next, based on the setup and notions defined in Proposition \ref{prop.2cluster}, we can divide the proof of Proposition \ref{prop.new2layer} into a few major cases, which we treat one by one.
\begin{proposition}\label{prop.newcase_1} Recall the setup and notations in Proposition \ref{prop.2cluster}. Define a \textbf{canonical cycle} to be a cycle which has a unique top atom $\pf_{\mathrm{top}}$ and bottom atom $\pf_{\mathrm{bot}}$, such that the cycle can be divided into two paths between $\pf_{\mathrm{bot}}$ and $\pf_{\mathrm{top}}$, and each one of the two path consists of iteratively taking parents going from $\pf_{\mathrm{bot}}$ to $\pf_{\mathrm{top}}$.

Then, Proposition \ref{prop.new2layer} is true, if for some $j$, either $A_j$ is empty or it does not contain a canonical cycle with bottom atom $\mf_j^+$ (recall the atom $\mf_j^+\in A_j$ in Proposition \ref{prop.2cluster})
\end{proposition}
\begin{proof} Assume $A_1$ does not contain a canonical cycle with bottom atom $\mf_1^+$. We also assume $A_1\neq \varnothing$ (and similarly $A_2\neq\varnothing$); otherwise we shall replace $\mf_1^+$ by $\nf_0$ and the proof proceeds in the same way (and is much easier). Define $B:=A_1\cup A_2\cup\{\mf_0,\nf_0\}$, now we consider three possibilities.

\emph{Case 1}: assume there is no other atom adjacent to two atoms in $B$. In this case, we first cut $B$ as free from $\Mb$, and apply the algorithm \textbf{UP} to $B$ but with the following modifications:
\begin{enumerate}
\item\label{it.new_alg_1} In the first step, we choose the deg 4 atom $\mf_2^+$ as the atom $\nf$ in Definition \ref{def.alg_up} (\ref{it.alg_up_2}); also, in subsequent steps of Definition \ref{def.alg_up} (\ref{it.alg_up_2}), we always choose $\nf$ to be a lowest atom \emph{in $B\backslash\{\mf_0\}$ (i.e. excluding $\mf_0$)}. Note that $B\backslash\{\mf_0\}$ will always have a deg 3 atom in subsequent steps because it is connected in the beginning.
\item\label{it.new_alg_2} Recall $S_\nf$ is the set of descendants of $\nf$ chosen in Definition \ref{def.alg_up} (\ref{it.alg_up_2}); now in Definition \ref{def.alg_up} (\ref{it.alg_up_3}), we always choose $\mf$ to be a highest atom \emph{in $S_\nf\backslash\{\mf_0\}$ (i.e. excluding $\mf_0$)} that has not been cut. Moreover, if $\mf^\pm$ in Definition \ref{def.alg_up} (\ref{it.alg_up_3}) equals $\mf_0$, we must cut \emph{this particular} \{33A\}-component given by $\{\mf,\mf_0\}$, instead of other possible \{33A\}-components containing $\mf$.
\end{enumerate}

With the above modifications, we claim that $\mf_0$ must belong to the \{33A\}-component $\{\mf_1^+,\mf_0\}$. In fact, $\mf_0$ becomes deg 3 after $\mf_2^+$ is cut in the first step. Now consider the atom $\mf_1^+$; it cannot be cut in Definition \ref{def.alg_up} (\ref{it.alg_up_1}) because it has a child $\mf_0$. Suppose it is cut in Definition \ref{def.alg_up} (\ref{it.alg_up_3}), then it must be a descendant of some $\nf\in B\backslash\{\mf_0\}$. Note that since $\mf_1^+\in A_1$, this $\nf$ must belong to $A_1\cup\{\nf_0\}$. The point is that \emph{at most one parent of $\mf_0$ can be in $S_\nf$}, otherwise there would exists a canonical cycle in $A_1$ with bottom atom $\mf_1^+$ and top atom $\nf$ (or $\nf_1^-$ if $\nf=\nf_0$), contradicting our assumptions. It follows that $\mf_1^+$ must be chosen as $\mf$ in Definition \ref{def.alg_up} (\ref{it.alg_up_3}) when it is cut, and it must have deg 3 at this time. Since $\mf_0$ also has deg 3 at this time, by the specification in (\ref{it.new_alg_2}) we must get a \{33A\}-component $\{\mf_1^+,\mf_0\}$ with two fixed ends being both top.

Note that since $\mf_1^+$ and $\mf_0$ are connected by a long bond (as $\mf_0\in\Mb_D$ and $\mf_1^+\in \Mb_U$), we get a \{33A\}-component with excess $\varepsilon^{d-1}(\varepsilon^*)^{-d}$ by Proposition \ref{prop.newprop_3} (\ref{it.newprop_2}). Now, after cutting $B$, we see that $\Mb_U$ has no top fixed end and no full component, and $\Mb_D$ has no bottom fixed end and no full component, so we can cut $\Mb_U$ using \textbf{UP} and $\Mb_D$ using \textbf{DOWN}. Moreover, there is no deg 2 atom after cutting $B$, and $\Mb_D$ still contains a cycle (because it has way more collisions than particle lines due to our construction of sub-layers $\Mb_j$ in the proof of Proposition \ref{prop.newcase} above), so by Proposition \ref{prop.alg_up} we get another \{33A\}-component in $\Mb\backslash B$ which has excess $\varepsilon^{1/(8d)}$. This leads to total excess $\varepsilon^{d-1+1/(8d)}(\varepsilon^*)^{-d}$, which meets the requirement of Proposition \ref{prop.comb_est_extra} (3).

\emph{Case 2}: assume there is an atom $\pf$ adjacent to two atoms in $B\backslash\{\mf_1^+,\mf_0\}$. In this case, we first cut $B\cup\{\pf\}$ as free from $\Mb$. In the same way as \emph{Case 1}, after cutting this, $\Mb_U$ has no top fixed end and no full component, and $\Mb_D$ has no bottom fixed end and no full component, so we can cut $\Mb_U$ using \textbf{UP} and $\Mb_D$ using \textbf{DOWN}. It then suffices to consider $B\cup\{\pf\}$; for this, we apply \textbf{UP} to $B\cup\{\pf\}$ with similar modifications as in \emph{Case 1} above, except that we \emph{also avoid choosing $\pf$} in Definition \ref{def.alg_up} (\ref{it.alg_up_2})--(\ref{it.alg_up_3}), and also prioritize cutting the \{33A\}-component \emph{containing $\pf$} in addition to (but with proprity lower than) the one containing $\mf_0$. 

Now, the same proof in \emph{Case 1} still implies that $\{\mf_1^+,\mf_0\}$ is cut as a \{33A\}-component; moreover, since $\pf$ is connected to two atoms in $B\backslash\{\mf_1^+,\mf_0\}$, this $\pf$ must belong to a \{33A\}- or \{2\}-component. In the latter case, the same proof of Proposition \ref{prop.alg_up} (1) implies that some atom adjacent to $\pf$ must belong to a \{33A\}-component, and this \{33A\}-component cannot be the same as $\{\mf_1^+,\mf_0\}$. This means we have $\{\mf_1^+,\mf_0\}$ (which is a \{33A\}-component with a long bond) and another \{33A\}-component, so we get total excess $\varepsilon^{d-1+1/(8d)}(\varepsilon^*)^{-d}$ same as in \emph{Case 1}, which meets the requirement of Proposition \ref{prop.comb_est_extra} (3).

\emph{Case 3}: assume there is an atom $\pf$ adjacent to one atom in $\{\mf_1^+,\mf_0\}$ and one atom in $B\backslash\{\mf_1^+,\mf_0\}$ (note that $\pf$ cannot be adjacent to both $\mf_1^+$ and $\mf_0$, otherwise we get a triangle with long bond). In this case we cut $B\cup\{\pf\}$ as free from $\Mb$, and cut the rest of $\Mb$ as in \emph{Case 2} above. Then, to cut $B\cup\{\pf\}$, we apply \textbf{UP} to $B\cup\{\pf\}$ with the same modification as in \emph{Case 2} above, except that: when we choose $\mf_1^+$ as the $\mf$ in Definition \ref{def.alg_up} (\ref{it.alg_up_3}) (at which time both $\mf_1^+$ and $\mf_0$ have deg 3, just as in \emph{Case 2}), if $\pf$ also have deg 3, then we should cut $\{\mf_1^+,\mf_0,\pf\}$ as a \{333A\}-component instead of cutting $\{\mf_1^+,\mf_0\}$ as a \{33A\}-component. There are then two cases:

\emph{Case 3.1}: assume $\pf$ has deg 4 when we cut $\{\mf_1^+,\mf_0\}$ as a \{33A\}-component. Then, since $\pf$ is connected to another atom $\qf\in B\backslash\{\mf_1^+,\mf_0\}$, by the same proof in \emph{Case 2}, we know that either $\pf$ or $\qf$ must belong to another \{33A\}-component which is different from $\{\mf_1^+,\mf_0\}$. This then leads to total excess $\varepsilon^{d-1+1/(8d)}(\varepsilon^*)^{-d}$, which meets the requirement of Proposition \ref{prop.comb_est_extra} (3).

\emph{Case 3.2}: assume we have a \{333A\}-component $\{\mf_1^+,\mf_0,\pf\}$ with a long bond between $\mf_1^+$ and $\mf_0$. Let $(e_1,e_7,e_{10})$ be as in Proposition \ref{prop.intmini_3} (\ref{it.intmini_3_1}) and $(x_j,v_j)=z_{e_j}$ associated with this \{333A\}-component, then we consider three cases (by inserting indicator functions).

\emph{Case 3.2.1}: assume $\min(|x_i-x_j|_\Tb,|v_i-v_j|)\geq \varepsilon^{1/(8d)}$ for $i\neq j\in\{1,7,10\}$, and
\begin{equation}\label{eq.newprop_3_2_cite}\inf_{|t|\leq|\log\varepsilon|^{C^*}}|x_1-x_7-t(v_1-v_7)|_\Tb\geq\varepsilon^{1/(8d)};\end{equation} in this case, by Proposition \ref{prop.newprop_3} (\ref{it.newprop_3}), we get that the excess of the \{333A\}-component $\{\mf_1^+,\mf_0,\pf\}$ is bounded by $\varepsilon^{d-1/2+1/4}(\varepsilon^*)^{-2d}$, which meets the requirement of Proposition \ref{prop.comb_est_extra} (3).

\emph{Case 3.2.2}: assume $\min(|x_i-x_j|_\Tb,|v_i-v_j|)\geq \varepsilon^{1/(8d)}$ for $i\neq j\in\{1,7,10\}$, and
\begin{equation}\label{eq.newprop_3_2_cite_2}\inf_{|t|\leq|\log\varepsilon|^{C^*}}|x_1-x_7-t(v_1-v_7)|_\Tb\leq\varepsilon^{1/(8d)}.\end{equation} In this case we will cut $\{\mf_1^+,\mf_0,\pf\}$ into a \{33A\}-component $\{\mf_1^+,\mf_0\}$ (which has excess $\varepsilon^{d-1}(\varepsilon^*)^{-d}$ by Proposition \ref{prop.newprop_3} (\ref{it.newprop_2})) and a \{2\}-component $\{\pf\}$. 

Note that $z_{e_1}$ and $z_{e_7}$ are respectively equal to $z_{e_1'}$ and $z_{e_7'}$ where $e_1'$ and $e_7'$ are two free ends at elementary components $\Xc$ and $\Yc$ cut before $\{\mf_1^+,\mf_0,\pf\}$ (cf. the pairing between free and fixed ends in Definition \ref{def.cutting}). We may assume all the components cut before $\{\mf_1^+,\mf_0,\pf\}$ are \{3\}-components (with only one \{4\}-component), otherwise we must have another \{33A\}-component which is already acceptable as in \emph{Case 2} above; moreover we must have $\Xc\neq\Yc$ because no atom can be adjacent to both $\mf_1^+$ and $\mf_0$.

We may assume that $\Xc$ is cut before $\Yc$, and let the fixed end at $\Yc$ be $f_7$ which corresponds to a free end $f_7'$ in a component $\Zc$ cut before $\Yc$. Note that $f_7'\neq e_1'$ though it is possible that $\Zc=\Xc$. Let $|v_{e_1'}-v_{f_7'}|\sim \mu'$ with $\varepsilon^{1/(8d)}\lesssim\mu'\lesssim 1$ (with $\sim$ replaced by $\lesssim$ if $\mu'=\varepsilon^{1/(8d)}$), then by Proposition \ref{prop.newprop_3} (\ref{it.newprop_1}) we know that $\Yc$ has excess $\varepsilon^{1/(8d)}(\mu')^{-1}$ and either $\Xc$ or $\Zc$ has excess $\mu'$. Putting together, we get total excess $\varepsilon^{d-1+1/(8d)}(\varepsilon^*)^{-d}$, which meets the requirement of Proposition \ref{prop.comb_est_extra} (3).

\emph{Case 3.2.3}: assume $\min(|x_i-x_j|_\Tb,|v_i-v_j|)\leq \varepsilon^{1/(8d)}$ for $i\neq j\in\{1,7,10\}$. In this case, using this smallness condition and arguing in the same way as in \emph{Case 3.2.2} above, we can also get another component of excess $\varepsilon^{1/(8d)}$ and the same conclusion holds. This completes the proof.
\end{proof}
\begin{proposition}\label{prop.newcase_2} Let $P$ be the set of all particle lines of $\Mb$. Then Proposition \ref{prop.new2layer} is true, if $|P_1|\neq 3$, or $|P_2|\neq 3$, or $P_1\cup P_2\neq P$.
\end{proposition}
\begin{proof} By Proposition \ref{prop.newcase_1} we may assume each of $A_1$ and $A_2$ has a canonical cycle with bottom atom $\mf_1^+$ and $\mf_2^+$ respectively (in particular $|P_1|\geq 3$ and $|P_2|\geq 3$). We may assume the dimension $d=3$ (the case $d=2$ is much easier), so it would suffice to get excess $\varepsilon^{2+1/(10d)}$. We consider three cases.

\emph{Case 1}: assume $P_1\cup P_2\neq P$, say particle line $\boldsymbol{q}\in P\backslash (P_1\cup P_2)$. In this case we cut $A_1\cup A_2\cup\{\nf_0\}$ as free from $\Mb$, and then cut $A_1$ as free from $A_1\cup A_2\cup\{\nf_0\}$. Note that $A_1$ has no fixed end, $A_2\cup\{\nf_0\}$ has exactly one free end, and each of them has a canonical cycle. We then cut both sets using \textbf{UP}; since each set has no deg 2 atom, at most one deg 3 atom and a cycle, by the same proof of Proposition \ref{prop.alg_up} (2), for each of them we can get a \{33A\}-component $\Xc_j$ such that the two fixed ends at $\Xc_j$ match the two free ends at single-molecule (\{3\}- or \{4\}-) components $\Yc_j$ and $\Zc_j$ cut before $\Xc_j$ (note that it is allowed that $\Yc_j=\Zc_j$). Let $(x_1,v_1)$ and $(x_7,v_7)$ be the vectors associated with these two fixed ends, and assume $|v_1-v_7|\sim\mu$ for some $\varepsilon\sim\mu\lesssim 1$ (with $\sim$ replaced by $\lesssim$ if $\mu=\varepsilon$), then by Proposition \ref{prop.newprop_3} (\ref{it.newprop_1})--(\ref{it.newprop_2}) we know that $\Xc_j$ has excess $\varepsilon\cdot\mu^{-1}$ and either $\Yc_j$ or $\Zc_j$ has excess $\mu$. Putting together we already get excess $\varepsilon^2$.

Next, after cutting $A_1\cup A_2\cup\{\nf_0\}$ as free, we cut all the remaining deg 2 atoms in $\Mb$ as free, until $\Mb$ has no more deg 2 atoms. Note that in the whole process $\Mb_U$ has no top fixed end, and $\Mb_D$ has no bottom fixed end, and all the atoms that have been cut only belong to particle lines in $P_1\cup P_2$. We then apply \textbf{UP} to $\Mb_U$ and apply \textbf{DOWN} to $\Mb_D$ but with the twist that when the $\mf$ chosen in Definition \ref{def.alg_up} (\ref{it.alg_up_3}) has deg 3 and a child $\mf^-$ in $\Mb_D$ of deg 3, then we cut $\{\mf,\mf^-\}$ as a \{33A\}-atom. Since at this time $\Mb_D$ has no deg 2 atom, and $\Mb_D$ will have deg 2 atoms one all atoms in $\Mb_U$ have been cut (since $\Mb_D$ \emph{intersects the particle line $\boldsymbol{q}\in P\backslash (P_1\cup P_2)$}, and any highest atom in $\Mb_D$ will become deg 2), we conclude that this application of \textbf{UP} must produce another \{33A\}-atom, which has excess $\varepsilon^{1/(8d)}$. This leads to total excess $\varepsilon^{2+1/(8d)}$ that meets the requirement of Proposition \ref{prop.comb_est_extra} (3).

\medskip
Below, with \emph{Case 1} proved, we can assume $P_1\cup P_2=P$, and (say) $|P_1|\geq 4$. Recall $A_1$ contains a canonical cycle with bottom atom $\mf_1^+$ and some top atom $\rf$; now choose a canonical cycle $\Cc$ in $A_1$ such that its bottom atom $\pf_{\mathrm{bot}}$ is an ancestor of $\mf_1^+$, the top atom $\pf_{\mathrm{top}}$ is a descendant of $\rf$, and the value of $|t_{\pf_{\mathrm{top}}}-t_{\pf_{\mathrm{bot}}}|$ is the smallest. Under this assumption, it is easy to see that \emph{any atom that is both a descendant of an atom in $\Cc$ and an ancestor of an atom in $\Cc$ must itself belong to $\Cc$}. Now we consider two remaining cases.

\emph{Case 2}: assume $\Cc$ is not a triangle. In this case we cut $A_1\cup A_2\cup\{\nf_0\}$ as free from $\Mb$, then cut $A_1$ as free from $A_1\cup A_2\cup\{\nf_0\}$, then cut $\Cc$ as free from $A_1$. By the same proof in \emph{Case 1} above, we get total excess $\varepsilon$ from $A_2\cup \{\nf_0\}$. Moreover, by the properties of $\Cc$ stated above, we can divide $A_1\backslash \Cc$ into two sets $A_1'$ and $A_1''$ (for example by choosing $A_1''$ as the set of all descendants of atoms in $\Cc$) such that (i) no atom in $A_1''$ is parent of atom in $A_1'$ and (ii) $A_1'$ has no top fixed end, and $A_1''$ has no bottom fixed end or full component. We then deal with $A_1\backslash\Cc$ by cutting $A_1''$ using \textbf{DOWN} and then cutting $A_1'$ using \textbf{UP}, and deal with $\Mb\backslash (A_1\cup A_2\cup\{\nf_0\})$ as in \emph{Case 1}.

Finally, we consider $\Cc$. Choose the top atom $\pf=\pf_{\mathrm{top}}$ and a child $\qf$ of it, and consider the two cases (by inserting cutoff functions) when $|t_{\pf}-t_\qf|\geq \varepsilon^{1/(8d)}$ or when $|t_{\pf}-t_\qf|\leq \varepsilon^{1/(8d)}$. In the first case we start by cutting $\pf_{\mathrm{bot}}$ and cut $\{\pf,\qf\}$ as the final \{33A\}-component, so by Proposition \ref{prop.newprop_3} (\ref{it.newprop_2}) we get excess $\varepsilon^{3/2}$. in the second case we start by cutting $\pf$ and the cutting $\qf$ as a \{3\}-component, then proceed with the rest of $\Cc$ and get a \{33A\}-component. By the same proof as in \emph{Case 1} above, we get total excess $\varepsilon$ for the \{33A\}-component and one previous \{3\}-component, however, the restriction $|t_{\pf}-t_\qf|\leq \varepsilon^{1/(8d)}$ provides excess $\varepsilon^{1/(8d)}$ at $\qf$ which makes the total excess $\varepsilon^{1+1/(8d)}$. Combining with $A_2\cup \{\nf_0\}$ we get total excess $\varepsilon^{2+1/(8d)}$, which meets the requirement of Proposition \ref{prop.comb_est_extra} (3).

\emph{Case 3}: assume $\Cc$ is a triangle. Consider all the atoms that are parent or child of two atoms in $\Cc$, and subsequent atoms that are parent or child of two chosen atoms, and so on. As seen in the proof of Lemma \ref{lem.newlem}, these atoms form a set of triangles piled up against each other. Let this set of atoms (including $\Cc$) be $\Cc_1$; note that all these atoms only belong to \emph{three particle lines} (say in a set $P_3$). There are then a few sub-cases.

\emph{Case 3.1}: assume $A_1\backslash \Cc_1$ is a forest with each component having exactly one bond connected to $\Cc_1$. In this case, recall that $A_1$ contains a canonical cycle with bottom atom $\mf_1^+$, which means that $\mf_1^+$ has two paths reaching some atom $\pf_{\mathrm{top}}$ by iteratively taking parents; since each component of $A_1\backslash \Cc_1$ is a tree and has only one bond connected to $\Cc_1$, it is clear that we must have $\mf_1^+\in \Cc_1$.

\emph{Case 3.1.1}: assume $\nf_1^-\not\in\Cc_1$ (cf. Proposition \ref{prop.2cluster}). In this case we first cut $A_1\cup A_2\cup\{\mf_0,\nf_0\}$ as free from $\Mb$ and cut its complement using \textbf{UP} and $\textbf{DOWN}$ as before. Then we cut $\Cc_1$ as free; since $\Cc_1$ contains a triangle, by the same proof in \emph{Case 1} we get total excess $\varepsilon$. Then we cut $A_2\cup\{\mf_0\}$ as free; note that it has only one fixed end and a canonical cycle, so by the same proof in \emph{Case 1} we again get total excess $\varepsilon$. Finally, note that $\nf_0$ is connected to exactly one atom in $A_1\backslash \Cc_1$ which belongs to a unique tree $T$ in $A_1\backslash\Cc_1$, and $T\cup\{\nf_0\}$ will have two deg 3 atoms (and no deg 2 atoms). We then get another \{33A\}-component by Proposition \ref{prop.alg_up} (1) which has excess $\varepsilon^{1/(8d)}$. This makes total excess $\varepsilon^{2+1/(8d)}$ that meets the requirement of Proposition \ref{prop.comb_est_extra} (3).

\emph{Case 3.1.2}: assume $\nf_1^-\in\Cc_1$. In this case we first cut $\Cc_1$ as free, and then cut $A_2\cup\{\nf_0\}$ as free; by similar arguments as above, we can get excess $\varepsilon$ for components within $\Cc_1$ and the same within $A_2\cup\{\nf_0\}$. After this, we cut all the deg 2 atoms in $\Mb$ (which includes $\mf_0$) until no more deg 2 atom is left. Similar to the proof of \emph{Case 1}, note that in this process $A_1\backslash\Cc_1$ remains the forest with each tree having only one fixed end, that $\Mb_U\backslash A_1$ has no top fixed end, and $\Mb_D$ has no bottom fixed end. Moreover all the atoms that have been cut \emph{only belong to particle lines in $P_2\cup P_3$}. Since $P_1\neq P_3$, we can choose a particle line $\boldsymbol{q}\in P_1\backslash P_3$. 

Then, we cut all the trees forming $A_1\backslash\Cc_1$ into \{3\}-components, and apply \textbf{UP} to the rest of $\Mb_U$ and apply \textbf{DOWN} to $\Mb_D$, but with the twist described in \emph{Case 1} in the whole process. Since $\Mb_D$ will have deg 2 atoms once all atoms in $\Mb_U$ have been cut (since $\Mb_D$ \emph{intersects the particle line $\boldsymbol{q}\in P_1\backslash P_3$}, and any highest atom in $\Mb_D$ will become deg 2), the same proof then guarantees another \{33A\}-component which has excess $\varepsilon^{1/(8d)}$. This makes total excess $\varepsilon^{2+1/(8d)}$ that meets the requirement of Proposition \ref{prop.comb_est_extra} (3).

\emph{Case 3.2}: finally, assume $A_1\backslash \Cc_1$ is not as desctribed in \emph{Case 3.1}. In this case we first cut $A_1\cup A_2\cup\{\nf_0\}$ as free from $\Mb$ and treat its complement as before. Then we cut $A_1$ as free and cut $A_2\cup\{\nf_0\}$ as before to get excess $\varepsilon$, so now we only need to deal with $A_1$. We cut $\Cc_1$ as free and get excess $\varepsilon$ as above; for $A_1\backslash\Cc_1$, we can cut it into \{2\}-, \{3\}- and \{33A\}-components in the same way as in Lemma \ref{lem.newlem}. Since at least one component of $A_1\backslash\Cc_1$ has a cycle or two deg 3 atoms by assumption, this procedure also provides at least one more \{33A\}-component which has excess $\varepsilon^{1/(8d)}$. This makes total excess $\varepsilon^{2+1/(8d)}$ that meets the requirement of Proposition \ref{prop.comb_est_extra} (3). This finishes the proof in all cases.
\end{proof}
Finally, with Propositions \ref{prop.newcase_1} and \ref{prop.newcase_2}, we can now prove Proposition \ref{prop.new2layer}.
\begin{proof}[Proof of Proposition \ref{prop.new2layer}] By Propositions \ref{prop.newcase_1} and \ref{prop.newcase_2}, we may assume $A_1$ and $A_2$ each has a cycle, and $|P_1|=|P_2|=3$ and $P_1\cup P_2=P$ (so $|P|=6$). Now choose a lowest atom $\qf_0$ in $A_1$ (which does not have a child in $\Mb_U$); by the same procedure as in Proposition \ref{prop.2cluster} (with $\mf_0$ replaced by $\qf_0$ and the roles of $\Mb_U$ and $\Mb_D$ reversed), we can define $(P_1',P_2')$ and $(A_1',A_2')$ associated with $\qf$. Again by Propositions \ref{prop.newcase_1} and \ref{prop.newcase_2}, we may assume $A_1'$ and $A_2'$ each has a cycle, and $|P_1'|=|P_2'|=3$ and $P_1'\cup P_2'=P$. Let $\pb_1'$ and $\pb_2'$ be the two particle lines containing $\qf_0$, then by Proposition \ref{prop.2cluster} we know $\pb_1',\pb_2'\in P_1$, while $\pb_1'\in P_1'$ and $\pb_2'\in P_2'$.

From the above, we know that $\{|P_1\cap P_1'|,|P_1\cap P_2'|\}=\{1,2\}$. Note also that the number of bonds between $A_j$ and $A_k'$ is also equal to $|P_j\cap P_k'|$ (because each such bond corresponds to a unique particle line that belongs to $P_j\cap P_k'$), so we may assume there is exactly one bond between $A_1$ and $A_1'$, and exactly two bonds between $A_1$ and $A_2'$. Now we cut as free $A_1$, then $A_1'$, then $A_2'$ in this order. After these cuttings, we know that $\Mb_U$ has no top fixed end and no full component, and $\Mb_D$ has no bottom fixed end and no full component, so we may cut $\Mb_U$ using \textbf{UP} and $\Mb_D$ using \textbf{DOWN}.

It remains to cut $A_1$, $A_1'$ and $A_2'$ into elementary components to meet the requirement in Proposition \ref{prop.comb_est_extra} (3). Since $A_1$ has no fixed end and $A_1'$ has one fixed end, by the same proof as in Proposition \ref{prop.newcase_2} \emph{Case 1}, we get total excess $\varepsilon$ for each of them. Finally, note that $A_2'$ has two fixed ends, so it either has no deg 2 atom or has only one deg 2 atom. In the first case we get a \{33A\}-component with $A_2'$ by Proposition \ref{prop.alg_up} (1); in the second case, we also get a \{33A\}-component by applying Proposition \ref{prop.alg_up} (1) after cutting this deg 2 atom. In any case we get one more \{33A\}-atom which has excess $\varepsilon^{1/(8d)}$, leading to total excess $\varepsilon^{2+1/(8d)}$, which meets the requirement of Proposition \ref{prop.comb_est_extra} (3). This finishes the proof of Proposition \ref{prop.new2layer}.
\end{proof}
\section{Proof of Theorem \ref{th.main}}\label{sec.proof_main} In this section we present the proof of Theorem \ref{th.main}. Most parts of the proof are identical to the corresponding $\Rb^d$ case proof in \cite{DHM24}, but with (i) the elementary molecule estimates in Section 9.1 in \cite{DHM24} replaced by those in Section \ref{sec.elem_int}, (ii) the arguments in Section 13 in \cite{DHM24} replaced by those in Section \ref{sec.error}, and (iii) a few isolated adjustments in Sections 10--11 in \cite{DHM24} concerning double bonds and their generalizations (double ov-segments) involving O-atoms. Below we will provide an overview of the proof in \cite{DHM24}, with explanations at each place where nontrivial adjustment is needed. For convenience, \emph{all section and proposition etc. numbers below refer to those in \cite{DHM24}, except those italized which come from this paper}.

\medskip
\textbf{Part 1: Cluster forest expansions.} This corresponds to Sections 3--4. Here, in the definition of the modified (Definition 3.1), extended (Definition 3.5) and truncated dynamics (Definition 3.7), we obviously replace $|x_i^M(t)-x_j^M(t)|$ by $|x_i^M(t)-x_j^M(t)|_\Tb$ etc., and keep the rest the same. In this way, all the proofs in Sections 3--4 carry out to the periodic case, with the main results (such as Propositions 4.17 and 4.21) being unaffected.

Then, we can reduce \emph{Theorem \ref{th.main}} in this paper to Propositions 5.1--5.4 as in Sections 3--4; note also that Proposition 5.1 can be proved in the same way as in Section 5.3.

\medskip
\textbf{Part 2: Molecules.} This corresponds to Sections 6--8. Here all the main results (such as Proposition 7.8) remain unchanged, and the proof is essentialy identical (of course, with the necessary notational adjustments corresponding to the torus case). The only difference is that, in the proof of Proposition 7.8, the ``$=$" sign in equation (7.8) should be replaced by ``$\geq$" to account for the possibilities that two linear trajectories of particles can collide or overlap for more than one time. This however does not affect the proof of Proposition 7.8, and the rest of Sections 6--8 also stay the same.

In this way, we can then reduce Proposition 5.3 to Proposition 7.11 with the same proof in Sections 6.4 and 7.3.

\medskip
\textbf{Part 3: Integral estimates.} This corresponds to Section 9. Here the results and proof in Section 9.2 carry out in the same way without any change. The estimates for elementary molecules in Section 9.1 (together with the notions of good and normal molecules etc. in Section 9.3) are now replaced by those in \emph{Sections \ref{sec.elem_int}--\ref{sec.summary}} in this paper. Note that the estimates in Section 9.1 essentially form a subset of those in \emph{Section \ref{sec.elem_int}} in this paper, and the proof of Proposition 7.11 in Section 9.3 only require the use of this subset (of course with suitable adaptations, plus a few extra results concerning double bonds, such as \emph{Proposition \ref{prop.intmini} (\ref{it.intmini_4_extra})--(\ref{it.intmini_extra_2})} in this paper).

In this way, we can then reduce Proposition 7.11 to Proposition 9.7 with the same proof in Section 9.3.

\medskip
\textbf{Part 4: The algorithm.} This corresponds to Sections 10--11, which proves Proposition 9.7. Here the same proof in Sections 10--11 carry out without any change, except at places which involve double bonds and their generalizations with O-atoms. These (nontrivial) adjustments include:
\begin{enumerate}
\item The proof of Proposition 10.2 (1). In the C-atom only case, this is explained in the proof of \emph{Proposition \ref{prop.alg_up} (1)} in this paper. In the case involving O-atoms, we only need to replace the double bond by its generalization involving O-atoms, namely the double ov-segment, which is defined as two atoms connected by two ov-segments (i.e. paths formed by serial edges at O-atoms, see Definition 8.2) along two different particle lines. In this case the condition $(\bigstar)$ is replaced by
\begin{itemize}[$(\triangle)$]
\itemindent=14pt
\item $\pf$ and $\mf$ are \emph{connected by a double ov-segment}.
\end{itemize} This still leads to a good component by \emph{Definition \ref{def.good_normal} (\ref{it.good_1}) and (\ref{it.good_3})} in this paper, and the rest of the proof then goes in the same way as \emph{Proposition \ref{prop.alg_up} (1)} in this paper.

\item The proof of Proposition 10.2 (2). There is only one case that requires adjustment, namely when $e_1'$ and $e_2'$ are serial at some O-atom, where $(e_1',e_2')$ are the two free ends that correspond to the two fixed ends $(e_1,e_2)$ at a \{33A\}-component generated in the cutting process. But in this case, the two atoms in this \{33A\}-component must satisfy $(\triangle)$, which again leads to a good component by \emph{Definition \ref{def.good_normal} (\ref{it.good_1}) and (\ref{it.good_3})} in this paper, after cutting the \{33A\}-component into a \{3\}- and a \{2\}-component. The rest of the proof then goes in the same way as \emph{Proposition \ref{prop.alg_up} (2)} in this paper.
\item The proof of Proposition 10.9. Here, for each component $A$ of $\Mb_D$ that contains a cycle, we consider the first atom $\nf\in A$ that has deg 2 when it is cut, and the atom $\pf$ which is cut in a cutting operation that turns $\nf$ into deg 2. If $\pf$ and $\nf$ are not connected by a double ov-segment, then the same proof carries out and implies that $\pf$ must belong to a \{33A\}-component which can be treated as good; if $\pf$ and $\nf$ are connected by a double ov-segment as in $(\triangle)$, then it is easy to see that $\pf$ must belong to either a \{33A\}- or a good \{3\}- or \{4\}-component by \emph{Definition \ref{def.good_normal} (\ref{it.good_1}) and (\ref{it.good_3})} in this paper. The rest of the proof then goes in the same way.
\item The proof of Proposition 11.2. Here the only place that needs adjustment is when a cutting operation in Definition 11.1 (2)--(4) breaks two ov-segments $\sigma_\pb$ and $\sigma_{\pb'}$ (where $\sigma_\pb$ are maximal ov-segments containg atoms in both $\Mb_U$ and $\Mb_D$, and $\pb$ is the particle line containing $\sigma_\pb$), such that $\pb\sim\pb'$ in the sense that $\sigma_\pb$ and $\sigma_{\pb'}$ intersect at an atom $\nf\in\Mb_D$. In this case $\mf$ and $\nf$ must be connected by a double ov-segment as in $(\triangle)$ (where $\mf$ is the atom cut in this cutting operation), which implies that $\mf$ must belong to either a \{33A\}- or a good \{3\}- or \{4\}-component by \emph{Definition \ref{def.good_normal} (\ref{it.good_1}) and (\ref{it.good_3})} in this paper. The rest of the proof then goes in the same way.
\end{enumerate}

In this way, we can finish the proof of Proposition 9.7, and thus Proposition 5.3, as in Sections 10--11.

\medskip
\textbf{Part 5: Asymptotics for $f^\Ac$.} This corresponds to Section 12, which proves Proposition 5.2. Here the same proof in Section 12 carry out without any change, except at the very end of the proof of Proposition 12.1 where we apply the \textbf{DOWN} algorithm to obtain a lower bound on the number of \{33A\}- (and good) components; here again extra argument is needed in the case of a double bond or double ov-segment, but these are essentially the same as those discussed in \textbf{Part 4} above, so we omit the details.

In this way, we can finish the proof of Proposition 5.2 as in Section 12.

\medskip
\textbf{Part 6: The algorithm for $f_s^{\mathrm{err}}$.} This corresponds to Section 13, which proves Proposition 5.4. Here, he same arguments in Section 13.1--13.2 carry out, which allows to reduce Proposition 5.4. to the counterpart of Proposition 13.1, which is \emph{Proposition \ref{prop.comb_est_extra}} in this paper. The proof in Section 13.3 is then replaced by those in \emph{Section \ref{sec.error} (Sections \ref{sec.cutting_final}--\ref{sec.newcase})} in this paper; note that it is in this part that we are using the extra estimates in \emph{Section \ref{sec.elem_int}} in this paper that do not correspond to Section 9.1.

In this way we can finish the proof of Proposition 5.4 and thus \emph{Theorem \ref{th.main}} in this paper, under the assumption that all parameters (such as $\Lf$ and $C_j^*$ etc. in Definition 2.2) are viewed as constants independent of $\varepsilon$. For the quantitative version of this main result with the $\log\log$ upper bound, see \textbf{Part 7} below.

\medskip
\textbf{Part 7: the iterated log upper bound.} Finally, we explain the reason for the upper bound $(\log|\log\varepsilon|)^{1/2}$ in \emph{Theorem \ref{th.main}} in this paper. Recall that in \cite{DHM24}, the final time $t_{\mathrm{fin}}$, the norms $A$ (and $B_0$ and $B$) of the Boltzmann solution, and (if applicable) the collision rate $\alpha$ are treated as constants independent of $\varepsilon$. This implies that the number of layers $\Lf$ (see Definition 2.2 (1)) and all the quantities $C_j^*$ and $C^*$ (see Definition 2.2 (2)) are also treated as constants independent of $\varepsilon$. Now, in this paper, we need to analyze how these quantities are allowed to depend on $\varepsilon$, as follows:
\begin{enumerate}
\item The choice of $\Lf$. We choose the number of layers $\Lf=\kappa\cdot (\log|\log\varepsilon|)^{1/2}$ for sufficiently small constant $\kappa$ depending only on $(\beta,d)$, so the length of each layer is $\tau=t_{\mathrm{fin}}\cdot \Lf^{-1}$. As such, for each molecule $\Mb$ with atom set $\Mc$, the contribution of collisions that are not recollisions to the integral expression associated with $\Mb$ has size (recall that $C$ is absolute constant depending only on $(\beta,d)$)
\begin{equation}\label{eq.remark1}
(C\alpha)^{|\Mc|}\cdot (1+A)^{|\Mc|}\cdot\tau^{|\Mc|}\leq\big(C\cdot \max(1,\alpha)\cdot\max(1,A)\cdot\max(1,t_{\mathrm{fin}})\cdot \Lf^{-1}\big)^{|\Mc|}
\end{equation} (see the proof of Proposition 7.11 in Section 9.3), which is convergent thanks to our choice of $\Lf$ and \emph{(\ref{eq.loglog})} in this paper, provided that the implicit constant there is small enough depending on $\kappa$.

\item The choice of $A_\ell$ and $\Lambda_\ell$ (see Definition 2.2 (3)). Here we set
\begin{equation}\label{eq.remark2}A_\Lf=|\log\varepsilon|;\quad \Lambda_\ell=A_\ell^{10d},\quad A_{\ell-1}=\Lambda_\ell^{10d}.\end{equation} Clearly these parameters are bounded by
\begin{equation}\label{eq.remark2.5}A_j,\Lambda_j\leq(|\log\varepsilon|)^{(10d)^{2\Lf}}\leq|\log\varepsilon|^{C^*}\end{equation} using the notation for $C^*$ in \emph{Remark \ref{rem.parameter}} in this paper.
\item The choice of $C_j^*$ (see Definition 2.2 (2)). Basically we set $C_1^*$ to be a large absolute constant and each $C_j^*$ to be large enough depending on $C_{j-1}^*$; however some caution is needed here. Note that in the proof of Proposition 10.10 and the layer selection process (Definition 10.12 and Proposition 10.13), we have used the assumption that
\begin{equation}\label{eq.remark3}C_{10}^*\gg (C_7^*)^{\Lf},\quad C_6^*\gg (C_5^*)^{\Gamma\Lf},
\end{equation} where $\Gamma$ is an absolute constant depending only on $d$; see equations (10.9), (10.12), (10.18), (10.20), and the arguments in the proof of Propositions 10.10 and 10.13. Note that $\Lf$ is the total number of layers and $\Gamma\Lf$ is the total number of thin layers after the layer refinement process in Secion 10.3. For the other $j$, by examining the proof in \cite{DHM24}, we can set $C_{j+1}^*=(C_j^*)^C$ for some absolute constant $C$.

As such, with our choice of $\Lf$, it is easy to see that all the $C_j^*$ are bounded above by
\begin{equation}\label{eq.remark4}C_j^*\leq C^{C\Gamma\Lf^2}\leq C^*\end{equation} using the notation for $C^*$ in \emph{Remark \ref{rem.parameter}} in this paper.
\end{enumerate}

By examining the proofs in  \cite{DHM24}, we can verify that (\ref{eq.remark1}), (\ref{eq.remark2.5}) and (\ref{eq.remark4}) are the only estimates we need to carry out these proofs quantitatively, with adjustments in \textbf{Parts 1--6} above. This then allows to prove \emph{Theorem \ref{th.main}} (under the quantitative assumption \emph{(\ref{eq.loglog})}) in this paper, With the above choice of parameters.

\end{document}